\documentclass[a4paper, 11pt]{amsart}
\usepackage{dsfont}
\usepackage{amssymb}
\usepackage{amsfonts}
\usepackage{amsthm}
\usepackage{amsmath}
\usepackage{verbatim}
\usepackage{a4wide}
\usepackage{bm}
\usepackage{color}
\usepackage{appendix}
\usepackage{amssymb}
\usepackage{bbm}
\usepackage{fancybox}
\usepackage{fancyhdr}
\usepackage{graphicx}
\usepackage{tristan}
\usepackage{color}
\usepackage[cm]{fullpage}
\usepackage{subfigure}

\newtheorem{lemma}{Lemma}[section]

\theoremstyle{definition}
\newtheorem{definition}[lemma]{Definition}
\theoremstyle{definition}
\newtheorem{Rem}[lemma]{Remark}
\newtheorem{The}[lemma]{Theorem}
\newtheorem{remark}[lemma]{Remark}
\theoremstyle{definition}

\newtheorem{Hyp}[lemma]{Hypothesis}

{\catcode `\@=11 \global\let\AddToReset=\@addtoreset}
\AddToReset{equation}{section}

\newcommand{\veps}{{\varepsilon }}

\newcommand{\bv}{\mathbf v}

\newcommand{\del}{\partial}

\renewcommand{\vec}[1]{\geovec{#1}}

\newcommand{\hatv}{\widehat{\vec v}}
\newcommand{\shatv}{\widehat{v}}

\newcommand{\hatf}{\widehat{\vec f}}

\newcommand{\eps}{\varepsilon}

\newcommand{\dgp}{\partial_x^+}
\newcommand{\dgm}{\partial_x^-}
\newcommand{\dgpm}{\partial_x^\pm}

\newcommand{\Cetad}{C_{\underline{\eta}}}
\newcommand{\Cetau}{C_{\overline{\eta}}}
\newcommand{\Cfu}{C_{\overline{\vec f}}}

\setboolean{showtodo}{true}
\setboolean{shownotes}{true}
\setboolean{showchanges}{true}%true
\setboolean{usemathrsfs}{true}
\author{Jan Giesselmann} 
\address{Jan Giesselmann\newline
 Institute of Applied Analysis and Numerical Simulation\newline
University of Stuttgart\newline
Pfaffenwaldring 57\newline
D-70563 Stuttgart\newline
  Germany} 
\curraddr{}
\email{jan.giesselmann@mathematik.uni-stuttgart.de}
\author{Tristan Pryer}
\address{
   Tristan Pryer,\newline
     Department of Mathematics and Statistics,\newline
     University of Reading,\newline
     Whiteknights,\newline
     PO Box 220,\newline
     Reading RG6 6AX,\newline
     UK
}
\email{T.Pryer@reading.ac.uk}
\thanks{
We gratefully acknowledge that this work  was supported by the 
German Research Foundation (DFG) via SFB TRR 75 `Tropfendynamische Prozesse unter extremen Umgebungsbedingungen'. TP also gratefully acknowledges the support of the EPSRC grant EP/P000835/1 ``Adaptive Regularisation''.
}

\title[]
{A posteriori analysis for dynamic model adaptation in convection dominated problems}
\date{\today}
\begin{document}
\begin{abstract}
  In this work we present an a posteriori error indicator for
  approximation schemes of Runge--Kutta--discontinuous--Galerkin type
  arising in applications of compressible fluid flows.  The purpose of
  this indicator is not only for mesh adaptivity, we also make use of
  this to drive model adaptivity. This is where a perhaps costly
  \emph{complex} model and a cheaper \emph{simple} model are solved
  over different parts of the domain. The a posteriori bound we derive
  indicates the regions where the \emph{complex} model can be
  relatively well approximated with the cheaper one. One such example
  which we choose to highlight is that of the Navier--Stokes--Fourier
  equations approximated by Euler's equations.
\end{abstract}
\maketitle

\section{Introduction}

This paper is concerned with model adaptation in the context of
advection dominated, compressible fluid flows as they appear in, for
example, aerospace engineering.  Compressible fluid flows can be
described by different models having different levels of complexity.
One example are the compressible Euler equations which are the limit
of the Navier-Stokes-Fourier (NSF) equations when heat conduction and
viscosity go to zero.  Arguably the NSF system provides a more
accurate description of reality since viscous effects which are
neglected in Euler's equation play a dominant role in certain flow
regimes like thin regions near obstacles, for example, aerofoils
exhibiting Prandtls's boundary layers \cite{Nic73}.  However, viscous
effects are negligible in large parts of the computational domain
where convective effects dominate \cite{BCR89,CW01,DGQ14}. Thus, it is
desirable to avoid the effort of handling the viscous terms in these
parts of the domain, that is, to use the NSF system only where needed
and simpler models, \eg (linearised) Euler equations, on the rest of
the computational domain.
 
This insight has lead to the development, of a certain type, of
so-called heterogeneous domain decomposition methods in which on a
certain part of the computational domain the NSF equations are solved
numerically, whereas the (linearised) Euler equations are used for far
field computations \cite[e.g.]{USDM06, CW01,Xu00,BHR11}.  In those
works the domain was decomposed a priori, \ie before the start of the
numerical computation.  The accuracy and efficiency of those schemes
depends sensitively on the placement of the domains.  Thus, the user
is required to have some physical intuition on where to put the domain
boundary.

It has been suggested that applying a more ``adaptive'' approach using
cutoff functions to the dissipative terms in the NSF equations might
lead to a modified model only containing second derivatives when a
certain threshold is exceeded \cite{BCR89,AP93}.  The disadvantage
with this approach is that error control for this type of model
adaptation is not available, and the main justification for its use is
that the modified equations converge to the original ones when the
cut-off parameter tends to zero.  A more rigorous model adaptation
approach for stationary linear advection-diffusion systems based on
optimal control techniques can be found in \cite{AGQ06} {and
  goal-oriented modelling error estimates for viscous and inviscid
  Burgers' equations were presented in \cite{CBB05}.}  Recently a
heterogeneous domain decomposition technique for linear model problems
based on factorisation was suggested in \cite{GHM16}.

Our approach to this issue differs from the previous ones, except
\cite{AGQ06} and \cite{CBB05} in that we aim at having an {\it a posteriori} criterion
which enables an automatic and adaptive choice of domains.  Our
approach is inspired by \cite{BE03} where a model adaptive algorithm
based on an a posteriori modeling error estimator is presented.  In
\cite{BE03} such an estimator was developed for hierarchies of
elliptic models based on dual weighted residuals.  This approach
easily extends to parabolic problems and was extended to viscous,
incompressible flows in \cite{BE04,BBRR07,SRO11,OBPB15}.
{It should be noted that  \cite{BE04,SRO11,PV14,OBPB15} combine model adaptation with mesh adaptation.
Earlier approaches to model adaptation based on a posteriori estimators or indicators are \cite{ASS99,SO99,OV00, VO01}.}

The situation which we consider differs from the one studied in
\cite{BE03} since the problems at hand are highly nonlinear and
convection dominated in nature.  In particular, it is known that error
estimators based on dual weighted residuals have certain theoretical
limitations for nonlinear systems of hyperbolic conservation laws, \eg
Euler's equations, since the dual problems may become ill-posed in
case the solution to the primal problem is discontinuous \cite{HH02}.

%Thus, we consider a different approach for error estimation. 
Our estimates are based on non-linear energy-like arguments using the
relative entropy framework.
{A similar methodology was used in \cite{Fis15} where rigorous a posteriori error estimates for the assumption of incompressibility in viscous flows were presented.} 
It is similar to our approach in that it does not rely on the solution of dual problems but uses the relative entropy stability framework. Our presentation will focus on
Runge--Kutta--discontinuous--Galerkin (RKDG) discretisations which are
an established tool for viscous as well as inviscid compressible fluid
flows. The schemes we study use discretisations of the NSF equations
on some part of the computational domain and discretisations of
Euler's equations on the rest of the domain.  Following
\cite{GMP_15,DG_15}, which were concerned with a posteriori estimators
for discretisation errors for hyperbolic conservation laws, we define
(explicitly computable) reconstructions of the numerical solution and
use the relative entropy stability framework in order to derive an
explicitly computable bound for the difference between the
reconstruction and the exact solution to the NSF equations.  Our error
estimator, in fact, consists of two parts. One part is related to the
modelling error while the other one is related to the numerical
discretisation error such that our estimator allows for simultaneous
model and mesh adaptation.  Model adaptation for hyperbolic
conservation laws and relaxation systems was addressed recently in
\cite{MCGS12}. In this work an error estimator for scalar problems,
based on Kruzhkovs $\leb{1}$ stability theory, was presented.
Additional relevant work is \cite{CCGMS13} where an error indicator
for general systems based on Chapman-Enskog expansions was derived.

{While the main motivation of our work stems from model adaptation between Euler and NSF equations and between isothermal Euler and isothermal Navier-Stokes equations,
we present our analysis on a more abstract level, which we will describe in the next paragraph. 
We do this to allow the treatment of many cases simultaneously and, as such, it constitutes the central part of our analysis.}

In the abstract framework we consider general models
that include dissipative effects of the form:
\begin{equation}\label{cx} \del_t \vec u +
  \sum_\alpha \partial_{x_\alpha} \vec f_\alpha (\vec u) =
  \sum_\alpha \partial_{x_\alpha}( \eps \vec g_\alpha(\vec u , \nabla
  \vec u)) \text{ on } \rT^d \times (0,T) \end{equation}
where $\vec u$ is the (unknown) vector of conserved quantities,
$\eps>0$ is a small parameter,
$\vec f_\alpha \in \cont{2}(U,\rR^{ n}),$ $\alpha=1,\dots,d$ are the
components of an advective flux and
$\vec g_\alpha \in \cont{2}(U \times \rR^{d \times n},\rR^{ n}),$
$\alpha=1,\dots,d$ are the components of a diffusive flux.  Moreover,
the so-called state space $U \subset \rR^n$ is an open set; $T>0$ is
some finite time and $\rT^d$ denotes the $d$-dimensional flat torus.
% By presenting the arguments in this form we are able to simultaneously
% analyse the NSF system as well as various other examples.

%The reliable simulation of \eqref{cx} requires us to resolve localised structures, e.g. transition layers, in detail.
We are interested in the case of $\eps$ being very small, so that on
large parts of the computational domain the inviscid model
\begin{equation}\label{sim} \del_t \vec u + \sum_\alpha \partial_{x_\alpha} \vec f_\alpha (\vec u) = 0 \text{ on } \rT^d \times (0,T)\end{equation}
is a reasonable approximation of \eqref{cx}, for which, numerical
methods are computationally cheaper.

We will show later how the NSF equations fit into the framework
\eqref{cx} and that the Euler equations have the form \eqref{sim}.
Indeed, we will present our analysis first for a scalar model problem,
where we are able to derive fully computable a posteriori estimators.
After treating the scalar case we advance to pairs of models \eqref{sim} and
\eqref{cx}, which also exist in applications beyond compressible fluid flows,
for example traffic modelling.

We will assume throughout this work that \eqref{cx} and \eqref{sim}
are endowed with an (identical) convex entropy/entropy flux pair.
This assumption is true for the Euler and NSF model hierarchy, and
expresses that both models are compatible with the second law of
thermodynamics.  For the scalar model problem it is trivially
satisfied.

This entropy pair gives rise to a stability framework via the
so-called relative entropy, which we will exploit in this work.  The
relative entropy framework for the NSF equations has received a lot of
interest in recent years, \eg \cite{FJN12,LT13}.  It enables us to
 study modelling errors for inviscid approximations of
viscous, compressible flows in an a posteriori fashion.
In the case $n=1$ and general $d$, which
is the scalar case, we are able to prove rigorous a
posteriori estimates. All constants are explicitly computable and we
call such quantities \emph{estimators}. When $n>1$ a non-computable
constant appears in the argument which we were not able to
circumvent. The resultant a posteriori bounds are called
\emph{indicators} in the sequel.
Based on this a posteriori estimator/indicator we
construct a model adaptive algorithm which we have implemented for
model problems.

There are many applications, \eg aeroacoustics \cite{USDM06}, where it
would be desirable to consider the pair of models consisting of NSF
and the {\it linearised} Euler equations for computational efficiency.
We are currently unable to deal with this setting since our analysis
relies heavily on \eqref{cx} and \eqref{sim} having the same entropy
functional, while the linearised Euler system has a different entropy.

{A posteriori analysis is always based on the stability
  framework of the underlying problem.  For the class of problems
  considered here we make use of the relative entropy framework which
  requires that (at least) one of the solutions is Lipschitz
  continuous in space.  In principle, the model adaptation idea is
  independent of the numerical discretisations of both models, but for
  most approximation schemes for compressible fluid flows (finite
  volumes, discontinuous Galerkin) the discrete solution itself is not
  smooth enough to be used in the relative entropy stability analysis.
  Thus, our a posteriori argument requires a \emph{reconstruction}
  approach.  We will describe this reconstruction in detail for dG
  schemes.  Note that, in principle, our a posteriori analysis gives
  rise to a modelling error indicator in a wide variety of numerical
  schemes.  The bound relies heavily on a computable, Lipschitz
  continuous approximation of the exact solution, but it does not
  matter whether this results from a continuous finite element scheme
  of Taylor-Hood type, for example, or a reconstruction of a finite
  volume or dG solution.  }

The outline of this paper is as follows: In \S \ref{sec:mees} we
present the general framework of designing a posteriori estimators
based on the study of the abstract equations (\ref{cx})--(\ref{sim})
in the scalar case.  In \S \ref{sec:struct} we establish some structural properties of the INS and NSF systems regarding the compatibility of their diffusion terms with their relative entropies.
In 
\S \ref{sec:meesy} we describe how the relative entropy
framework can be extended to those generic pairs of systems of the form (\ref{cx})
and (\ref{sim}) that satisfy the structural properties that we established for the specific example of the NSF/Euler 
problem in \S \ref{sec:struct}.   In \S \ref{sec:recon} we
give an overview of the reconstruction approach presented in
\cite{GMP_15,DG_15}.  We also take the opportunity to extend the
operators to 2 spatial dimensions for use in numerical experiments
presented. Finally, we conclude the presentation with \S \ref{sec:num}
where we summarise extensive numerical experiments aimed at testing
the performance of the indicator as a basis for model adaptation.

%-------------------------------------------------------------------------------------------------------------------------------------------------------------------------------------------
\section{Estimating the modeling error for generic reconstructions - the scalar case}\label{sec:mees}
Our goal in this section is to show how the entropic structure can be used to obtain estimates for the difference between a (sufficiently regular) solution to \eqref{cx} and
 the solution $\vec v_h$ of a numerical scheme which is a discretisation of \eqref{sim} on  part of the (space-time) domain and 
of \eqref{cx} everywhere else.
We will pay particular attention to the {\it model adaptation error}, \ie the part of the error which is due to approximating not \eqref{cx} but \eqref{sim} on part of the domain.
In this section we present the arguments in the scalar case, as in this case all arguments can be given in a tangible way, see \S \ref{subs:sca}.

\subsection{Relative Entropy}
Before treating the scalar case we will review the classical concept of entropy/entropy flux pair which will be used in this Section and in \S \ref{sec:meesy}.
We will show how it can be employed in the relative entropy stability framework.
We will also show explicitly that the relative entropy in the NSF model satisfies the conditions which are necessary for our analysis.
Throughout this exposition we will use $d$ to denote the spatial dimension of the problem and $n$ as the number of equations in the system.

\begin{definition}[Entropy pair]
We call a tuple $(\eta,{\bf q}) \in \cont{1}(U,\rR) \times C^1(U,\rR^d)$ an entropy pair to \eqref{cx}
 provided the following compatibility conditions are satisfied:
\begin{equation}\label{cc1}
 \D \eta \D \vec f_\alpha = \D q_\alpha \quad \text{for } \alpha = 1,\dots, d \end{equation}
 and
 \begin{equation}\label{cc2}
  \partial_{x_\alpha} (\D \eta (\vec y)) : \vec g_\alpha(\vec y, \nabla \vec y) \geq 0 \quad \text{ for any } \vec y \in  \cont{1}(\rT^d,U) \text{ and } \alpha =1,\dots,d
 \end{equation}
where $\D$ denotes the Jacobian of functions defined on $U$ (the state space).
We call an entropy pair strictly convex if $\eta$ is strictly convex.
\end{definition}

\begin{remark}[Entropy equality]
 Note that every solution $\vec u \in \sobh{1}(\rT^d \times (0,T), U)$ of \eqref{cx} satisfies the additional companion balance law
 \[\partial_t \eta(\vec u) + \sum_\alpha \partial_{x_\alpha} \Big(  q_\alpha (\vec u) - \eps \vec g_\alpha(\vec u, \nabla \vec u) \D \eta(\vec u)\Big) =
 - \eps \sum_\alpha  \vec g_\alpha (\vec u, \nabla \vec u)\partial_{x_\alpha} \D \eta(\vec u).\]
 We refer to $\sum_\alpha  \vec g_\alpha (\vec u, \nabla \vec u)\partial_{x_\alpha} \D \eta(\vec u)$ as entropy dissipation.
\end{remark}

\begin{Rem}[Entropic structure]
We restrict our attention to systems \eqref{cx} which are endowed with at least one strictly convex entropy pair.
Such a convex entropy pair gives rise to a stability theory based on the relative entropy which we recall in Definition \ref{def:red}.
We will make the additional assumption that $\eta \in \cont{3}(U,\rR).$
While this last assumption is not standard in relative entropy estimates, it does not exclude any important cases.
\end{Rem}

\begin{remark}[Commutation property]
 Note that the existence of $q_\alpha$ means that for each $\alpha$ the vector field $\vec u \mapsto \D \eta(\vec u) \D \vec f_\alpha(\vec u)$ has  a potential and, thus,
 gives rise to the following commutation property:
 \begin{equation}\label{eq:dfdeta}
  \Transpose{(\D \vec f_\alpha)} \D^2 \eta = \D^2 \eta  \D \vec f_\alpha \quad \text{ for } \alpha=1,\dots, d.
 \end{equation}
\end{remark}

\begin{definition}[Relative entropy]\label{def:red}
 Let \eqref{cx} be endowed with an entropy pair $(\eta, \vec q).$ Then the relative entropy and relative entropy flux between the states
 $\vec u,\vec v  \in U$ are defined by
 \begin{equation}\label{red}
  \begin{split}
   \eta(\vec u|\vec v) &:= \eta(\vec u) - \eta(\vec v) - \D \eta(\vec v) (\vec u - \vec v)\\
   \vec q(\vec u|\vec v) &:= \vec q(\vec u) -  \vec q(\vec v) - \D \eta(\vec v) (\vec f(\vec u) - \vec f(\vec v)).
  \end{split}
 \end{equation}
\end{definition}

%its assumptions are satisfied in the special case of the isothermal Euler equations and the isothermal Navier-Stokes equations as the simple and the complex problem, respectively.

\begin{Hyp}[Existence of reconstruction]\label{hyp:recon}
 In the remainder of this section we will assume existence of a reconstruction $\hatv $ of a numerical solution $\vec v_h$  which weakly
solves
\begin{equation}\label{interm} \del_t \hatv + \div \vec f(\hatv) = \div \qp{ \widehat \eps \vec g(\hatv, \nabla \hatv)} + \cR_H + \cR_P\end{equation}
with explicitly computable residuals $\cR_H,\cR_P$ and $\widehat \eps: \rT^d \times [0,T] \rightarrow [0,\eps]$ being a function 
which is a consequence of the model adaptation procedure and determines in which part of the space-time domain which model is discretised.
We assume that the residual can be split into a part denoted $\cR_H \in \leb{2}(\rT^d \times (0,T), \rR^n)$ and a  part proportional 
to $\widehat \eps,$ denoted $\cR_P,$ which is an element of $ \leb{2}(0,T;\sobh{-1}(\rT^d, \rR^n)).$  

We also assume that an explicitly computable bound for $\norm{\hatv - \vec v_h}$ is available.
%Moreover, we assume existence of an explicitly computable reconstruction of the convective flux $\hatf$ which approximates $\vec f(\vec u).$
\end{Hyp}

\begin{Rem}[Model interpolation function]
 {Our theory is applicable for arbitrary non-negative functions $\widehat \eps \in L^\infty(\rT^d \times (0,T))$. 
In our numerical experiments, in \S \ref{sec:num}, we use a function $\widehat \eps$ only taking values in $\{0,\eps\}$ and which is piecewise constant within computational cells and time-steps
in order to avoid introducing cells in which we have an {\it intermediate model}. However, other choices are possible, and continuous choices for $\widehat \eps$ will be explored in a later work.}
\end{Rem}

\begin{Rem}[Reconstructions]
 We present some reconstructions in \S \ref{sec:recon} and point out that different choices of reconstructions will lead to different behaviours of the residuals
$\cR_H,\cR_P$ with respect to the mesh width $h.$ However, our main interest in the work at hand is a rigorous 
estimation of the modelling error, and not the derivation of optimal order discretisation error estimates,
which is a challenging task.
\end{Rem}

Throughout this paper we will make the following assumption on the exact solution.
\begin{Hyp}[Values in a compact set]\label{hyp:comp}
 We assume that we have {\it a priori} knowledge of a pair $(T,\mathfrak{O}),$
 where $T>0$ and $\mathfrak{O} \subset U$ is compact and convex, such that the exact solution $\vec u$ of \eqref{cx} takes values in 
 $\mathfrak{O}$ up to time $T,$ \ie 
\[\vec u({\vec x},t) \in \mathfrak{O} \quad \forall \ ({\vec x},t) \in \rT^d \times [0,T].\]
\end{Hyp}

%---------------------------------------------------------------------------------------------------------------------------------------------------
\subsection{A posteriori analysis of the scalar case}\label{subs:sca}
In the scalar setting, \ie $n=1,$ we restrict ourselves to the ``complex'' model being given by 
\begin{equation}\label{vb}
 \del_t u + \div \qp{\vec f(u)} = \div (\eps \nabla u) 
\end{equation}
and the simple model by 
\begin{equation}\label{ivb}
 \del_t u + \div \qp{\vec f(u)} = 0.
\end{equation}
Therefore, our model adaptive algorithm can be viewed as a numerical discretisation of
\begin{equation}\label{intb}
 \del_t u + \div \qp{\vec f(u)} = \div (\widehat \eps \nabla u) 
\end{equation}
with a space and time dependent function $\widehat \eps$ taking values in $[0,\eps].$
The spatial distribution of $\widehat \eps$ determines which model is solved on which part of the domain.
The reconstruction $\shatv$ of the numerical solution is a Lipschitz continuous weak solution to the perturbed problem
\begin{equation}\label{vb_R}
 \del_t \shatv + \div \qp{\vec f(\shatv)} = \div (\widehat \eps \nabla \shatv)  + \cR_H + \cR_P,
\end{equation}
where $\widehat \eps$ is a space dependent function with values in $[0,\eps],$ $\cR_H$ is the ``hyperbolic'' part of the discretisation residual
which is in $\leb{2}(\rT^d \times (0,T),\rR^n)$, and 
$\cR_P$ is the ``parabolic'' part of the discretisation residual.
Note that $\cR_P$ is not in $\leb{2}(\rT^d \times (0,T),\rR^n)$ but in $\leb{2}(0,T;\sobh{-1}(\rT^d,\rR^n))$ and that $\Norm{\cR_P}_{\leb{2}(\sobh{-1})}$ is proportional to $\widehat \eps.$
See Hypothesis \ref{hyp:recon} for our general assumption and \eqref{eq:dres} for such a splitting in case of a specific reconstruction.

In what follows we will use the relative entropy stability framework to derive a bound for the difference between the solution $u$
of \eqref{vb} and the reconstruction $\shatv$ of the numerical solution.
In the scalar case every strictly convex  $\eta \in \cont{2}(U,\rR)$ is an entropy for \eqref{vb}, because to each such $\eta$ a consistent entropy flux may be defined by 
\begin{equation}
 q_\alpha (u):= \int^u \eta'(v) f_\alpha'(v) \d v \text{ for } \alpha =1,\dots, d
\end{equation}
and the compatibility with the diffusive term boils down to
\begin{equation}
 - (\partial_{x_\alpha } y ) \eta''(y)  (\partial_{x_\alpha } y ) \leq 0 \quad \text{ for all } y \in \cont{1}(U,\rR) \text{ and } \alpha =1,\dots,d,
\end{equation}
which is satisfied as a consequence of the convexity of $\eta.$

We choose 
$\eta(u)=\tfrac{1}{2} u^2$
as this simplifies the subsequent calculations.
In particular,
\[ \eta(u|v)= \frac{1}{2} ( u- v)^2\]
for all $u, v\in U.$

\begin{remark}[Stability framework]
Note that in the scalar case we might also use Kruzhkov's $\leb{1}$ stability framework \cite{Kru70} instead of the relative entropy.
However, the exposition at hand is supposed to serve as a blueprint for what we are going to do for systems in \S \ref{sec:meesy}.
\end{remark}

\begin{Rem}[Bounds on flux]
 Due to the regularity of $\vec f$  and the compactness of $\mathfrak{O}$ there exists a
  constant $0 < \Cfu < \infty$  such
  that
  \begin{equation}
    \label{eq:consts1d}
\norm{ \vec f''( u) } 
    \leq \Cfu \quad \forall u \in \mathfrak{O} ,    
  \end{equation}
  where $\norm{\cdot}$ is the Euclidean norm for vectors.
  Note that $\Cfu$ can be explicitly computed from $\mathfrak{O}$ and $\vec f.$
\end{Rem}
The main result of this Section is then the following Theorem:
\begin{The}[A posteriori modelling error control]\label{the:1}
 Let 
 $u \in \sobh{1}(\rT^d \times (0,T),\rR)$ be a solution to \eqref{vb} and let $\shatv \in\sob{1}{\infty}(\rT^d\times (0,T),\rR)$ solve \eqref{vb_R}. 
 Then, for almost all $t \in (0,T)$ the following bound for the difference between $u$ and $\shatv$ holds:
 \begin{multline}\label{eq:the1}
  \Norm{ u(\cdot,t)  - \shatv (\cdot,t) }_{\leb{2}(\rT^d)}^2 + \int_0^t \eps \norm{u(\cdot,s) - \shatv (\cdot,s)}_{\sobh{1}(\rT^d)}^2 \d s 
  \\ \leq
  \Big( \Norm{ u(\cdot,0)  - \shatv (\cdot,0) }_{\leb{2}(\rT^d)}^2 + \cE_M + \cE_D 
   \Big) 
  \exp\big( (\Norm{ \nabla \shatv}_{\leb{\infty}(\rT^d\times (0,t))} \Cfu + 1) t\big),
  \end{multline}
  with
  \begin{equation}
   \begin{split}
    \cE_M :=& \Norm{\qp{\eps - \widehat \eps}^{1/2} \nabla \shatv}_{\leb{2}(\rT^d \times (0,t))}^2,\\
    \cE_D :=& \Norm{\cR_H}_{\leb{2}(\rT^d \times (0,t))}^2  +\frac{1}{\eps} \Norm{\cR_P}_{\leb{2}( 0,t; \sobh{-1}(\rT^d) )}^2.
   \end{split}
  \end{equation}

\end{The}

\begin{remark}[Dependence of the estimator on $\eps.$]
We expect that the modelling residual part of the estimator in \eqref{eq:the1}, $\cE_M$,
becomes small in large parts of the computational domain in case $\eps$ is sufficiently small, even if $\widehat \eps$ vanishes everywhere.
This means that \eqref{sim} is a reasonable approximation of \eqref{cx}.
It should be
noted that  $\Norm{\cR_P}_{\leb{2}( 0,t; \sobh{-1}(\rT^d) )} \sim \cO(\eps)$ (as $\widehat \eps$ only takes the values $0$ and $\eps$),
therefore we expect $\frac{1}{\eps} \Norm{\cR_P}_{\leb{2}( 0,t; \sobh{-1}(\rT^d) )}^2 \sim \cO(\eps).$ 
\end{remark}

\begin{proof}[Proof of Theorem \ref{the:1}]
 Testing \eqref{vb} and \eqref{vb_R} with $(u - \shatv)$ and subtracting both equations we obtain
 \begin{multline}\label{eq:re1}
 \int_{\rT^d}\frac{1}{2} \pdt ((u- \shatv)^2)  - \nabla \shatv \big(\vec f(u) - \vec f(\shatv) - \vec f'(\shatv)(u - \shatv) \big) + \eps | \nabla (u - \shatv)|^2 \\
 =\int_{\rT^d} (\widehat \eps - \eps) \nabla \shatv \cdot \nabla (u - \shatv) + \cR_H (u - \shatv) + \cR_P (u - \shatv),
 \end{multline}
where we have used that $\rT^d$ has no boundary and 
\begin{equation}
 \div (\vec q(u|\shatv)) - \nabla \shatv \big(\vec f(u) - \vec f(\shatv) - \vec f'(\shatv)(u - \shatv) \big) = \div( \vec f(u) - \vec f(\shatv)) (u - \shatv).
\end{equation}
Applying Young's inequality in \eqref{eq:re1} we obtain
 \begin{multline}\label{eq:re2}
 \d_t \left(\frac{1}{2}  \Norm{u- \shatv}_{\leb{2}(\rT^d)}^2\right)   + \eps \norm{ (u - \shatv)}_{\sobh{1}(\rT^d)}^2 \\
 \leq  \Norm{\qp{\eps - \widehat \eps}^{1/2} \nabla \shatv}_{\leb{2}(\rT^d)}^2  + \frac{\eps}{4} \norm{u - \shatv}_{\sobh{1}(\rT^d)}^2 + \Norm{\cR_H}_{\leb{2}(\rT^d)}^2 + \Norm{u - \shatv}_{\leb{2}(\rT^d)}^2\\
 +\frac{1}{\eps} \Norm{\cR_P}_{\sobh{-1}(\rT^d)}^2 + \frac{\eps}{4}  \norm{u - \shatv}_{\sobh{1}(\rT^d)}^2 +  \norm{ \shatv}_{\sob{1}{\infty}(\rT^d)} \Cfu \Norm{u - \shatv}_{\leb{2}(\rT^d)}^2 .
 \end{multline}
 Several terms in \eqref{eq:re2} cancel each other and  we obtain
  \begin{multline}\label{eq:re3}
  \d_t \left(\frac{1}{2}  \Norm{u- \shatv}_{\leb{2}(\rT^d)}^2\right)   +\frac{ \eps}{2} \norm{ (u - \shatv)}_{\sobh{1}(\rT^d)}^2\\
 \leq \Norm{\qp{\eps - \widehat \eps}^{1/2} \nabla \shatv}_{\leb{2}(\rT^d)}^2  + \Norm{\cR_H}_{\leb{2}(\rT^d)}^2 \\
 +\frac{1}{\eps} \Norm{\cR_P}_{\sobh{-1}(\rT^d)}^2 +  (\norm{ \shatv}_{\sob{1}{\infty}(\rT^d)} \Cfu + 1) \Norm{u - \shatv}_{\leb{2}(\rT^d)}^2 .
 \end{multline}
 Integrating \eqref{eq:re3} in time implies the assertion of the theorem.
\end{proof}

%----------------------------------------------------------------------------------------------------------------------------------------------------------------------------------------------------------------------------

%----------------------------------------------------------------------------------------------------------------------------------------------------------------------------------------------------------------------------
\section{Some structural properties of isothermal Navier-Stokes and Navier-Stokes-Fourier equations}\label{sec:struct}
As pointed out before we are mainly interested in model adaptation between isothermal Navier-Stokes (INS) and isothermal Euler and between Navier-Stokes-Fourier (NSF) and Euler.
Still, we will derive the a posteriori error estimates on an abstract level. These abstract arguments rely, in particular, on some structural properties which are satisfied by the systems mentioned above.
In addition to existence of a convex entropy pair we need certain compatibility properties between the relative entropy and diffusive fluxes  $\vec g$ as well as the parabolic part of the residual $\cR_P$.
The purpose of this section is to state those properties explicitly and to verify that they are satisfied for INS and NSF.

Let us begin by introducing the INS and NSF systems. Note that in this Section  the notation differs from that in the rest of the paper so that the physical quantities 
 under consideration are denoted as is standard in the fluid mechanics literature.
 The isothermal Navier-Stokes system (INS), where we replace the Navier-Stokes stress by $\nabla \bv$ for simplicity,
reads:
  \begin{equation}\label{isoNS}
  \begin{split}
   \partial_t \rho + \div(\rho \bv) &=0\\
   \partial_t (\rho \bv) + \div(\rho \bv  \otimes \bv) + \nabla (p(\rho)) &= \div (\mu \nabla \bv) 
   \end{split}
  \end{equation}
  where $\rho$ denotes density, $\bv$ denotes velocity and $p=p(\rho)$ is the pressure, given by a constitutive relation as a monotone function of density,
%   and 
%   \[ \sigma_{NS} = \mu ( \nabla \bv + \Transpose{(\nabla \bv)} ) + \lambda \div (\bv),\]
%   where $I$ denotes the unit matrix 
and $\mu\geq0$ is the viscosity parameter.

 In this case the simple model are the isothermal Euler equations which are obtained from \eqref{isoNS} by setting $\mu=0.$
  For these models the vector of conserved quantities is $\vec u=\Transpose{(\rho, \rho \bv)}$ and the mathematical entropy is given by
  \[ \eta(\rho, \rho \bv) = W(\rho) + \frac{|\rho \bv|^2 }{2\rho},\]
where Helmholtz energy $W$ and pressure $p$ are related by the Gibbs-Duhem relation
\[ p'(\rho)= \rho W''(\rho).\]\medskip

The Navier-Stokes-Fourier (NSF) equations for an ideal gas, where we again replace the Navier-Stokes stress by $\nabla \bv$, differ from the INS equations by also containing a conservation law for the energy:

 \begin{equation}\label{NSF}
  \begin{split}
   \partial_t \rho + \div(\rho \bv) &=0\\
   \partial_t (\rho \bv) + \div(\rho \bv  \otimes \bv) + \nabla (p(\rho,\epsilon)) &= \div (\mu \nabla \bv) \\
    \partial_t e + \div((e+ p)\bv )  &= \div (\mu(\nabla \bv) \cdot \bv + \kappa \nabla T),
   \end{split}
  \end{equation}
  where we understand $(\nabla \bv) \cdot \bv$ by $((\nabla \bv) \cdot \bv)_\alpha  =\sum_\beta v_\beta \partial_{x_\alpha } v_\beta.$
  The variables and parameters which appear in \eqref{NSF} and did not appear before are the temperature $T,$ the specific inner energy $\epsilon,$
 and the heat conductivity $\kappa >0$.
  In an ideal gas it holds
  \begin{equation}
    \label{eq:ideal}
    \begin{split}
   e&= \rho \epsilon + \frac{1}{2} \rho |\bv|^2,\\
   p&=(\gamma - 1) \rho \epsilon =  \rho R T,
   \end{split}
  \end{equation}
  where $R$ is the specific gas constant and $\gamma$ is the adiabatic coefficient.
  For air it holds $R=287$ and $\gamma=1.4.$
In this case the simple model are the Euler equations which are obtained from \eqref{NSF} by setting $\mu,  \kappa =0.$
The vector of conserved quantities is $\vec u=\Transpose{(\rho, \rho \bv,e)}$ and the (mathematical) entropy is given by 
\[ \eta(\vec u)= - \rho \ln \left( \frac{p}{\rho^\gamma}\right).\]
We will impose in both cases that the state space $\mathfrak{O}$ enforces positive densities.

Let us now state the compatibility conditions:
 \begin{definition}[Relative entropy compatible]\label{Hyp:new}
 We call a system \eqref{cx} and a subspace $V \subset \rR^n$ \emph{relative entropy compatible} provided there exists a computable
 function $\cD :  U \times \rR^{n \times d} \times U \times \rR^{n \times d} \rightarrow [0,\infty)$ 
and a computable constant $k>0$ such that for all $\vec w , \widetilde{ \vec w} \in \sob{1}{\infty}(\rT^d)$ taking values in $\mathfrak{O}$ and all $\cR \in \sobh{-1}(\rT^d)$, taking values in $V$, the following holds
  \begin{multline}\label{ca1}
\sum_\alpha (  \vec g_\alpha(\vec w, \nabla \vec w) - \vec g_\alpha(\widetilde{ \vec w}, \nabla \widetilde{ \vec w}) ): \partial_{x_\alpha} (\D \eta(\vec w) - \D \eta(\widetilde {\vec w}))\\ \geq
\frac{1}{k} \cD ( \vec w, \nabla \vec w, \widetilde {\vec w}, \nabla \widetilde {\vec w}) - k (\norm{\vec w}_{\sob{1}{\infty}}^2 + \norm{\widetilde {\vec w}}_{\sob{1}{\infty}}^2) \eta(\vec w| \widetilde {\vec w})
 \end{multline}
 and 
  \begin{multline}\label{ca2}
 \norm{ \nabla(\D \eta(\vec w) - \D \eta(\widetilde {\vec w}) ) \cdot \vec g(\widetilde {\vec w}, \nabla \widetilde {\vec w}) }\\
 \leq
 k^2 (\norm{\vec w}_{\sob{1}{\infty}}^2 + \norm{\widetilde {\vec w}}_{\sob{1}{\infty}}^2 +1)  \eta(\vec w| \widetilde {\vec w})+ \frac{1}{2k} \cD(\vec w, \nabla \vec w, \widetilde{\vec w}, \nabla \widetilde{\vec w}) 
 + k^2 \sum_\alpha \vec g_\alpha (\widetilde{\vec w}, \nabla \widetilde{\vec w}) \partial_{x_\alpha} \D \eta(\widetilde {\vec w})
 \end{multline}
 and 
  \begin{multline}\label{rp:comp}
 \int_{\rT^d}  \cR (\D \eta (\vec w) - \D \eta(\widetilde {\vec w}) ) \\
  \leq  k \Norm{\cR}_{\sobh{-1}(\rT^d)} \left((\norm{\vec w}_{\sob{1}{\infty}} 
  + \norm{\widetilde {\vec w}}_{\sob{1}{\infty}} +1) \Big( \int_{\rT^d} \eta( \vec w | \widetilde {\vec w} ) \Big)^{\frac{1}{2}}
  + \Big(\int_{\rT^d}\cD ( \vec w, \nabla \vec w, \widetilde {\vec w}, \nabla \widetilde {\vec w}) \Big)^{\frac{1}{2}}\right).
 \end{multline}
 \end{definition}
 
 The reason for introducing the space $V$ in \eqref{rp:comp} is that  for any reasonable discretisation of the isothermal Navier-Stokes equations there is no parabolic residual
 in the mass conservation equation. The same is true for Navier-Stokes-Fourier.
 
 We will now verify the relative entropy compatibility for the INS system. 
 
 \begin{lemma}\label{rem:vhins}
  The isothermal Navier-Stokes equations \eqref{isoNS} and the space $V =\{0\} \times \rR^d$ are relative entropy compatible.
  \end{lemma}
\begin{proof}
 For the isothermal Navier-Stokes equations \eqref{isoNS} the diffusive fluxes are given by
  \begin{equation}
    \label{bdd1}
    \vec g_\alpha = \Transpose{(0, \partial_{x_\alpha} \bv)}
  \end{equation}
  so that in the inequalities in Definition \ref{Hyp:new} the derivative $ \tfrac{\partial \eta}{\partial \rho }$ does not enter.  
  We  compute
  \begin{equation}
    \label{bdd2}
    \frac{\partial \eta}{\partial (\rho \bv) } = \frac{\rho \bv}{\rho} = \bv.
  \end{equation}
  Thus, in this case entropy dissipation is given by
  \[ \mu \sum_\alpha \vec g_\alpha (\vec w, \nabla \vec w) \partial_{x_\alpha} \D \eta(\vec w)= \mu \norm{ \nabla \bv}^2, \]
  where $|\cdot|$ denotes the Frobenius norm and $\vec w= \Transpose{( \rho, \rho \bv)} .$
  With $\widetilde{\vec w} = \Transpose{(\widetilde \rho,\widetilde \rho \widetilde\bv)}$ we define
 \begin{equation}
   \cD ( \vec w, \nabla \vec w, \widetilde {\vec w}, \nabla \widetilde {\vec w})\\
   := \norm{ \nabla \bv - \nabla \widetilde \bv}^2 .
 \end{equation}
Let us now verify that there exists a constant $k$ such that the inequalities from Definition \ref{Hyp:new} are valid.
Making use of the identities \eqref{bdd1} and \eqref{bdd2} we obtain
 \begin{equation}\label{ca1t1}
 (  \vec g(\vec w, \nabla \vec w) - \vec g(\widetilde{ \vec w}, \nabla \widetilde{ \vec w}) ): \nabla (\D \eta(\vec w) - \D \eta(\widetilde {\vec w}))\\
= \left( \nabla \bv - \nabla \widetilde \bv \right) : (\nabla \bv - \nabla \widetilde \bv ) = \cD ( \vec w, \nabla \vec w, \widetilde {\vec w}, \nabla \widetilde {\vec w}),
\end{equation}
so that \eqref{ca1} is satisfied for any $k \geq 1.$
Now, again using \eqref{bdd1} and \eqref{bdd2}, we find for any $k \geq 1$
\begin{multline}\label{ca2t1}
 \norm{ \nabla(\D \eta(\vec w) - \D \eta(\widetilde {\vec w}) ) \cdot \vec g(\widetilde {\vec w}, \nabla \widetilde {\vec w}) }
 = \norm{ \nabla (\bv - \widetilde \bv ) : \nabla \widetilde \bv}\\
 \leq  \frac{1}{2k} \norm{ \nabla (\bv - \widetilde \bv )}^2 + k \norm{ \nabla \widetilde \bv}^2
 =
 \frac{1}{2k} \cD(\vec w, \nabla \vec w, \widetilde{\vec w}, \nabla \widetilde{\vec w}) 
 + k \sum_\alpha \vec g_\alpha (\widetilde{\vec w}, \nabla \widetilde{\vec w}) \partial_{x_\alpha} \D \eta(\widetilde {\vec w}),
 \end{multline}
 \ie \eqref{ca2} is satisfied for any $k \geq 1.$
 %We have to make this argument somewhere else!!!
 %If a reasonable discretisation of the isothermal Navier-Stokes equations is used there is only a parabolic residual in the momentum balance equations, \ie
 
 For any $\cR \in \sobh{-1}(\rT^d)$, with values in $\{0\}\times \rR^d$, we have
 \begin{multline}
  \int_{\rT^d} \cR (\D \eta (\vec w) - \D \eta(\widetilde {\vec w}) ) \leq \Norm{\cR}_{\sobh{-1}} \Norm{\bv - \widetilde \bv }_{\sobh{1}}\\
  \leq \Norm{\cR}_{\sobh{-1}} \qp{k \Norm{\vec w - \widetilde{\vec w} }_{\leb{2}} + \Big( \int_{\rT^d} \cD(\vec w, \nabla \vec w, \widetilde{\vec w}, \nabla \widetilde{\vec w})\Big)^{\frac 1 2}} 
 \end{multline}
which shows that \eqref{rp:comp} is satisfied with a constant $k$ depending only on $\mathfrak{O}$.
\end{proof}

Next, we show relative entropy compatibility for the NSF system. This system contains two parameters $\mu,\kappa$ scaling different dissipative mechanisms.
 We will identify $\mu$ with the small parameter in \eqref{cx} and keep the ratio $\tfrac{\kappa}{\mu}$, which we will treat as a constant, of order $1.$

\begin{lemma}\label{rem:vhnsf}
  The  Navier-Stokes-Fourier equations for an ideal gas \eqref{NSF} and the space $V =\{0\} \times \rR^d \times \rR$ are relative entropy compatible.
  \end{lemma}
\begin{proof}
 For NSF the diffusive flux is $\vec g_\alpha = \Transpose{(0, \partial_{x_\alpha} \bv, \bv \cdot \partial_{x_\alpha} \bv + \tfrac{\kappa}{\mu} \partial_{x_\alpha} T )}.$
 Therefore, the derivative $ \tfrac{\partial \eta}{\partial \rho }$ does not enter the inequalities in Definition \ref{Hyp:new} and  it is sufficient to 
  compute
   \begin{equation}\label{NSF1} \frac{\partial \eta}{\partial (\rho \bv) } = \frac{\gamma-1}{R} \frac{\bv}{T}; \quad  \frac{\partial \eta}{\partial e} =- \frac{\gamma-1}{R} \frac{1}{T}\end{equation}
   for understanding relations in Definition \ref{Hyp:new}. 
   From equation \eqref{NSF1}  we may compute entropy dissipation:
   \begin{equation}
    \mu \sum_\alpha \vec g_\alpha (\vec u, \nabla \vec u) \partial_{x_\alpha} \D \eta(\vec u) 
    = \frac{\gamma-1}{R} \frac{\mu}{T} \norm{\nabla \bv}^2 + 
    \kappa \frac{\gamma-1}{R} \frac{|\nabla T|^2}{T^2}
   \end{equation}
and we define
 \begin{equation}
   \cD ( \vec w, \nabla \vec w, \widetilde {\vec w}, \nabla \widetilde {\vec w})\\
   := \frac{1}{\widetilde T} \norm{ \nabla \bv - \nabla \widetilde \bv}^2 + \frac{\kappa}{\mu} \frac{\norm{\nabla T - \nabla \widetilde T}^2}{\widetilde T^2} .
 \end{equation}
 Note that  $\cD ( \vec w, \nabla \vec w, \widetilde {\vec w}, \nabla \widetilde {\vec w})$ is not symmetric in $\vec w$ and $\widetilde{\vec w}.$
Let us now determine a constant $k$ such that the conditions from Definition \ref{Hyp:new} are valid.
By inserting the definitions of $\vec g$ and $\eta$ into the left hand side of \eqref{ca1} we obtain
 \begin{multline}\label{ca1t2}
 (  \vec g(\vec w, \nabla \vec w) - \vec g(\widetilde{ \vec w}, \nabla \widetilde{ \vec w}) ): \nabla (\D \eta(\vec w) - \D \eta(\widetilde {\vec w}))\\
= \frac{\gamma-1}{R}\Big[ \nabla\qp{ \frac{\bv}{T} - \frac{\widetilde \bv}{T}} : \nabla (\bv - \widetilde \bv)
- \nabla\qp{ \frac{1}{T} - \frac{1}{\widetilde T}} ( \nabla \bv \cdot \bv - \nabla \widetilde \bv \cdot \widetilde \bv)\\
- \frac{\kappa}{\mu} \nabla\qp{ \frac{1}{T} - \frac{1}{\widetilde T}} \nabla (T - \widetilde T)
\Big].
  \end{multline}
   After a lengthy but straightforward computation we arrive at
   \begin{multline}\label{ca1t3}
 \frac{R}{\gamma-1}(  \vec g(\vec w, \nabla \vec w) - \vec g(\widetilde{ \vec w}, \nabla \widetilde{ \vec w}) ): \nabla (\D \eta(\vec w) - \D \eta(\widetilde {\vec w}))\\
= \frac{1}{\widetilde T} \norm{\nabla (\bv - \widetilde \bv)}^2 - \frac{\nabla \widetilde T}{\widetilde T^2} (\nabla  \bv - \nabla \widetilde \bv) (\bv - \widetilde \bv) 
\\
+ \qp{ \frac{1}{T} - \frac{1}{\widetilde T}} \nabla \bv : \nabla (\bv - \widetilde \bv) 
- \qp{ \frac{1}{T^2} - \frac{1}{\widetilde T^2}}\nabla T \nabla \widetilde \bv (\bv - \widetilde \bv) 
\\
+ \frac{\nabla (T - \widetilde T)}{\widetilde T^2} \nabla \widetilde \bv (\bv - \widetilde \bv) 
+ \frac{\kappa}{\mu} \frac{\norm{\nabla (T - \widetilde T)}^2}{\widetilde T^2}
+  \frac{\kappa}{\mu}  \qp{ \frac{1}{T^2} - \frac{1}{\widetilde T^2}}\nabla T \nabla (T - \widetilde T).
  \end{multline}
  Note that the two summands of $\cD$ both appear on the right hand side of \eqref{ca1t3}.
  Applying Young's inequality to the other terms on the right hand side of \eqref{ca1t3} shows that \eqref{ca1}
  is true for some $k$ only depending on $\mathfrak{O}.$
  
   By inserting the particular forms of $\vec g$ and $\eta$ for NSF into the left hand side of \eqref{ca2} we obtain
    \begin{multline}\label{ca2t2}
 \frac{R}{\gamma-1}\norm{ \nabla(\D \eta(\vec w) - \D \eta(\widetilde {\vec w}) ) \cdot \vec g(\widetilde {\vec w}, \nabla \widetilde {\vec w}) }\\
 \leq
 \nabla\qp{ \frac{\bv}{T} - \frac{\widetilde \bv}{T}} \nabla \widetilde \bv - \nabla\qp{ \frac{1}{T} - \frac{1}{\widetilde T}}  \nabla \widetilde \bv \cdot\widetilde  \bv 
 - \frac{\kappa}{\mu}\nabla\qp{ \frac{1}{T} - \frac{1}{\widetilde T}} \nabla \widetilde T.
 \end{multline}
 We may rewrite this as 
     \begin{multline}\label{ca2t3}
 \frac{R}{\gamma-1}\norm{ \nabla(\D \eta(\vec w) - \D \eta(\widetilde {\vec w}) ) \cdot \vec g(\widetilde {\vec w}, \nabla \widetilde {\vec w}) }\\
 \leq
 \frac{\nabla \bv - \nabla \widetilde \bv}{T} : \nabla \widetilde \bv - \frac{\nabla T}{T^2} \nabla \widetilde \bv (\bv - \widetilde \bv) + \qp{ \frac{1}{T} - \frac{1}{\widetilde T}} \nabla \widetilde \bv : \nabla \widetilde \bv
 \\
 + \frac{\kappa}{\mu} \frac{1}{\widetilde T^2} \nabla (T - \widetilde T) \cdot \nabla \widetilde T +  \frac{\kappa}{\mu}\qp{ \frac{1}{T^2} - \frac{1}{\widetilde T^2}} \nabla T \cdot \nabla \widetilde T.
 \end{multline}
 Using Young's inequality we may infer from \eqref{ca2t3}  that \eqref{ca2}
  is true for some $k$ only depending on $\mathfrak{O}.$
  
Finally, for any $\cR$ taking values in $ \{0\} \times \rR^{d+1}$ we have
 \begin{multline}
   \frac{R}{\gamma-1} \int_{\rT^d} \cR (\D \eta (\vec w) - \D \eta(\widetilde {\vec w}) ) \leq \Norm{\cR}_{\sobh{-1}}\qp{ \Norm{\frac{\bv}{T} - \frac{\widetilde \bv}{\widetilde T} }_{\sobh{1}}   
  + \Norm{\frac{1}{T} - \frac{1}{\widetilde T} }_{\sobh{1}} }\\
  \leq \Norm{\cR}_{\sobh{-1}} \qp{C ( \norm{\vec w}_{\sob{1}{\infty}} +
  \norm{\widetilde {\vec w}}_{\sob{1}{\infty}} +1) \Norm{\vec w - \widetilde{\vec w} }_{\leb{2}} + \Big( \int_{\rT^d}\cD(\vec w, \nabla \vec w, \widetilde{\vec w}, \nabla \widetilde{\vec w})\Big)^{\frac 1 2} } 
 \end{multline}
which shows that \eqref{rp:comp} is satisfied with a constant $k$ depending only on $\mathfrak{O}$.
\end{proof}

%------------------------------------------------------------------------------------------------------------------------------------------------------

%------------------------------------------------------------------------------------------------------------------------------------------------------

\section{Estimating the modeling error for generic reconstructions - the systems case}\label{sec:meesy}
In this section we set up a more abstract framework allowing for the analysis of systems of equations.
This framework relies of the relative entropy compatibility, Definition \ref{Hyp:new}.
 For clarity of exposition we do not explicitly track the constants, rather denote a generic constant $C$ which may depend on $\mathfrak{O},\vec f, \vec g, \eta$ but is independent of mesh size, solution
 {and $\veps$}.
 { Robustness of the estimates in $\veps$ is of particular importance since we are particularly interested in the case of small $\veps$.}

As in \S \ref{sec:mees} we assume Hypothesis \ref{hyp:comp} which gives rise to the following bounds:

\begin{Rem}[Bounds on flux and entropy]
 Due to the regularity of $\vec f$ and $\eta,$ and the compactness of $\mathfrak{O}$ there are
  constants $0 < \Cfu < \infty$ and $0< \Cetad < \Cetau < \infty$ such
  that
  \begin{equation}
    \label{eq:consts}
 \norm{\Transpose{\vec v} \D^2 \vec f(\vec u) \vec v} 
    \leq \Cfu \norm{\vec v}^2,    
    \qquad 
 \Cetad 
    \norm{\vec v}^2
    \leq 
    \Transpose{\vec v} 
    \D^2 \eta(\vec u)
    \vec v
    \leq \Cetau \norm{\vec v}^2
  \Foreach \vec v \in \reals^n, \vec u \in \mathfrak{O},
  \end{equation}
  and 
  \begin{equation}\label{eq:constt}
   \norm{ \D^3 \eta (\vec u)} \leq \Cetau \Foreach  \vec u \in \mathfrak{O},
  \end{equation}
  where $\norm{\cdot}$ is the Euclidean norm for vectors in \eqref{eq:consts} and a Frobenius norm for 3-tensors in \eqref{eq:constt}.
  Note that $\Cfu$, $\Cetad$ and $\Cetau$ can be explicitly computed from $\mathfrak{O}$, $\vec f$ and $\eta.$
\end{Rem}

Now we are in position to state the main result of this section
\begin{The}[A posteriori modelling error control]\label{the:2}
 Let 
 $\vec u \in \sob{1}{\infty}(\rT^d \times (0,T),\rR^n)$ be a weak solution to \eqref{cx}.
 Let $\hatv \in\sob{1}{\infty}(\rT^d\times (0,T),\rR^n)$ weakly solve \eqref{interm}, with $\cR_P$ taking values in some subspace $V \subset \rR^n$.
 Let \eqref{cx} be
 endowed with a strictly convex entropy pair, and let \eqref{cx} and $V$ be relative entropy compatible.
Let
  $\vec u$ and $\hatv$ only take values in $\mathfrak{O}.$
  Then, the following a posteriori error estimate holds:
 \begin{multline}\label{eq:the2}
  \Norm{ \vec u(\cdot,t)  - \hatv (\cdot,t) }_{\leb{2}(\rT^d)}^2 + \int_{\rT^d \times (0,t)} \frac{\eps}{4k}  \cD (\vec u,\nabla \vec u, \hatv , \nabla \hatv )  
 \\  
   \leq  C \left( \Norm{ \vec u(\cdot,0)  - \hatv (\cdot,0) }_{\leb{2}(\rT^d)}^2 + \cE_D + \cE_M \right)\\
  \times \exp\left( C ( \Norm{\vec u}_{\sob{1}{\infty}}^2 + \Norm{ \hatv}_{\sob{1}{\infty}}^2 +1) t \right)
  \end{multline}
  with $C,k$ being constants depending on $(\mathfrak{O}, \vec f, \vec g, \eta)$ and
\begin{equation}\label{def:ce}
\begin{split}
 \cE_M &:= \Norm{\widehat \eps \vec g(\hatv, \nabla \hatv)}_{\leb{2}(\rT^d \times (0,t))}^2  +\int_{\rT^d \times (0,t)} (\eps - \widehat \eps) k^2 \sum_\alpha \vec g_\alpha (\hatv, \nabla \hatv) \partial_{x_\alpha} \D \eta(\hatv) ,\\
 \cE_D &:= \frac{k^2}{\eps}  \Norm{\cR_P}_{\leb{2}(0,t;\sobh{-1}(\rT^d))}^2 
   + \Norm{\cR_H}_{\leb{2}(\rT^d \times (0,t))}^2 .
 \end{split}
\end{equation}

\end{The}

\begin{remark}[Energy dissipation degeneracy]
 It should be noted that the estimate in Theorem \ref{the:2} contains certain assumptions, \ie $\vec u\in\sob{1}{\infty}$, which are not verifiable in an a posteriori fashion. 
 In particular, we assume more regularity than can be expected for solutions of systems of the form \eqref{cx}; see \cite{FNS11,FJN12} for existence results for systems of this type.
 However, the weak solutions defined in those references are not unique and only weak-strong uniqueness results are available, cf. \cite{FJN12}.
 Thus, convergent a posteriori error estimators can only be expected in case the problem \eqref{cx} has a more regular solution than 
 is guaranteed analytically. Note that the corresponding term, \ie $\Norm{ \vec u}_{\sob{1}{\infty}}$
 does not appear in the scalar case and is a consequence of the dissipation only being present in momentum and energy conservation laws but not in the mass conservation equation.
 This leads to a form of degeneracy of the energy dissipation governing the underlying system.
 
 { In addition, the presence of $\Norm{\vec u}_{\sob{1}{\infty}}$ on the right hand side of \eqref{eq:the2} makes it impossible to compute this quantity a posteriori without knowledge of the exact solution.
 Practically, we expect  $\Norm{\vec u}_{\sob{1}{\infty}}$ to be smaller than $\Norm{\hatv}_{\sob{1}{\infty}}$ since $\hatv$ results from a system with less dissipation, but to the best of our knowledge there is no guarantee
 that this is the case.
 For practical computations we believe that $ \cE_D, \cE_M $ are good indicators for discretisation and modelling errors, respectively.
 %% Vraiment???, and $ \cE_D + \cE_M $ is the quantity which we use for model adaptation in our numerical experiments.
 }
\end{remark}

\begin{remark}[Structure of the estimator]
Note that the first factor in the {right hand side of \eqref{eq:the2}} consists of three parts. The first part is the error in the discretisation and reconstruction of the initial data.
The second part $\cE_D$ is due to the residuals caused by the discretisation error. The third part $\cE_M$ consists of residuals caused by the model approximation error.
\end{remark}

\begin{remark}[Structure of modelling error residual]
Recall that we are interested in the case of $\eps$ being (very) small. In this case the term 
\[ \int_{\rT^d \times (0,t)} (\eps - \widehat \eps) k^2 \sum_\alpha \vec g_\alpha (\hatv, \nabla \hatv) \partial_{x_\alpha} \D \eta(\hatv) \]
is the dominating part in the modelling error residual $\cE_M$ and, thus, letting $\widehat \eps$ be $\eps$ in larger parts of the domain will usually reduce the
modelling error residual.
\end{remark}

\begin{proof}[Proof of Theorem \ref{the:2}]
 We test \eqref{cx} by $\D \eta(\vec u) - \D \eta (\hatv)$ and \eqref{interm} by $\D^2 \eta(\vec u) (\vec u - \hatv)$ and subtract both equations.
 By rearranging terms and using \eqref{eq:dfdeta} we obtain
 \begin{equation}\label{remd1}
  \int_{\rT^d} \partial_t \eta(\vec u| \hatv) + \sum_\alpha \partial_{x_\alpha} \vec q_\alpha (\vec u| \hatv) + \eps( \vec g(\vec u, \nabla \vec u) - \vec g(\hatv,\nabla \hatv) ): \nabla (\D \eta(\vec u) - \D \eta(\hatv)) 
   = E_1 + E_2 + E_3,
    \end{equation}
   with
   \begin{equation}\label{remd2}
    \begin{split}
        E_1 &:= \int_{\rT^d}\widehat \eps \vec g(\hatv, \nabla \hatv) : \nabla (\D^2 \eta(\hatv) (\vec u - \hatv))
        -\eps \vec g(\hatv, \nabla \hatv) : \nabla (\D \eta(\vec u) - \D \eta(\hatv)) ,
   \\
  E_2 &:=-\int_{\rT^d} \sum_\alpha \partial_{x_\alpha}\hatv  \D^2 \eta(\hatv) (\vec f_\alpha(\vec u) - \vec f_\alpha (\hatv) - \D \vec f_\alpha (\hatv) (\vec u - \hatv)),\\
  E_3 &:= \int_{\rT^d} (\cR_H + \cR_P) \D^2 \eta(\hatv) (\vec u - \hatv).
    \end{split}
 \end{equation}
As $\rT^d$ does not have a boundary and because of \eqref{ca1} we obtain
 \begin{equation}\label{remd3}
  \int_{\rT^d} \partial_t \eta(\vec u| \hatv)  + \frac{\eps}{k} \cD (\vec u, \nabla \vec u,\hatv,\nabla \hatv) 
   \leq E_1 + E_2 + E_3 + \eps k (\norm{\vec u}_{\sob{1}{\infty}}^2 + \norm{\hatv}_{\sob{1}{\infty}}^2)\int_{\rT^d} \eta(\vec u| \hatv).
    \end{equation}
We are now going to derive estimates for the terms $E_1,E_2,E_3$ on the right hand side of \eqref{remd3}.
We may rewrite $E_1$ as
\begin{equation}\label{e1}
 E_1 = E_{11} + E_{12}
\end{equation}
with
\begin{multline}\label{e11}
 \abs{E_{11}} :=\norm{ \int_{\rT^d}\widehat \eps \vec g(\hatv, \nabla \hatv): \nabla (\D^2 \eta (\hatv) (\vec u - \hatv ) 
 - \D \eta(\vec u) + \D \eta (\hatv) ) }\\
 \leq \Norm{\widehat \eps \vec g(\hatv, \nabla \hatv) }_{\leb{2}}^ 2
  + \Cetau ( \Norm{\vec u}_{\sob{1}{\infty}}^2 + \Norm{\hatv}_{\sob{1}{\infty}}^2) \Norm{\vec u - \hatv}_{\leb{2}}^2,
\end{multline}
where we used \eqref{eq:constt} and 
\begin{equation}\label{e12a}
 E_{12} := 
  -\int_{\rT^d} (\eps - \widehat \eps) \vec g(\hatv , \nabla \hatv)  : \nabla ( \D \eta(\vec u) - \D \eta (\hatv)).
  \end{equation}
  Using \eqref{ca2} we find 
\begin{multline}\label{e12} 
| E_{12}|  \leq 
  \eps k^2 (\norm{\vec u}_{\sob{1}{\infty}}^2 + \norm{\hatv}_{\sob{1}{\infty}}^2 +1)  \eta(\vec u| \hatv)\\
  + \frac{\eps}{2k} \cD(\vec u, \nabla \vec u, \hatv, \nabla \hatv) 
 + (\eps - \widehat \eps) k^2 \sum_\alpha \vec g_\alpha (\hatv,\nabla \hatv) \partial_{x_\alpha} \D \eta(\hatv).
\end{multline}
Concerning $E_2$ we note
\begin{equation}\label{e2}
|E_2| \leq \Cetau \Cfu \norm{\hatv}_{\sob{1}{\infty}} \Norm{\vec u - \hatv}_{\leb{2}}^2.
\end{equation}
We decompose $E_3$ into two terms
\begin{equation}\label{e3}
 E_3 = \int_{\rT^d} \cR_H  \D^2 \eta(\hatv) (\vec u - \hatv) + \cR_P \D^2 \eta(\hatv) (\vec u - \hatv) =: E_{31} + E_{32}.
\end{equation}
We have
\begin{equation}\label{e31}
 |E_{31} |\leq \Norm{\cR_H}_{\leb{2}}^2+ \Cetau^2 \Norm{\vec u - \hatv}_{\leb{2}}^2.
\end{equation}
We rewrite $E_{32}$ as
\begin{equation}
 E_{32} = \int_{\rT^d}  \cR_P \left( - \D \eta (\vec u ) + \D \eta (\hatv) +  \D^2 \eta(\hatv) (\vec u - \hatv) \right) + \cR_P \left(  \D \eta (\vec u ) - \D \eta (\hatv)\right)
\end{equation}
such that we get the following estimate
\begin{multline}\label{e32b}
 |E_{32} |\leq \Norm{\cR_P}_{\sobh{-1}(\rT^d)} \Norm{\D \eta (\vec u ) - \D \eta (\hatv) -  \D^2 \eta(\hatv) (\vec u - \hatv) }_{\sobh{1}(\rT^d)} \\
 + k \Norm{\cR_P}_{\sobh{-1}(\rT^d)} \left((\norm{\vec u}_{\sob{1}{\infty}} + \norm{\hatv}_{\sob{1}{\infty}} +1) \Norm{ \vec u- \hatv }_{\leb{2}(\rT^d)} 
  + \Big( \int_{\rT^d}\cD ( \vec u, \nabla \vec u, \hatv, \nabla \hatv) \Big)^{\frac{1}{2}}\right)
\end{multline}
due to \eqref{rp:comp}. We have 
\begin{equation}\label{e32c}
 \Norm{\D \eta (\vec u ) - \D \eta (\hatv) -  \D^2 \eta(\hatv) (\vec u - \hatv) }_{\sobh{1}} \leq \Cetau ( \Norm{\vec u}_{\sob{1}{\infty}} + \Norm{\hatv}_{\sob{1}{\infty}}) \Norm{\vec u- \hatv}_{\leb{2}},
\end{equation}
such that \eqref{e32b} becomes
\begin{equation}\label{e32d}
 | E_{32}| \leq \frac{k^2}{\eps}  \Norm{\cR_P}_{\sobh{-1}}^2 +  2 \eps \Cetau^2 ( \Norm{\vec u}_{\sob{1}{\infty}}^2 + \Norm{\hatv}_{\sob{1}{\infty}}^2 + 1) \Norm{\vec u- \hatv}_{\leb{2}}^2
  + \frac{\eps}{4k} \int_{\rT^d}\cD ( \vec u, \nabla \vec u, \hatv, \nabla \hatv).
\end{equation}
Combining \eqref{e31} and \eqref{e32d} we obtain 
\begin{multline}\label{e3f}
 | E_3 | \leq \frac{k^2}{\eps}  \Norm{\cR_P}_{\sobh{-1}}^2 +  2 \Cetau^2 (\eps \Norm{\vec u}_{\sob{1}{\infty}}^2 + \eps \Norm{\hatv}_{\sob{1}{\infty}}^2 + 2) \Norm{\vec u- \hatv}_{\leb{2}}^2\\
 + \Norm{\cR_H}_{\leb{2}}^2 + \frac{\eps}{4k} \int_{\rT^d}\cD ( \vec u, \nabla \vec u, \hatv, \nabla \hatv).
\end{multline}
Upon inserting \eqref{e11}, \eqref{e12}, \eqref{e2} and \eqref{e3f} into \eqref{remd3} we obtain for $\eps <1$
 \begin{multline}\label{remd4}
  \int_{\rT^d} \partial_t \eta(\vec u| \hatv)  + \frac{\eps}{4k} \cD (\vec u, \nabla \vec u,\hatv,\nabla \hatv) \\
   \leq 
\Norm{\widehat \eps \vec g(\hatv, \nabla \hatv) }_{\leb{2}}^ 2
  + C (\norm{\vec u}_{\sob{1}{\infty}}^2 + \norm{\hatv}_{\sob{1}{\infty}}^2 +1) \int_{\rT^d}  \eta(\vec u| \hatv)\\
 + (\eps - \widehat \eps) k^2 \sum_\alpha \vec g_\alpha (\hatv, \nabla \hatv) \partial_{x_\alpha} \D \eta(\hatv)
 + \frac{k^2}{\eps}  \Norm{\cR_P}_{\sobh{-1}}^2
 + \Norm{\cR_H}_{\leb{2}}^2 
    \end{multline}
    where we have used that $\Norm{\vec u - \hatv}_{\leb{2}}^2$ is bounded  in terms of the relative entropy.
Using Gronwall's Lemma we obtain
\begin{multline}
  \Norm{ \vec u(\cdot,t)  - \hatv (\cdot,t) }_{\leb{2}(\rT^d)}^2 +
 \frac{\eps}{4k} \int_{\rT^d \times (0,t)} \cD (\vec u,\nabla \vec u, \hatv, \nabla \hatv) \d s \\
  \leq  \Cetau \left( \Norm{ \vec u(\cdot,0)  - \hatv (\cdot,0) }_{\leb{2}(\rT^d)}^2 +  \cE_D + \cE_M\right)\\
  \times \exp\left( C(\norm{\vec u}_{\sob{1}{\infty}}^2 +  \norm{\hatv}_{\sob{1}{\infty}}^2 +1) t \right)
\end{multline}
with $\cE_D, \cE_M$ defined in \eqref{def:ce}.
\end{proof}

%-------------------------------------------------------------------------------------------------------------------------------------------------------------------

%-------------------------------------------------------------------------------------------------------------------------------------------------------------------

\section{Reconstructions}\label{sec:recon}
In \S \ref{sec:mees} and \S \ref{sec:meesy} we have assumed existence of reconstructions of numerical solutions whose residuals are computable, see Hypothesis \ref{hyp:recon}. We have also assumed a certain regularity of these reconstructions. In this Section we will describe one way to obtain such reconstructions for semi-(spatially)-discrete dG schemes.

In previous works reconstructions for dG schemes have been mainly used for deriving a posteriori bounds of {\it discretisation errors} \cite[c.f.]{DG_15,GMP_15,GeorgoulisHallMakridakis:2014} for hyperbolic problems.
In these works the main idea is to compare the numerical solution $\vec v_h$ and the exact solution $\vec u$ not directly,
but to introduce an intermediate quantity, the reconstruction $\hatv$ of the numerical solution.
This reconstruction must have two crucial properties:
\begin{itemize}
\item Explicit a posteriori bounds for the difference $\Norm{\hatv -\vec v_h}_\cX$ for some appropriate $\cX$ need to be 
  available and,
\item The reconstruction $\hatv$ needs to be globally smooth enough to apply the appropriate stability theory of the underlying PDE.
\end{itemize}
These two properties allow the derivation of an a posteriori bound for the difference $\Norm{\vec u - \vec v_h}_\cX.$

In the sequel we will provide a methodology for the explicit computation of $\hatv$ \emph{only} from the numerical solution $\vec v_h$. This means trivially that the difference $\Norm{\hatv - \vec v_h}_\cX$ can be controlled explicitly.

From \S \ref{sec:mees} and \S \ref{sec:meesy} the stability theory we advocate is that of relative entropy and we have extended the classical approach 
such that not only \emph{discretisation} but also \emph{modelling errors} are accounted for.
Note also that for our results from \S \ref{sec:mees} and \S \ref{sec:meesy} to be applicable we require $\hatv \in \sob{1}{\infty}(\rT^d \times [0,T]).$

In this Section we describe how to obtain reconstructions $\hatv$ of numerical solutions $\vec v_h$ 
which are obtained by solving \eqref{cx} on part of the space-time domain and \eqref{sim} on the rest of the space-time domain.
For brevity we will focus on numerical solutions obtained by semi-(spatially)-discrete dG schemes,
which are a frequently used tool for the numerical simulation of models of the forms \eqref{cx} and \eqref{sim} alike.
We will view $\vec v_h$ as a discretisation of the ``intermediate'' problem 
\begin{equation}
  \label{interm-c}
  \del_t \vec v + \div \vec f(\vec v) = \div \qp{ \widehat \eps \vec g(\vec v, \nabla \vec v)}\end{equation}
where $\widehat \eps$ is the model adaptation function, which will be chosen as part of the numerical method.

\begin{remark}[Alternative types of reconstruction]
If \eqref{cx} was a parabolic problem, this would be a quite strong argument in favour of using elliptic reconstruction, see \cite{MN03},
but this would make the residuals scale with $\frac{1}{\eps}.$ Recall that we are interested in the case of $\eps$ being small.
As important examples, \eg the  Navier-Stokes-Fourier equations, are not parabolic we will describe a reconstruction approach here
which was developed for semi-discrete dG schemes for hyperbolic problems in one space dimension in \cite{GMP_15}.
An extension to fully discrete methods can be found in \cite{DG_15}.
\end{remark}

Note that we state reconstructions in this paper to keep it self contained and to describe how we proceed in our numerical experiments in \S \ref{sec:num}.
It is, however, beyond the scope of this work to derive optimal reconstructions for \eqref{cx}, \eqref{sim}.
For all of these problems the derivation of optimal reconstructions of the numerical solution is a problem in its own right.
Note that in this framework {\it optimality} of a reconstruction  means that the error estimator, which is obtained
based on this reconstruction, is of the same order as the (true) error of the numerical
scheme.

We will first outline the reconstructions for \eqref{sim} in one space dimension, proposed in \cite{GMP_15}, and investigate in which sense they lead to reconstructions of 
numerical solutions to
\eqref{cx} or \eqref{interm-c}.
Afterwards we describe how the reconstruction approach can be extended to dG methods on Cartesian meshes in two space dimensions.
We choose Cartesian meshes because they lend themselves to an extension of the approach from \cite{GMP_15}.
We are not able to show the optimality of $\cR_H$ in this case, though.
Finding suitable (optimal) reconstructions for non-linear hyperbolic systems on unstructured meshes is the topic of ongoing research.

%-------------------------------------------------------------------------------------------------------------------------------------------------------------------
\subsection{A reconstruction approach for dG approximations of hyperbolic conservation laws}\label{subs:rec:claw}
In this section we recall a reconstruction approach for semi-(spatially)-discrete dG schemes for systems of hyperbolic conservation laws \eqref{sim}
complemented with initial data $\vec u(\cdot,0) = \vec u_0 \in L^\infty(\rT).$
We consider the one dimensional case. An extension to fully discrete schemes can be found in \cite{DG_15}. Let $\cT$ be a set of open intervals such that
\begin{equation}
 \bigcup_{S \in \cT} \bar S = \rT \text{ (the 1d torus)}, \text{ and } \text{ for all } S_1,S_2 \in \cT \text{ it holds } S_1=S_2 \text{ or }  S_1 \cap  S_2 = \emptyset.
\end{equation}
By $\cE$ we denote the set of interval boundaries.

The space of piecewise polynomial functions of degree $q \in \rN$ is defined by
\begin{equation}
 \fes_q := \{ \vec w : \rT \rightarrow \rR^n \, : \, \vec w|_S \in \rP_q(S,\rR^n) \ \forall \, S \in \cT\},
\end{equation}
where $ \rP_q(S,\rR^n)$ denotes the space of polynomials of degree $\leq q$ on $S$ with values in $\rR^n.$

For defining our scheme we also need jump and average operators which require the definition of a broken Sobolev space:
\begin{definition}[Broken Sobolev space]
 The broken Sobolev space $\sobh{1}(\cT,\rR^n)$ is defined by
 \begin{equation}
 \sobh{1}(\cT,\rR^n) := \{ \vec w : \rT \rightarrow \rR^n \, : \, \vec w|_S \in \sobh{1}(S,\rR^n) \, \forall \, S \in \cT\}.
\end{equation}
\end{definition}

\begin{definition}[Traces, jumps and averages]
 For any $\vec w \in \sobh{1}(\cT,\rR^n)$ we define 
 \begin{itemize}
  \item $\vec w^\pm : \cE \rightarrow \rR^n $ by $ \vec w^\pm (\cdot):= \lim_{ s \searrow 0} \vec w(\cdot \pm s) $ ,
  \item $\avg{\vec w} : \cE \rightarrow \rR^n $ by $ \avg{\vec w} = \frac{\vec w^- + \vec w^+}{2},$
   \item $\jump{\vec w} : \cE \rightarrow \rR^n $ by $ \jump{\vec w} = \vec w^- - \vec w^+.$
 \end{itemize}

\end{definition}

Now we are in position to state the numerical schemes under consideration:
\begin{definition}[Numerical scheme for \eqref{sim}]
  The numerical scheme is to seek $\vec v_h \in \cont{1}((0,T), \fes_q)$ such that:
\begin{equation}\label{eq:sddg}
\begin{split}
 &\vec v_h(0)= \cP_q [\vec u_0] \\
 &\int_{\cT} \partial_t \vec v_h \cdot \vec \phi - \vec f(\vec v_h) \cdot \partial_x \vec \phi + \int_{\cE} \vec F (\vec v_h^-,\vec v_h^+) \jump{\vec \phi} =0 \quad \text{for all } \vec \phi \in \fes_q,
 \end{split}
\end{equation}
where $\int_{\cT}$ is an abbreviation for $\sum_{S \in \cT} \int_S$, $\cP_q$ denotes $\leb{2}$-orthogonal projection into $\fes_q,$
and $\vec F: U \times U \rightarrow \mathbb{R}^n$ is a numerical flux function.
We impose that the numerical flux function satisfies the following condition:
There exist $L>0$ and $\vec w : U \times U \rightarrow U$ such that
\begin{equation}
 \label{eq:repf}
 \vec F (\vec u, \vec v) = \vec f(\vec w(\vec u, \vec v)) \quad \text{ for all } \vec u, \vec v \in U,
\end{equation}
and
\begin{equation}
 \label{eq:condw}
 \norm{\vec w(\vec u, \vec v) - \vec u} +  \norm{\vec w(\vec u, \vec v) - \vec v} \leq  L \norm{\vec u - \vec v} \quad \text{ for all } \vec u, \vec v \in \mathfrak{O}.
\end{equation}
\end{definition}

\begin{remark}[Conditions on the flux]
Note that conditions \eqref{eq:repf} and \eqref{eq:condw} imply the consistency and local Lipschitz continuity conditions usually imposed on numerical fluxes in the convergence 
analysis of dG approximations of hyperbolic conservation laws.
The conditions do not make the flux monotone nor do they ensure stability of \eqref{eq:sddg}. 
They do, however, ensure that the right hand side of
\eqref{eq:sddg} is Lipschitz continuous and, therefore, \eqref{eq:sddg} has unique solutions for small times.
Obviously, practical interest is restricted to numerical fluxes leading to reasonable stability properties of \eqref{eq:sddg} at least as long as the exact solution 
to \eqref{sim} is Lipschitz continuous.
Fluxes of Richtmyer and Lax-Wendroff type lead to stable numerical schemes (as long as the exact solution is smooth) and 
satisfy  a relaxed  version of conditions \eqref{eq:repf} and \eqref{eq:condw}.
It was shown in \cite{DG_15} that these relaxed conditions (see \cite[Rem. 3.6]{DG_15}) are sufficient for obtaining optimal a posteriori error estimates.
\end{remark}

Let us now return to the main purpose of this section: the definition of a reconstruction operator.
In addition, we present a reconstruction of the numerical flux which will be used for splitting the residual into a parabolic and a hyperbolic part, in \S \ref{subs:rec:hp}.
They are based on information from the numerical scheme:
\begin{definition}[Reconstructions]
  For each $t \in [0,T]$ we define the flux reconstruction $\hatf(\cdot,t) \in \fes_{q+1}$ through
 \begin{equation}\label{eq:rf1}
  \begin{split}
   \int_{\cT} \partial_x \hatf(\cdot,t) \cdot \vec \phi &= -\int_{\cT} \vec f(\vec v_h(\cdot,t)) \cdot \partial_x \vec \phi + \int_{\cE} \vec F (\vec v_h^-(\cdot,t),\vec v_h^+(\cdot,t)) \jump{\vec \phi}  \quad \text{for all } \vec \phi \in \fes_q\\
   \hatf^+(\cdot,t) &= \vec F (\vec v_h^-(\cdot,t),\vec v_h^+(\cdot,t)) \quad \text{ on } \cE.
  \end{split}
\end{equation}
For each $t \in [0,T]$ we define the reconstruction $\hatv(\cdot,t) \in \fes_{q+1}$ through
 \begin{equation}\label{eq:rf2}
  \begin{split}
   \int_{\rT} \hatv(\cdot,t) \cdot \vec \psi &=  \int_{\rT} \vec v_h (\cdot,t) \cdot \vec \psi  \quad \text{for all } \vec \psi \in \fes_{q-1}\\
   \hatv^\pm(\cdot,t) &= \vec w (\vec v_h^-(\cdot,t),\vec v_h^+(\cdot,t))  \quad \text{ on } \cE.
  \end{split}
\end{equation}
\end{definition}

\begin{remark}[Properties of reconstruction]
 It was shown in \cite{GMP_15} that these reconstructions are well-defined, explicitly and locally computable and Lipschitz continuous in space.
 Due to the Lipschitz continuity of $\vec w$ they are also Lipschitz continuous in time.
 Recall from \S \ref{sec:mees} and \S \ref{sec:meesy} that the Lipschitz continuity of $\hatv$ in space was crucial for our arguments.
\end{remark}

Due to the Lipschitz continuity of $\hatv$ the definition of the discretisation residual satisfies
\begin{equation}\label{eq:hypres}
 \cR := \partial_t \hatv + \partial_x \vec f(\hatv) \in \leb{\infty}.
\end{equation}
At this point the reader might ask why we have defined $\hatf$ as it is not present in \eqref{eq:hypres}
and  is not needed for computing the residual $\cR$  either.
We will use $\hatf$ in \S \ref{subs:rec:hp} to split the residual into a parabolic and a hyperbolic part.
As a preparation to this end let us note that upon combing 
\eqref{eq:sddg} and \eqref{eq:rf1} we obtain
\[ \partial_t \vec v_h + \partial_x \hatf =0\]
pointwise almost everywhere. Thus, we may split the residual as follows:
\[ \cR =  \partial_t ( \hatv  - \vec v_h) + \partial_x (\vec f(\hatv) - \hatf).\]
This splitting was used in \cite{GMP_15} to argue that the residual is of optimal order.
% Splitting residuals into different parts (along the lines of \eqref{eq:dpr}) will also be the purpose of introducing
% $\hatf$ in the forthcoming sections.

%-------------------------------------------------------------------------------------------------------------------------------------------------------------------
\subsection{A reconstruction approach for dG approximations of hyperbolic/parabolic problems}\label{subs:rec:hp}
We will describe in this Section how the reconstruction methodology described above can be used in case of dG semi- (spatial) discretisations
of \eqref{interm-c} in one space dimension following the local dG methodology.

\begin{definition}[Discrete gradients]
 By $\dgm, \dgp : \sobh{1}(\cT,\rR^m) \rightarrow \fes_q$ we denote discrete gradient operators defined through
\begin{equation}
 \int_{\rT} \dgpm \vec y \cdot \vec \phi = - \int_{\cT}  \vec y \cdot \partial_x \vec \phi + \int_{\cE} \vec y^\pm \jump{\vec \phi} \quad \text{ for all } \vec y, \vec \phi \in \fes_q.
\end{equation}
\end{definition}
\begin{lemma}[Discrete integration by parts]\label{lem:dibp}
The operators $\dgpm$ satisfy the following duality property:
For any $\vec \phi, \vec \psi \in \fes_q$ it holds
\[ \int_{\rT} \vec \phi \dgm \vec \psi = - \int_{\rT} \vec \psi \dgp \vec \phi .\]
\end{lemma}
The proof of Lemma \ref{lem:dibp} can be found in \cite{DE12}. 
Rewriting \eqref{interm-c} as 
\begin{equation}
 \begin{split}
  \vec s &= \partial_x \vec v\\
  \del_t \vec v + \partial_{x} \vec f (\vec v) &=   \partial_{x} (\widehat \eps \vec g(\vec v , \vec s)) 
 \end{split}
\end{equation}
 motivates the following semi-discrete numerical scheme:
\begin{definition}[Numerical scheme]
 The numerical solution $(\vec v_h, \vec s_h ) \in \qb{\cont{1}((0,T), \fes_q)}^2$
 is given as the solution to 
 \begin{equation}\label{eq:sddg2}
\begin{split}
 \vec v_h(0)&= \cP_q [\vec u_0] \\
  \vec s_h &= \dgm \vec v_h\\
 \int_{\cT} \partial_t \vec v_h \cdot \vec \phi - \vec f(\vec v_h) \cdot \partial_x \vec \phi + \widehat \eps \vec g(\vec v_h, \vec s_h) \dgm \vec \phi + \int_{\cE} \vec F (\vec v_h^-,\vec v_h^+) \jump{\vec \phi} 
  &=0  \quad \text{for all } \vec \phi \in \fes_q.
 \end{split}
\end{equation}
\end{definition}
% Probably we should add in references here, like Cockburn& Shu, but also more recent practitioners.

Defining $\hatf$ as in \eqref{eq:rf1} allows us to rewrite \eqref{eq:sddg2}$_3$, using Lemma \ref{lem:dibp}, as
\begin{equation}\label{nssf}
 \partial_t \vec v_h + \partial_x \hatf - \dgp \cP_q[ \widehat \eps \vec g(\vec v_h, \dgm \vec v_h)]=0.
\end{equation}

Due to the arguments given in \S \ref{subs:rec:claw} the reconstruction $\hatv$ is an element of $ \sob{1}{\infty}(\rT \times (0,T), \rR^n)$ such that the following 
residual makes sense in $\leb{2}(0,T;\sobh{-1}(\rT)):$
\begin{equation}
 \cR := \partial_t \hatv + \partial_x \vec f(\hatv) - \partial_x \qp{ \widehat \eps \vec g(\hatv, \partial_x \hatv)}.
\end{equation}
Using \eqref{nssf} we may rewrite the residual as
\begin{equation}\label{eq:dres}
 \cR = \underbrace{\partial_t (\hatv - \vec v_h) + \partial_x (\vec f(\hatv)- \hatf)}_{=: \cR_H} + \underbrace{\dgp \cP_q[ \widehat \eps \vec g(\vec v_h, \dgm \vec v_h)] - \partial_x (\widehat \eps \vec g(\hatv, \partial_x \hatv)) }_{=: \cR_P},
\end{equation}
\ie we have a decomposition of the residual as assumed in previous Sections, see \eqref{interm} in particular.

%-------------------------------------------------------------------------------------------------------------------------------------------------------------------
\subsection{Extension of the reconstruction to $2$ space dimensions}
\label{subs:rec:2d}
In this section we present an extension of the reconstruction approach described before to semi-(spatially)-discrete dG schemes for systems of hyperbolic conservation laws \eqref{sim}
complemented with initial data $\vec v(\cdot,0) = \vec v_0 \in L^\infty(\rT^2)$ using Cartesian meshes in two space dimensions.
The extension to Cartesian meshes in more than two dimensions is straightforward.

We consider a system of hyperbolic conservation laws in two space dimensions
\begin{equation}\label{claw}
 \del_t \vec v + \del_{x_1} \vec  f_1(\vec v) + \del_{x_2} \vec f_2(\vec v) =0,
\end{equation}
where $\vec f_{1,2} \in \cont{2}(U , \rR^n).$ 

We discretise $\rT^2$ using partitions 
\[-1=x_0 < x_1 < \dots < x_N =1, \quad -1=y_0 < y_1 < \dots < y_M =1.\] 
We consider a Cartesian mesh $\cT$ such that  each element  satisfies $K=[x_i,x_{i+1}] \times [y_j,y_{j+1}]$ 
for some $(i,j)\in \{0,\dots,N-1\} \times \{0,\dots,M-1\}.$
For any $p,q \in \rN$ and $K \in \cT$ let
\[ \rP_q \otimes \rP_p (K) := \rP_q([x_i,x_{i+1}]) \otimes \rP_p ([y_j,y_{j+1}]).\]
By $\fes_{p,q}$ we denote the space of trial and test functions, \ie
\[ \fes_{p,q}:= \{ \Phi : \rT^2  \rightarrow \rR^m \, : \, \Phi|_K \in (\rP_p \otimes \rP_q (K))^m \ \forall K \in \cT\}.\]
Note that our dG space has a tensorial structure on each element.

As before $\cE$ denotes the set of all edges, which can be decomposed into the sets of horizontal and vertical edges $\cE^h, \cE^v,$ respectively.
Let us define the following jump operators:
For $\Phi \in \sobh{1}(\cT,\rR^n)$ we define
\begin{align*}
  \jump{\Phi}^h &: \cE^v \rightarrow \rR^n\, \quad  \jump{\Phi}^h(\cdot) := \lim_{s \searrow 0} \Phi (\cdot - s \vec e_1) - \lim_{s \searrow 0} \Phi (\cdot + s\vec e_1)\\
    \jump{\Phi}^v & : \cE^h \rightarrow \rR^n, \quad  \jump{\Phi}^v(\cdot) := \lim_{s \searrow 0} \Phi ( \cdot - s \vec e_2) - \lim_{s \searrow 0} \Phi (\cdot + s\vec  e_2).
\end{align*}

Let $\vec F_{1,2}$ be  numerical flux functions which satisfy conditions \eqref{eq:repf} and \eqref{eq:condw}  with functions $\vec w_1, \vec w_2: U \times U \rightarrow U,$
\ie
\[ \vec F_i (\vec u, \vec v) = \vec f_i(\vec w_i(\vec u, \vec v)) \quad \text{ for all } \vec u, \vec v \in U \text{ and } i=1,2.\]
Then, we consider semi-(spatially)-discrete discontinuous Galerkin schemes  given as follows:
Search for $\vec v_h\in \cont{1}([0,\infty) , \fes_{q,q})$ satisfying
\begin{multline}\label{sch1}
 \int_{\T{}} \del_t \vec v_h \Phi  -  \vec f_1(\vec v)\del_{x_1} \Phi  -  \vec f_2(\vec v)\del_{x_2} \Phi\\
 + \int_{\E^v} \vec F_1(\vec v_h^-,\vec v_h^+) \jump{\Phi}^h + \int_{\E^h} \vec F_2(\vec v_h^-,\vec v_h^+) \jump{\Phi}^v =0 \quad \forall \Phi \in \fes_{q,q}.
\end{multline}

While we have avoided choosing particular bases of our dG spaces we will do so now as we believe that it makes the presentation of our reconstruction approach more concise.
We choose so-called \emph{nodal basis functions} consisting of Lagrange polynomials, see \cite{HW08}, and as we use a Cartesian mesh we may use tensor-products of one-dimensional Lagrange polynomials
to this end. 
We associate the Lagrange polynomials with Gauss points, as in this way the nodal basis functions form an orthogonal basis of our dG space $\fes_{q,q},$ due to
the exactness properties of Gauss quadrature, see \cite[e.g.]{Hin12}.
We will introduce some notation now:
Let $\{ \xi_0,\dots,\xi_{q}\} $ denote the Gauss points on $[-1,1].$ 
% 
% For any element $K= [x_i,x_{i+1}] \times [y_j,y_{j+1}] \in \cT$ let $\{ \xi_0^{K,1},\dots,\xi_{q}^{K,2}\} $ denote their image under the linear map $[-1,1] \rightarrow [x_i,x_{i+1}],$
% while $\{ \xi_0^{K,2},\dots,\xi_{q}^{K,2}\} $ denotes their image under the map $[-1,1] \rightarrow [y_j,y_{j+1}].$
%
For any element $K= [x_i,x_{i+1}] \times [y_j,y_{j+1}] \in \cT$ let $\{ \xi_0^{K,1},\dots,\xi_{q}^{K,1}\} $ and $\{ \xi_0^{K,2},\dots,\xi_{q}^{K,2}\} $ denote their image
under the linear bijections $[-1,1] \rightarrow [x_i,x_{i+1}]$ and $[-1,1] \rightarrow [y_j,y_{j+1}].$
For $i=1,2$ we denote by $\mathit{l}^{K,i}_j$  the Lagrange polynomial satisfying $\mathit{l}^{K,i}_j(\xi_k^{K,i})=\delta_{jk}$. 
%Due to the exactness property of  Gauss quadrature these Lagrange polynomials form an orthogonal basis of $\fes_{q,q}.$

\begin{definition}[Flux reconstruction]
Let $\hatf_1 \in \fes_{q+1,q} $ satisfy
\begin{equation}
 \label{recon1a}
 \int_{\T{}} (\del_{x_1} \hatf_1) \Phi  = - \int_{\T{}} \vec f_1(\vec v_h)\del_{x_1} \Phi 
 + \int_{\E^v} \cP_q[\vec F_1(\vec v_h^-,\vec v_h^+)] \jump{\Phi}^h \quad \forall \Phi \in \fes_{q,q},
\end{equation}
where $\cP_q$ denotes $\leb{2}$-orthogonal projection in the space of piece-wise polynomials of degree $\leq q$ on $\cE^v,$ and 
\begin{equation}
 \label{recon1b} \hatf_1(x_i,\xi^{K,2}_k)^+ := \lim_{s \searrow 0} \hatf_1(x_i+s,\xi^{K,2}_k) = \cP_q[\vec F_1(\vec v_h^-,\vec v_h^+)](x_i,\xi^{K,2}_k)
\end{equation}
for $k=0,\dots,q$ and all $K \in \cT.$
The definition of $\hatf_2 \in \fes_{q,q+1} $ is analogous.
\end{definition}

\begin{remark}[Regularity of flux reconstruction]
 Note that in order to split the residual in two space dimensions in a way analogous to what we did in \eqref{eq:dres} we require that for $\alpha=1,2$ the components of 
 the flux reconstruction
 $\hatf_\alpha $ are Lipschitz continuous in $x_\alpha$-direction.
 This is exactly what is needed such that $\partial_{x_\alpha} \hatf_\alpha$ makes sense in $\leb{\infty}.$
\end{remark}

\begin{lemma}[Properties of flux reconstruction]
 The flux reconstructions $\hatf_1$, $\hatf_2$ are well defined; and $\hatf_1$ is Lipschitz continuous in $x_1$-direction and $\hatf_2$ is Lipschitz continuous in $x_2$-direction.
\end{lemma}

\begin{proof}
We will give the proof for $\hatf_1$.  
For every $K \in \T{}$ the restriction $\hatf_1|_K$  is determined by \eqref{recon1a} up to a linear combination (in each component) of
\[ 1 \otimes \mathit{l}^{K,2}_0, \dots, 1 \otimes \mathit{l}^{K,2}_{q},\]
where $1$ denotes the polynomial having the value $1$ everywhere.
Prescribing \eqref{recon1b} obviously fixes these degrees of freedom.
Therefore, $\hatf_1$ exists, is uniquely determined, and locally computable.

For showing  that $\hatf_1$ is Lipschitz in the $x_1$-direction
it suffices to prove that $\hatf_1$ is continuous along the 'vertical' faces.
Let $K= [x_i,x_{i+1}] \times [y_j,y_{j+1}] \in \cT$ then we define 
\[ \chi_K^k := 1_{[x_i,x_{i+1}]} \otimes (l^{K,2}_k \cdot 1_{[y_j,y_{j+1}]})\]
where for any interval $I$ we denote the characteristic function of that interval by $1_I.$
For any  $k \in \{0,\dots, q\}$ we have on the one hand
\begin{equation}\label{recon1c} \int_{\T{}} \del_{x_1} \hatf_1  \chi_K^k
 = \omega_k h^y_j \int_{x_i}^{x_{i+1}} \del_{x_1} \hatf_1(\cdot, \xi^{K,2}_k) =  \omega_k h^y_j  \big(\hatf_1(x_{i+1}, \xi^{K,2}_k)^- - \hatf_1(x_{i}, \xi^{K,2}_k)^+\big)
 \end{equation}
 where $h^y_j = y_{j+1} - y_j$ and $\omega_k$ is the Gauss quadrature weight associated to $\xi_k$.
On the other hand we find, using \eqref{recon1a},
\begin{equation}\label{recon1d} \int_{\T{}} \del_{x_1} \hatf_1 \chi_K^k \\
 = \omega_k h^y_j \left( \cP_q[\vec F_1(\vec v_h^-,\vec v_h^+)](x_{i+1},\xi^{K,2}_k ) - \cP_q[\vec F_1(\vec v_h^-,\vec v_h^+)](x_{i},\xi^{K,2}_k )\right).
 \end{equation}
Combining \eqref{recon1c}, \eqref{recon1d} and \eqref{recon1b} 
we obtain
 \begin{equation}\label{recon1f} \hatf_1(x_{i+1}, \xi^{K,2}_k)^- 
  = \cP_q[\vec F_1(\vec v_h^-,\vec v_h^+)](x_{i+1},\xi^{K,2}_k )  = \hatf_1(x_{i+1}, \xi^{K,2}_k)^+ \quad \text{for } k=0,\dots,q.
 \end{equation}
As $\hatf_1(x_{i+1}, \cdot)^\pm|_{[y_j,y_{j+1}]}$ is a polynomial of degree $q$  and $k$ is arbitrary in equation \eqref{recon1f} we find 
\begin{equation}\label{recon1g}
 \hatf_1(x_{i+1}, \cdot)^+|_{[y_j,y_{j+1}]} = \hatf_1(x_{i+1}, \cdot)^-|_{[y_j,y_{j+1}]}.
\end{equation}
As $i,j$ were arbitrary this implies Lipschitz continuity of $\hatf_1$ in $x_1$-direction.

\end{proof}

From equations \eqref{sch1} and \eqref{recon1a} we obtain
the following pointwise equation almost everywhere:
\begin{equation}\label{sch4}
\del_t \vec v_h + \del_{x_1} \hatf_1 + \del_{x_2} \hatf_2=0 .
\end{equation}

\begin{remark}[Main idea of a $2$ dimensional reconstruction]
Recalling the arguments presented in previous Sections our main priority is to make $\hatv $ Lipschitz continuous.
 The particular reconstruction we describe is based on the following principles inspired by \S \ref{subs:rec:claw}.
We wish  $\hatv|_K - \vec v_h|_K$ to be orthogonal  to polynomials on $K$ of degree $q-1$ which is ensured by imposing them to coincide on the tensor product Gauss points.
We wish $\hatf_1 $ and $\vec f_1(\hatv)$ to be similar on vertical faces which is ensured by fixing the values of $\hatv $ on points of the form $(x_i, \xi^{K,2}_l)$ when $K= [x_i,x_{i+1}]\times [y_j,y_{j+1}].$
%After imposing the conditions described before there will be four degrees of freedom left which we fix by some averaging procedure.
Imposing the conditions described above on a reconstruction in $\fes_{q+1,q+1}$ is impossible because it does not have
enough degrees of freedom.
Thus, we define a reconstruction $\hatv \in \fes_{q+2,q+2}.$ For such a function imposing the degrees of freedom described above
leaves four degrees of freedom per cell undefined. Thus, we may prescribe values in corners.
To this end let us fix an averaging operator $\bar {\vec w}: U^4 \rightarrow U.$
\end{remark}

\begin{definition}[Solution reconstruction]
We define (at each time) the reconstruction $\hatv \in \fes_{q+2,q+2}$ of $\vec v_h \in \fes_{q,q}$ by prescribing for every $K= [x_i,x_{i+1}]\times [y_j,y_{j+1}] \in \cT$
\begin{equation}\label{urec}
 \begin{split}
  \hatv|_K (\xi^{K,1}_k, \xi^{K,2}_l) &= \vec v_h (\xi^{K,1}_k, \xi^{K,2}_l) \quad \text{ for } k,l=0,\dots,q \\
  \hatv|_K  (x_i,\xi^{K,2}_l ) &= \vec w_1 ( \vec v_h (x_i, \xi^{K,2}_l)^-, \vec v_h (x_i, \xi^{K,2}_l)^+) \quad \text{ for }  l=0,\dots,q \\
  \hatv|_K  (x_{i+1}, \xi^{K,2}_l) &=\vec w_1 (  \vec v_h (x_{i+1}, \xi^{K,2}_l)^-, \vec v_h (x_{i+1}, \xi^{K,2}_l)^+) \quad \text{ for }  l=0,\dots,q \\
   \hatv|_K  ( \xi^{K,1}_k, y_j) &=\vec w_2 (  \vec v_h ( \xi^{K,1}_k, y_j)^-, \vec v_h ( \xi^{K,1}_k, y_j)^+) \quad \text{ for }  k=0,\dots,q \\
      \hatv|_K  ( \xi^{K,1}_k, y_{j+1}) &=\vec w_2 (  \vec v_h ( \xi^{K,1}_k, y_{j+1})^-, \vec v_h ( \xi^{K,1}_k, y_{j+1})^+) \quad \text{ for }  k=0,\dots,q \\
  \hatv|_K  ( x_i, y_j) &=\bar{\vec w} ( \lim_{s \searrow 0} \vec v_h (x_i +s, y_j +s), 
                                         \lim_{s \searrow 0} \vec v_h (x_i -s, y_j +s),\\
                  & \quad \qquad        \lim_{s \searrow 0} \vec v_h (x_i +s, y_j -s),
                                         \lim_{s \searrow 0} \vec v_h (x_i -s, y_j -s))
 \end{split}
\end{equation}
and analogous prescriptions for the remaining three corners of $K$.
\end{definition}

\begin{lemma}[Properties of $\hatv$]
 The reconstruction $\hatv,$ is well-defined, locally computable and Lipschitz continuous.
 Moreover,  for $q \geq 1$ the following local conservation property is satisfied:
 \begin{equation}\label{eq:consrec} \int_{K} \hatv - \vec v_h =0 \quad \forall \ K \in \cT.\end{equation}
\end{lemma}

\begin{proof}
We will only prove the Lipschitz continuity and the conservation property. As $\hatv$ is piecewise polynomial it is sufficient to prove continuity to show Lipschitz continuity.
Let $K=[x_i,x_{i+1}]\times [y_j,y_{j+1}]$ and $K'= [x_{i-1},x_{i}]\times [y_j,y_{j+1}]$ then
$\hatv|_K $ and $\hatv|_{K'}$ coincide on $(x_i,y_j)$, $(x_i, \xi^{K,2}_k)_{k=0,\dots,q}$ and $(x_i,y_{j+1})$.
Therefore, $\hatv|_K $ and $\hatv|_{K'}$ coincide on $\{x_i\} \times [y_j,y_{j+1}]$. 
Analogous arguments hold for the other edges, such that $\hatv$ is indeed (Lipschitz) continuous.

As the nodal points on each element have tensor structure we can use the exactness properties of one-dimensional Gauss quadrature.
The conservation property \eqref{eq:consrec} follows from the fact that one-dimensional Gauss quadrature with $q+1$ Gauss points is exact for polynomials
of degree up to $2q+1$ which is larger or equal $q+2$ provided $q \geq 1.$
\end{proof}

\begin{remark}[Reconstructions for hyperbolic/parabolic problems in $2$ dimensions]
In order to obtain reconstructions and splittings of residuals into hyperbolic and parabolic parts for numerical discretisations of \eqref{interm-c}
 the reconstructions $\hatv, \hatf_\alpha$ described in this section can be used in the same way the reconstructions from \S \ref{subs:rec:claw} were used in  \S \ref{subs:rec:hp}.
 In particular, $\hatv$ described above is already regular enough to serve as a reconstruction in case of a numerical scheme
 for \eqref{interm-c}.
 The flux reconstructions $(\hatf_\alpha)_{\alpha=1,2}$ can be used to obtain a splitting analogous to \eqref{eq:dres} by making use of \eqref{sch4}.
\end{remark}

%-------------------------------------------------------------------------------------------------------------------------------------------------------------------

%-------------------------------------------------------------------------------------------------------------------------------------------------------------------
\section{Numerical experiments}

\label{sec:num}

% 
% %-------------------------------------------------------------------------------------------------------------------------------------------------------------------------------------------
% \newpage 
% Some thoughts on model coupling in this case: 
% For ease of notation let us restrict ourselves to an approximation of the one-dimensional Navier-Stokes equations
% \begin{equation}
% \begin{split}
%  \del_t \rho + \del_x m &=0
% \\
% \del_t m + \del_x (\frac{m^2}{\rho} + p(\rho)) &= \mu \del_{xx} \frac{m}{\rho}
%  \end{split}
% \end{equation}
% using a dG discretisation of the Euler equations on $\mathbb{R}_-$ and a dG discretisation of the Navier-Stokes equations on $\mathbb{R}_+.$
% 
% Let us assume that we discretize the convective part of the Navier-Stokes as well as the Euler equations using a Godunov numerical flux with components $F_1,F_2 : (0,\infty)^2 \times \mathbb{R}^2 \rightarrow \mathbb{R}$.
% We obtain the following spatial semi-discretisation of the Euler equations, where search for $\rho_h,m_h \in \cont{1}(0,T;\fes)$
% \begin{equation}
% \begin{split}
%  \sum_T \int_T \del_t \rho_h \Phi - m_h \del_x \Phi  + \sum_\E F_1(\rho_-,m_-,\rho_+,m_+) \jump{\Phi}&=0 \Foreach \Phi \in \fes \\
%   \sum_T \int_T \del_t m_h \Psi - (\frac{m_h^2}{\rho} + p(\rho_h)) \del_x \Phi  + 
%   \sum_\E F_1(\rho_-,m_-,\rho_+,m_+) \jump{\Psi} &=0 \Foreach \Phi \in \fes
% \end{split}\end{equation}
% while we consider the following discretisation of the Navier-Stokes equations

In this section we study the numerical behaviour of the error
indicators $\cE_M$ and $\cE_D$ presented in the previous Sections and
compare this with the ``error'', which we quantify as the difference
between the numerical approximation of the adaptive model and the
numerical approximation of the full model, on some test problems.

The model adaptive algorithm we employ is encapsulated by the following pseudocode:
\subsection{{$\Algoname{Model Adaptation}$}}
\label{alg:model-adapt}
\begin{algorithmic}
  \Require
  $(\tau,t_0,T,\vec u^0,\tol,\tol_c,\varepsilon)$
  \Ensure $(\vec u_h^n)_{\rangefromto n1N}$, model adaptive solution
  \State $\widehat \varepsilon(x,0):=0$
  \Comment{Initialise parameters}  
  \State $t = t_0 + \tau, n=1$
  \While{$t\leq T$}
  \Comment{Loop in time}
  \State
  $(\vec u_h^n) := \Algoname{Solve one timestep of dg scheme} (\vec u_h^{n-1},\widehat \varepsilon)$
  
  \State Compute $\cE_M$
  
  \For{$K\in \T{}$}
  \Comment{Model ``coarsening'' strategy}
  \If{$\cE_M|_K < \abs{K} \ tol_c \ tol/\varepsilon$}
  
  \State $\widehat \varepsilon(x,t)|_K = 0$
  
  \EndIf

  \EndFor

  \State Recompute $\cE_M$ and compute $\cE_D$
  
  \If{$\cE_D + \cE_M > \tol$}
  \Comment{Model ``refinement'' strategy}
  
  \State Mark a subset of elements, $\{ J \}$ where $\cE_D + \cE_M$ is large
  
  \For{$K\in \{ J \}$}
  
  \State Set $\widehat \varepsilon(x,t)|_K := \varepsilon$
  
  \EndFor

  \EndIf
  
  \State $t := t + \tau, n := n+1$
    
  \EndWhile
  
  \State return $(\vec u_h^n)_{\rangefromto n1N}$,
\end{algorithmic}

\begin{Rem}[Coupling to other adaptive procedures]
  The a posteriori bound given in Theorem \ref{the:2} has a structure
  which allows for both model and mesh adaptivity. This means that
  Algorithm \ref{alg:model-adapt} could be  coupled with other
  mechanisms employing $h$-$p$ spatial adaptivity in addition to local
  timestep control.  As can be seen from the pseudocode we use the
  complex model even in the case the discretisation error $\cE_D$ is
  large and the modelling error $\cE_M$ is small. In the first few
  tests we focus on the effect of model adaptation only and will
  demonstrate one possibility of coupling model and mesh adaptivity in
  the final test.
\end{Rem}

\subsection{Test 1 : The scalar case - the 1d viscous and inviscid Burgers' equation}
\label{sec:Burger1d}
We conduct an illustrative experiment using Burgers' equation. In this
case the ``complex'' model which we want to approximate is given by
\begin{equation}
  \label{eq:viscburger}
  \pdt u_\varepsilon + \pd{x}{\qp{\frac{u_\varepsilon^2}{2}}} = \varepsilon \pd{xx} u_\varepsilon
\end{equation}
for fixed $\varepsilon = 0.005$ with homogeneous Dirichlet boundary
data and the ``simple'' model we will use in the majority of the domain is
given by
\begin{equation}
  \label{eq:burger}
  \pdt u + \pd{x}{\qp{\frac{u^2}{2}}} = 0.
\end{equation}

We discretise the problem (\ref{eq:burger}) using a piecewise linear
dG scheme (\ref{eq:sddg}) together with Richtmyer type fluxes given by
\begin{equation}
  \label{eq:Richtmyer}
  F (v_h^-,v_h^+)
  =
  f\qp{\frac{1}{2}\qp{v_h^- + v_h^+} - \frac{\tau}{h}\qp{f(v_h^+) - f(v_h^-)}}.
\end{equation}
Note that these fluxes satisfy the assumptions
(\ref{eq:repf})--(\ref{eq:condw}). The dG formulation is then given by
\eqref{eq:sddg}. In the region where the ``complex'' model
(\ref{eq:viscburger}) is implemented for the discretisation of the
diffusion term we use an interior penalty (IP) discretisation with
piecewise constant $\widehat \varepsilon$, that is
\begin{equation}
  \widehat \varepsilon =
  \begin{cases}
    0.005 \text{ over cells where the a posteriori model error bound is large}
    \\
    0 \text{ otherwise.}
  \end{cases}
\end{equation}
This means the discretisation becomes
\begin{equation}\label{eq:sddg2n}
\begin{split}
  \int_{\cT} \partial_t u_h \cdot \phi -  f(u_h) \cdot \partial_x \phi + \int_{\cE} F (u_h^-, u_h^+) \jump{\phi}
  +
  \bih{u_h}{\phi; \widehat \varepsilon } &= 0
  \quad \text{for all } \phi \in \fes_q,
  \\
  u_h(0) &= \cP_q [u_0]
\end{split}
\end{equation}
where
\begin{equation}
  \bih{w_h}{\phi; \widehat \varepsilon }
  =
  \int_{\T{}} \widehat \varepsilon \pd x {w_h} \cdot \pd x {\phi}
  -
  \int_\E
  \jump{w_h}\cdot \avg{\widehat \varepsilon \pd x {\phi}}
  +
  \jump{\phi}\cdot \avg{\widehat \varepsilon \pd x {w_h}}
  -
  \frac{\sigma(\epsilon)}{h} \jump{w_h}\cdot\jump{\phi}
\end{equation}
and $\sigma(\epsilon) = 10\epsilon$ is the penalty parameter. Initial conditions are chosen as
\begin{equation}
  u(x,0) := \sin{x}
\end{equation}
over the interval $[-\pi,\pi]$. We use a first order IMEX temporal
discretisation where the diffusion is taken implicitly and the other
terms explicitly. We take $\tau = 10^{-4}$ and $h = \pi/500$ uniformly
over the space-time domain. For this test the parameters for Algorithm
\ref{alg:model-adapt} are $\tol = 10^{-2} \AND \tol_c = 10^{-3}$. Note
that these are user-specified parameters.  The results are summarised
in Figures \ref{fig:burger} and \ref{fig:burger-err}. In Figure
\ref{fig:burger} we show snapshots of the solution over time together
with the value of the model adaptivity parameter. In Figure
\ref{fig:burger-err} we display the error induced by solving the
``complex'' model only over part of the domain.

\begin{figure}[!ht]
  \caption[]
          {\label{fig:burger}
            %%%%%%%%%%%%%%%%%%%%%%%%%%%%%%%%%%%%%%%%%%%%%%%%%%%%%%%%%%%%%%%%%
            A numerical experiment testing model adaptivity on Burgers' equation. The simulation is described in \S\ref{sec:Burger1d}. 
            Here we display the solution at various times (top) together with a representation of the model adaptation parameter $\widehat \varepsilon$ (bottom).
            Blue is the region $\widehat \varepsilon=0$, where the simplified (inviscid Burgers') problem is being computed and red is where $\widehat \varepsilon = \varepsilon \neq 0$, where the full (viscous Burgers') problem is being computed.
            We see that initially only the simplified model is computed but as time progresses the full model is solved in a region around where the steep layer forms. As this forms the domain where the complex model is solved collapses and eventually is very localised around the layer.
    %%%%%%%%%%%%%%%%%%%%%%%%%%%%%%%%%%%%%%%%%%%%%%%%%%%%%%%%%%%%%%%%%% 
  }
  \begin{center}
    \subfigure[{
        %%%%%%%%%%%%%%%%%%%%%%%%%%%%%%%%%%%%%%%%%%%%%%%%%%%%%%%%%%%%%%%%
        $t=0$
        %%%%%%%%%%%%%%%%%%%%%%%%%%%%%%%%%%%%%%%%%%%%%%%%%%%%%%%%%%%%%%%% 
    }]{
      \includegraphics[scale=\figscale,width=0.47\figwidth]{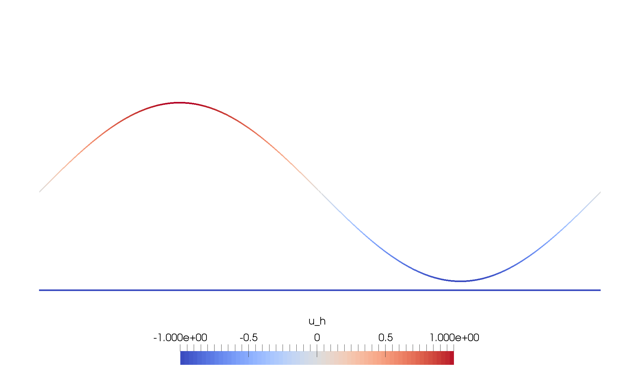}
    }
    %%%%%%%%%%%%%%%%%%%%%%%%%%%%%%%%%%%%%%%%%%%%%%%%%%%%%%%%%%%%%%%%%%%% 
    \hfill
    \subfigure[{
        %%%%%%%%%%%%%%%%%%%%%%%%%%%%%%%%%%%%%%%%%%%%%%%%%%%%%%%%%%%%%%%%
        $t=0.5375$
        %%%%%%%%%%%%%%%%%%%%%%%%%%%%%%%%%%%%%%%%%%%%%%%%%%%%%%%%%%%%%%%% 
    }]{
      \includegraphics[scale=\figscale,width=0.47\figwidth]{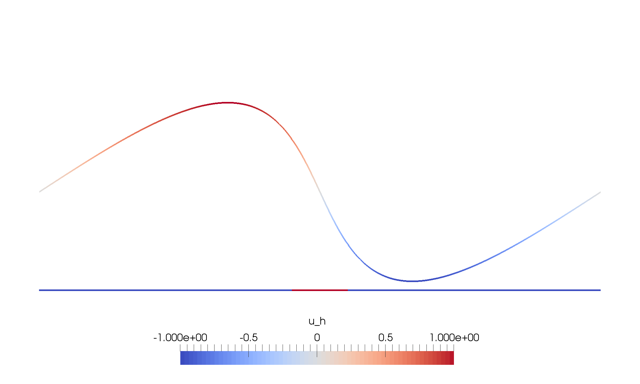}
    }
    %%%%%%%%%%%%%%%%%%%%%%%%%%%%%%%%%%%%%%%%%%%%%%%%%%%%%%%%%%%%%%%%%%%% 
    \hfill
    \subfigure[{
        %%%%%%%%%%%%%%%%%%%%%%%%%%%%%%%%%%%%%%%%%%%%%%%%%%%%%%%%%%%%%%%%
        $t=1.1625$
        %%%%%%%%%%%%%%%%%%%%%%%%%%%%%%%%%%%%%%%%%%%%%%%%%%%%%%%%%%%%%%%% 
    }]{
      \includegraphics[scale=\figscale,width=0.47\figwidth]{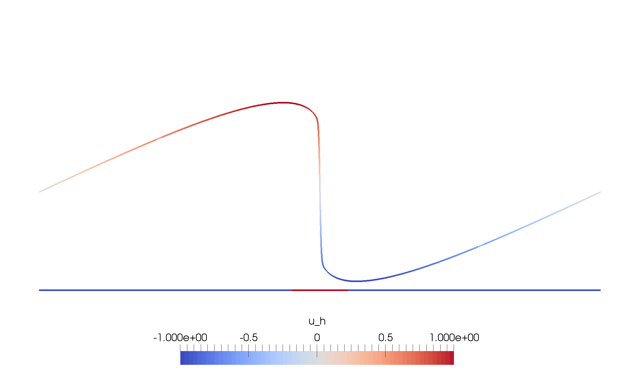}
    }
    %%%%%%%%%%%%%%%%%%%%%%%%%%%%%%%%%%%%%%%%%%%%%%%%%%%%%%%%%%%%%%%%%%%% 
    \hfill
    \subfigure[{
        %%%%%%%%%%%%%%%%%%%%%%%%%%%%%%%%%%%%%%%%%%%%%%%%%%%%%%%%%%%%%%%%
        $t=1.3$
        %%%%%%%%%%%%%%%%%%%%%%%%%%%%%%%%%%%%%%%%%%%%%%%%%%%%%%%%%%%%%%%% 
    }]{
      \includegraphics[scale=\figscale,width=0.47\figwidth]{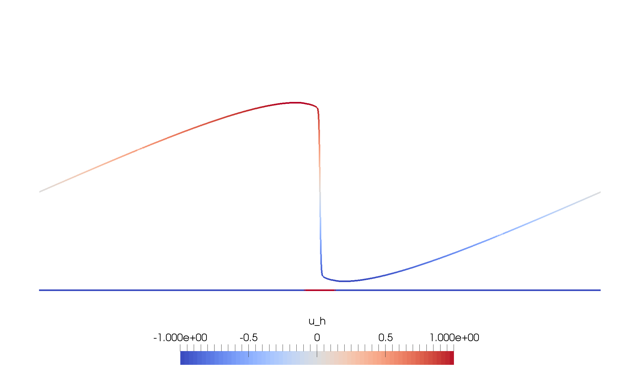}
    }
    %%%%%%%%%%%%%%%%%%%%%%%%%%%%%%%%%%%%%%%%%%%%%%%%%%%%%%%%%%%%%%%%%%%% 
    \hfill
    \subfigure[{
        %%%%%%%%%%%%%%%%%%%%%%%%%%%%%%%%%%%%%%%%%%%%%%%%%%%%%%%%%%%%%%%%
        $t=1.55$
        %%%%%%%%%%%%%%%%%%%%%%%%%%%%%%%%%%%%%%%%%%%%%%%%%%%%%%%%%%%%%%%% 
    }]{
      \includegraphics[scale=\figscale,width=0.47\figwidth]{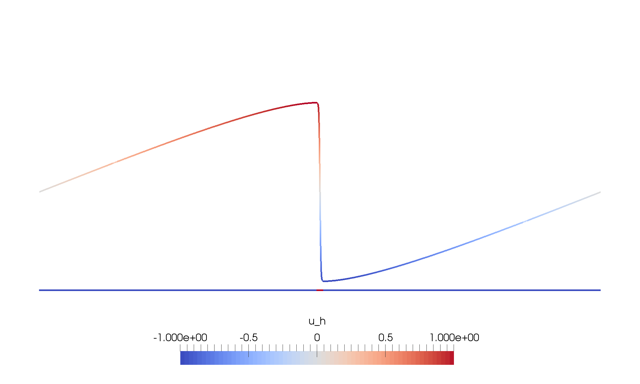}
    }
    %%%%%%%%%%%%%%%%%%%%%%%%%%%%%%%%%%%%%%%%%%%%%%%%%%%%%%%%%%%%%%%%%%%% 
    \hfill
    \subfigure[{
        %%%%%%%%%%%%%%%%%%%%%%%%%%%%%%%%%%%%%%%%%%%%%%%%%%%%%%%%%%%%%%%%
        $t=2.5$
        %%%%%%%%%%%%%%%%%%%%%%%%%%%%%%%%%%%%%%%%%%%%%%%%%%%%%%%%%%%%%%%% 
    }]{
      \includegraphics[scale=\figscale,width=0.47\figwidth]{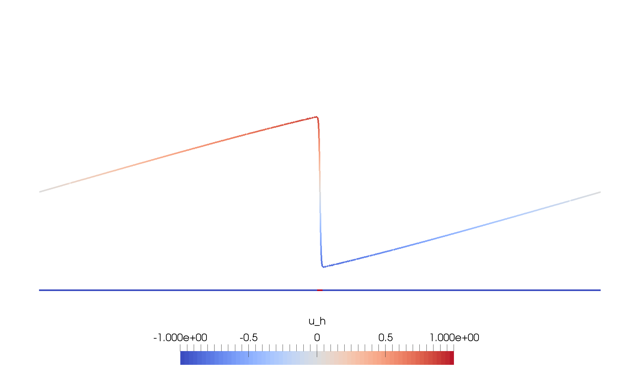}
    }
    %%%%%%%%%%%%%%%%%%%%%%%%%%%%%%%%%%%%%%%%%%%%%%%%%%%%%%%%%%%%%%%%%%%% 
    \hfill
    \end{center}
  \end{figure}

\begin{figure}[!ht]
  \caption[]
          {\label{fig:burger-err}
            %%%%%%%%%%%%%%%%%%%%%%%%%%%%%%%%%%%%%%%%%%%%%%%%%%%%%%%%%%%%%%%%%
            A numerical experiment testing model adaptivity on Burgers' equation. The simulation is described in \S\ref{sec:Burger1d}. Here we display the error $\norm{u_h-u_{\varepsilon,h}}$, that is, the difference between the approximation of the full expensive model and that of the adaptive approximation at the same times as in Figure \ref{fig:burger} together with a representation of $\widehat \varepsilon$ (bottom).
            An interesting phenomenon is the propagation of dispersive waves emanating from the interface between the region where $\widehat \varepsilon=0$ and that of $\widehat \varepsilon= \varepsilon \neq 0$.
    %%%%%%%%%%%%%%%%%%%%%%%%%%%%%%%%%%%%%%%%%%%%%%%%%%%%%%%%%%%%%%%%%% 
  }
  \begin{center}
    \subfigure[{
        %%%%%%%%%%%%%%%%%%%%%%%%%%%%%%%%%%%%%%%%%%%%%%%%%%%%%%%%%%%%%%%%
        $t=0$
        %%%%%%%%%%%%%%%%%%%%%%%%%%%%%%%%%%%%%%%%%%%%%%%%%%%%%%%%%%%%%%%% 
    }]{
      \includegraphics[scale=\figscale,width=0.47\figwidth]{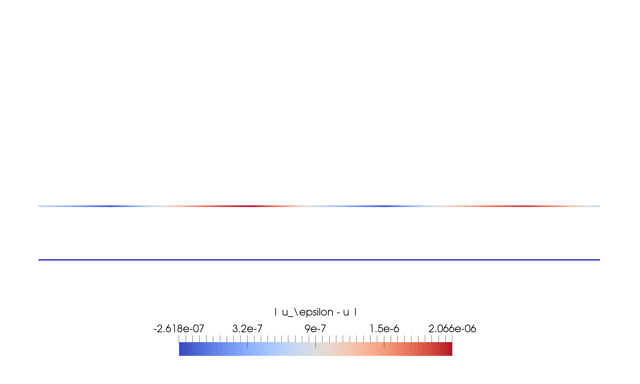}
    }
    %%%%%%%%%%%%%%%%%%%%%%%%%%%%%%%%%%%%%%%%%%%%%%%%%%%%%%%%%%%%%%%%%%%% 
    \hfill
    \subfigure[{
        %%%%%%%%%%%%%%%%%%%%%%%%%%%%%%%%%%%%%%%%%%%%%%%%%%%%%%%%%%%%%%%%
        $t=0.5375$
        %%%%%%%%%%%%%%%%%%%%%%%%%%%%%%%%%%%%%%%%%%%%%%%%%%%%%%%%%%%%%%%% 
    }]{
      \includegraphics[scale=\figscale,width=0.47\figwidth]{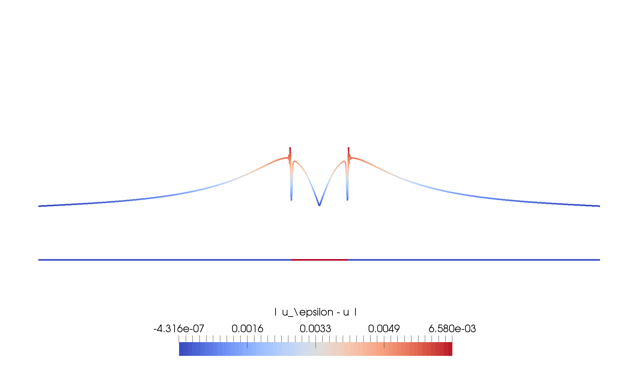}
    }
    %%%%%%%%%%%%%%%%%%%%%%%%%%%%%%%%%%%%%%%%%%%%%%%%%%%%%%%%%%%%%%%%%%%% 
    \hfill
    \subfigure[{
        %%%%%%%%%%%%%%%%%%%%%%%%%%%%%%%%%%%%%%%%%%%%%%%%%%%%%%%%%%%%%%%%
        $t=1.1625$
        %%%%%%%%%%%%%%%%%%%%%%%%%%%%%%%%%%%%%%%%%%%%%%%%%%%%%%%%%%%%%%%% 
    }]{
      \includegraphics[scale=\figscale,width=0.47\figwidth]{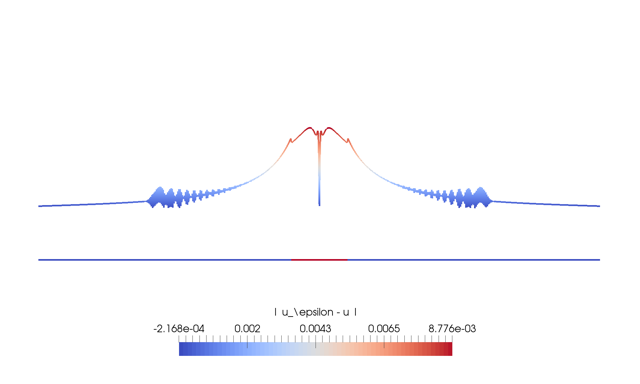}
    }
    %%%%%%%%%%%%%%%%%%%%%%%%%%%%%%%%%%%%%%%%%%%%%%%%%%%%%%%%%%%%%%%%%%%% 
    \hfill
    \subfigure[{
        %%%%%%%%%%%%%%%%%%%%%%%%%%%%%%%%%%%%%%%%%%%%%%%%%%%%%%%%%%%%%%%%
        $t=1.3$
        %%%%%%%%%%%%%%%%%%%%%%%%%%%%%%%%%%%%%%%%%%%%%%%%%%%%%%%%%%%%%%%% 
    }]{
      \includegraphics[scale=\figscale,width=0.47\figwidth]{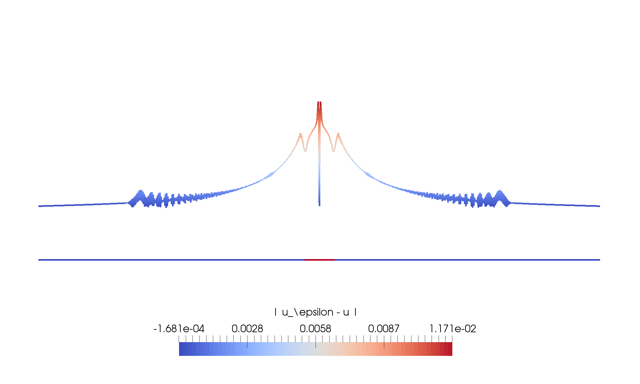}
    }
    %%%%%%%%%%%%%%%%%%%%%%%%%%%%%%%%%%%%%%%%%%%%%%%%%%%%%%%%%%%%%%%%%%%% 
    \hfill
    \subfigure[{
        %%%%%%%%%%%%%%%%%%%%%%%%%%%%%%%%%%%%%%%%%%%%%%%%%%%%%%%%%%%%%%%%
        $t=1.55$
        %%%%%%%%%%%%%%%%%%%%%%%%%%%%%%%%%%%%%%%%%%%%%%%%%%%%%%%%%%%%%%%% 
    }]{
      \includegraphics[scale=\figscale,width=0.47\figwidth]{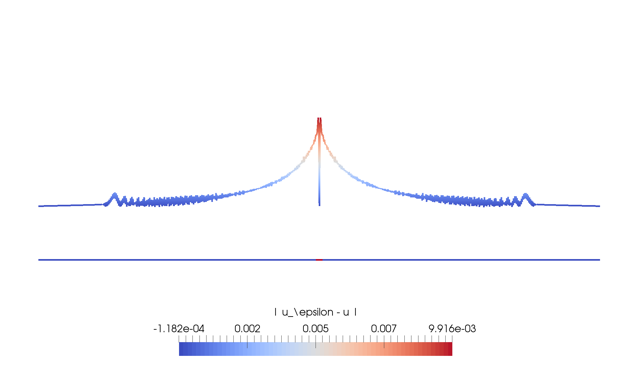}
    }
    %%%%%%%%%%%%%%%%%%%%%%%%%%%%%%%%%%%%%%%%%%%%%%%%%%%%%%%%%%%%%%%%%%%% 
    \hfill
    \subfigure[{
        %%%%%%%%%%%%%%%%%%%%%%%%%%%%%%%%%%%%%%%%%%%%%%%%%%%%%%%%%%%%%%%%
        $t=2.5$
        %%%%%%%%%%%%%%%%%%%%%%%%%%%%%%%%%%%%%%%%%%%%%%%%%%%%%%%%%%%%%%%% 
    }]{
      \includegraphics[scale=\figscale,width=0.47\figwidth]{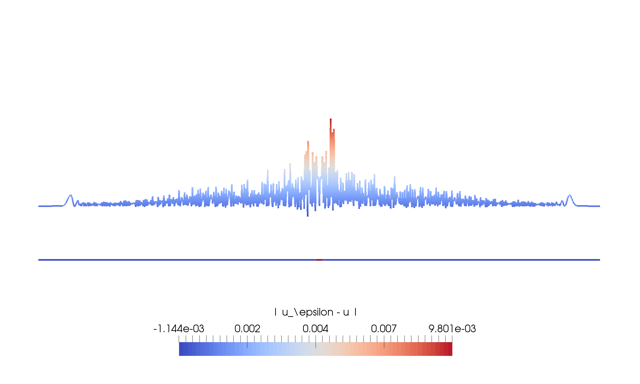}
    }
    %%%%%%%%%%%%%%%%%%%%%%%%%%%%%%%%%%%%%%%%%%%%%%%%%%%%%%%%%%%%%%%%%%%% 
    \hfill
    \end{center}
  \end{figure}

\subsection{Test 2 : The scalar case - the 2d viscous and inviscid Burgers' equation}
\label{sec:Burger2d}
In this test we examine how the adaptive procedure extends into the multi-dimensional setting again using Burgers' equation as an illustrative example. In this
case the ``complex'' model which we want to approximate is given by
\begin{equation}
  \label{eq:viscburger2d}
  \pdt u_\varepsilon + \div{ \qp{\frac{\one u_\varepsilon^2}{2}}} = \varepsilon \Delta u_\varepsilon,
\end{equation}
where $\one = \Transpose{\qp{1,1}}$. The simple model we will use in the majority of the domain is given by 
\begin{equation}
  \label{eq:burger2d}
  \pdt u + \div{\qp{ \frac{\one u^2}{2}}} = 0.
\end{equation}
These are coupled with homogenous Dirichlet boundary conditions.  As
in Test 1 we make use of a 1st order IMEX, piecewise linear dG scheme
together with Richtmyer fluxes and an IP method for the viscosity. We
pick an initial condition
\begin{equation}
  u(\vec x,0) = \exp\qp{-10\norm{\vec x}^2}
\end{equation}
and use the parameters $\varepsilon = 0.01$, $h = \sqrt{2}/50$,
$\tau = \sqrt{2}/400$, $\tol = 10^{-2}$ and $\tol_c = 10^{-3}$ in
Algorithm \ref{alg:model-adapt}. The results are summarised in Figures
\ref{fig:burger2d} and \ref{fig:burger2d-err}. In Figure \ref{fig:burger2d}  snapshots of
the solution over time together with the value of the model adaptivity
parameter are shown. Figure \ref{fig:burger2d-err} displays the error induced by solving the ``complex'' model only
over part of the domain.

\begin{figure}[!ht]
  \caption[]
          {\label{fig:burger2d}
            %%%%%%%%%%%%%%%%%%%%%%%%%%%%%%%%%%%%%%%%%%%%%%%%%%%%%%%%%%%%%%%%%
            A numerical experiment testing model adaptivity on Burgers' equation. The simulation is described in \S \ref{sec:Burger2d}. Here we display the solution at various times (top) together with a representation of $\widehat \varepsilon$ (bottom). Blue is the region $\widehat \varepsilon=0$,
            where the simplified (inviscid Burgers') problem is being computed and red is where $\widehat \varepsilon= \varepsilon  \neq 0$, where the full (viscous Burgers') problem is being computed.
    %%%%%%%%%%%%%%%%%%%%%%%%%%%%%%%%%%%%%%%%%%%%%%%%%%%%%%%%%%%%%%%%%% 
  }
  \begin{center}
    \subfigure[{
        %%%%%%%%%%%%%%%%%%%%%%%%%%%%%%%%%%%%%%%%%%%%%%%%%%%%%%%%%%%%%%%%
        $t=0.0025$
        %%%%%%%%%%%%%%%%%%%%%%%%%%%%%%%%%%%%%%%%%%%%%%%%%%%%%%%%%%%%%%%% 
    }]{
      \includegraphics[scale=\figscale,width=0.47\figwidth]{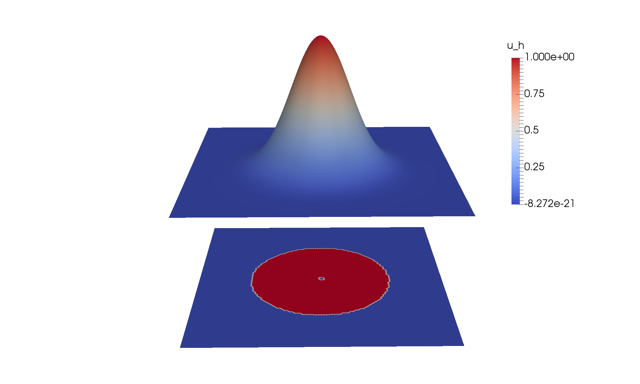}
    }
    %%%%%%%%%%%%%%%%%%%%%%%%%%%%%%%%%%%%%%%%%%%%%%%%%%%%%%%%%%%%%%%%%%%% 
    \hfill
    \subfigure[{
        %%%%%%%%%%%%%%%%%%%%%%%%%%%%%%%%%%%%%%%%%%%%%%%%%%%%%%%%%%%%%%%%
        $t=0.25$
        %%%%%%%%%%%%%%%%%%%%%%%%%%%%%%%%%%%%%%%%%%%%%%%%%%%%%%%%%%%%%%%% 
    }]{
      \includegraphics[scale=\figscale,width=0.47\figwidth]{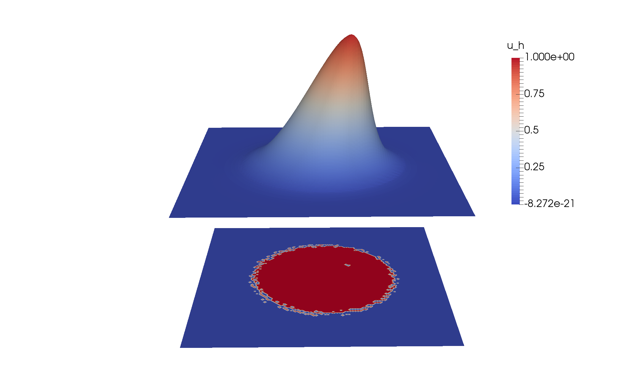}
    }
    %%%%%%%%%%%%%%%%%%%%%%%%%%%%%%%%%%%%%%%%%%%%%%%%%%%%%%%%%%%%%%%%%%%% 
    \hfill
    \subfigure[{
        %%%%%%%%%%%%%%%%%%%%%%%%%%%%%%%%%%%%%%%%%%%%%%%%%%%%%%%%%%%%%%%%
        $t=0.5$
        %%%%%%%%%%%%%%%%%%%%%%%%%%%%%%%%%%%%%%%%%%%%%%%%%%%%%%%%%%%%%%%% 
    }]{
      \includegraphics[scale=\figscale,width=0.47\figwidth]{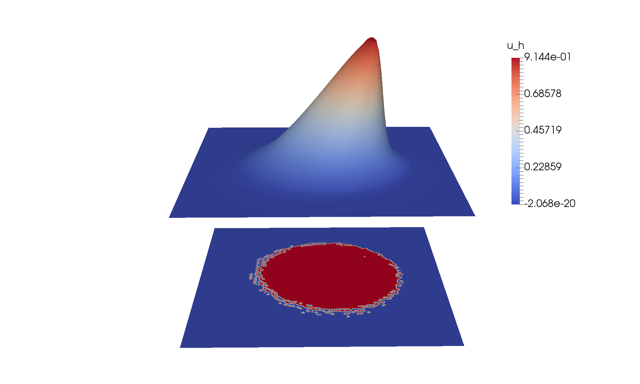}
    }
    %%%%%%%%%%%%%%%%%%%%%%%%%%%%%%%%%%%%%%%%%%%%%%%%%%%%%%%%%%%%%%%%%%%% 
    \hfill
    \subfigure[{
        %%%%%%%%%%%%%%%%%%%%%%%%%%%%%%%%%%%%%%%%%%%%%%%%%%%%%%%%%%%%%%%%
        $t=1.$
        %%%%%%%%%%%%%%%%%%%%%%%%%%%%%%%%%%%%%%%%%%%%%%%%%%%%%%%%%%%%%%%% 
    }]{
      \includegraphics[scale=\figscale,width=0.47\figwidth]{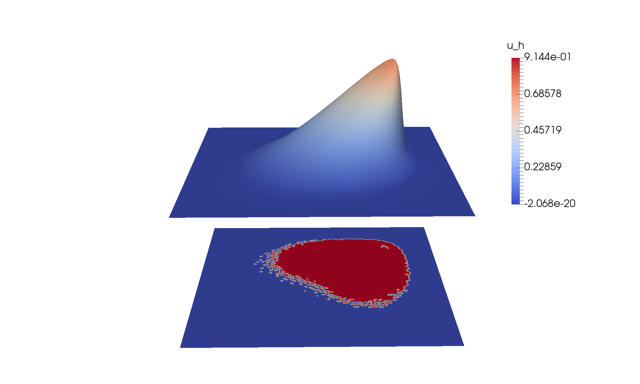}
    }
    %%%%%%%%%%%%%%%%%%%%%%%%%%%%%%%%%%%%%%%%%%%%%%%%%%%%%%%%%%%%%%%%%%%% 
    \hfill
    \subfigure[{
        %%%%%%%%%%%%%%%%%%%%%%%%%%%%%%%%%%%%%%%%%%%%%%%%%%%%%%%%%%%%%%%%
        $t=1.25$
        %%%%%%%%%%%%%%%%%%%%%%%%%%%%%%%%%%%%%%%%%%%%%%%%%%%%%%%%%%%%%%%% 
    }]{
      \includegraphics[scale=\figscale,width=0.47\figwidth]{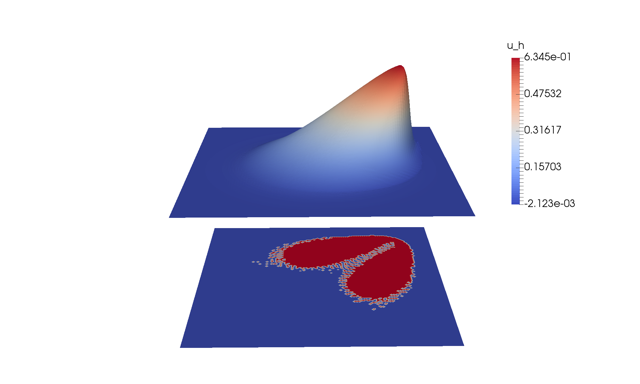}
    }
    %%%%%%%%%%%%%%%%%%%%%%%%%%%%%%%%%%%%%%%%%%%%%%%%%%%%%%%%%%%%%%%%%%%% 
    \hfill
    \subfigure[{
        %%%%%%%%%%%%%%%%%%%%%%%%%%%%%%%%%%%%%%%%%%%%%%%%%%%%%%%%%%%%%%%%
        $t=1.5$
        %%%%%%%%%%%%%%%%%%%%%%%%%%%%%%%%%%%%%%%%%%%%%%%%%%%%%%%%%%%%%%%% 
    }]{
      \includegraphics[scale=\figscale,width=0.47\figwidth]{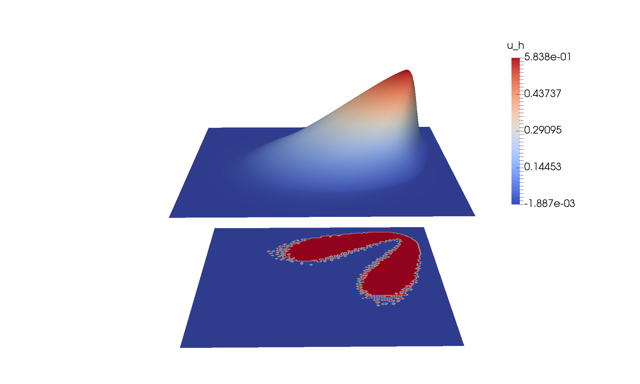}
    }
    %%%%%%%%%%%%%%%%%%%%%%%%%%%%%%%%%%%%%%%%%%%%%%%%%%%%%%%%%%%%%%%%%%%% 
    \hfill
    \end{center}
  \end{figure}

\begin{figure}[!ht]
  \caption[]
          {\label{fig:burger2d-err}
            %%%%%%%%%%%%%%%%%%%%%%%%%%%%%%%%%%%%%%%%%%%%%%%%%%%%%%%%%%%%%%%%%
            A numerical experiment testing model adaptivity on Burgers' equation. The simulation is described in \S\ref{sec:Burger2d}. Here we display the error $\norm{u_h-u_{\varepsilon,h}}$, that is, the difference between the approximation of the full expensive model and that of the adaptive approximation at the same times as in Figure \ref{fig:burger2d}.
    %%%%%%%%%%%%%%%%%%%%%%%%%%%%%%%%%%%%%%%%%%%%%%%%%%%%%%%%%%%%%%%%%% 
  }
  \begin{center}
    \subfigure[{
        %%%%%%%%%%%%%%%%%%%%%%%%%%%%%%%%%%%%%%%%%%%%%%%%%%%%%%%%%%%%%%%%
        $t=0.0025$
        %%%%%%%%%%%%%%%%%%%%%%%%%%%%%%%%%%%%%%%%%%%%%%%%%%%%%%%%%%%%%%%% 
    }]{
      \includegraphics[scale=\figscale,width=0.47\figwidth]{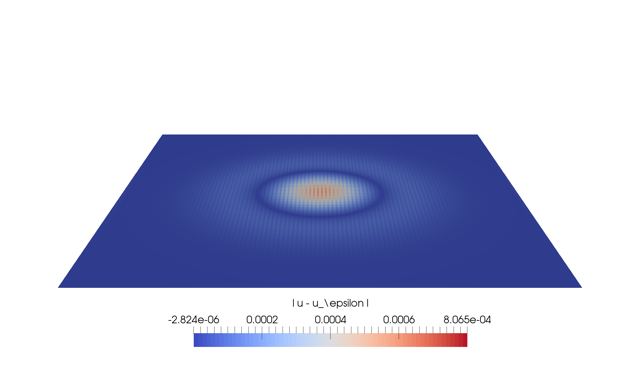}
    }
    %%%%%%%%%%%%%%%%%%%%%%%%%%%%%%%%%%%%%%%%%%%%%%%%%%%%%%%%%%%%%%%%%%%% 
    \hfill
    \subfigure[{
        %%%%%%%%%%%%%%%%%%%%%%%%%%%%%%%%%%%%%%%%%%%%%%%%%%%%%%%%%%%%%%%%
        $t=0.25$
        %%%%%%%%%%%%%%%%%%%%%%%%%%%%%%%%%%%%%%%%%%%%%%%%%%%%%%%%%%%%%%%% 
    }]{
      \includegraphics[scale=\figscale,width=0.47\figwidth]{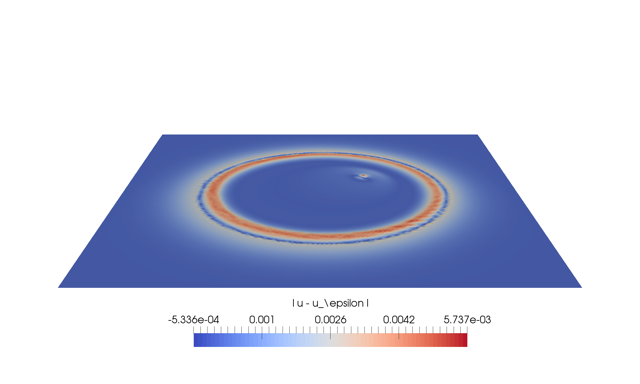}
    }
    %%%%%%%%%%%%%%%%%%%%%%%%%%%%%%%%%%%%%%%%%%%%%%%%%%%%%%%%%%%%%%%%%%%% 
    \hfill
    \subfigure[{
        %%%%%%%%%%%%%%%%%%%%%%%%%%%%%%%%%%%%%%%%%%%%%%%%%%%%%%%%%%%%%%%%
        $t=0.5$
        %%%%%%%%%%%%%%%%%%%%%%%%%%%%%%%%%%%%%%%%%%%%%%%%%%%%%%%%%%%%%%%% 
    }]{
      \includegraphics[scale=\figscale,width=0.47\figwidth]{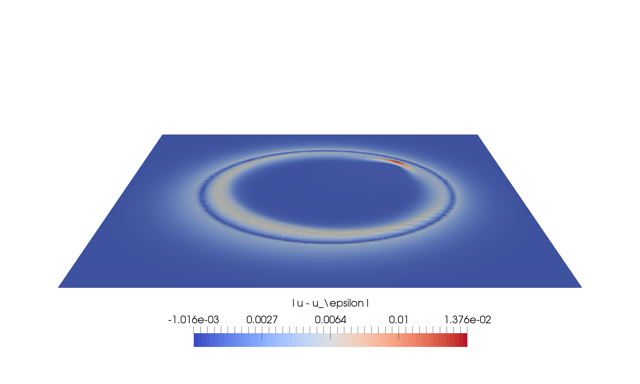}
    }
    %%%%%%%%%%%%%%%%%%%%%%%%%%%%%%%%%%%%%%%%%%%%%%%%%%%%%%%%%%%%%%%%%%%% 
    \hfill
    \subfigure[{
        %%%%%%%%%%%%%%%%%%%%%%%%%%%%%%%%%%%%%%%%%%%%%%%%%%%%%%%%%%%%%%%%
        $t=1.$
        %%%%%%%%%%%%%%%%%%%%%%%%%%%%%%%%%%%%%%%%%%%%%%%%%%%%%%%%%%%%%%%% 
    }]{
      \includegraphics[scale=\figscale,width=0.47\figwidth]{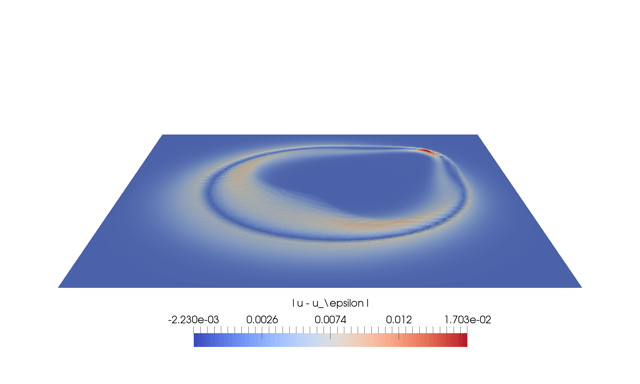}
    }
    %%%%%%%%%%%%%%%%%%%%%%%%%%%%%%%%%%%%%%%%%%%%%%%%%%%%%%%%%%%%%%%%%%%% 
    \hfill
    \subfigure[{
        %%%%%%%%%%%%%%%%%%%%%%%%%%%%%%%%%%%%%%%%%%%%%%%%%%%%%%%%%%%%%%%%
        $t=1.25$
        %%%%%%%%%%%%%%%%%%%%%%%%%%%%%%%%%%%%%%%%%%%%%%%%%%%%%%%%%%%%%%%% 
    }]{
      \includegraphics[scale=\figscale,width=0.47\figwidth]{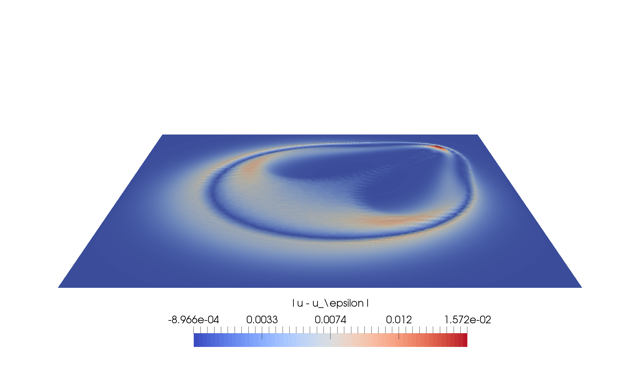}
    }
    %%%%%%%%%%%%%%%%%%%%%%%%%%%%%%%%%%%%%%%%%%%%%%%%%%%%%%%%%%%%%%%%%%%% 
    \hfill
    \subfigure[{
        %%%%%%%%%%%%%%%%%%%%%%%%%%%%%%%%%%%%%%%%%%%%%%%%%%%%%%%%%%%%%%%%
        $t=1.5$
        %%%%%%%%%%%%%%%%%%%%%%%%%%%%%%%%%%%%%%%%%%%%%%%%%%%%%%%%%%%%%%%% 
    }]{
      \includegraphics[scale=\figscale,width=0.47\figwidth]{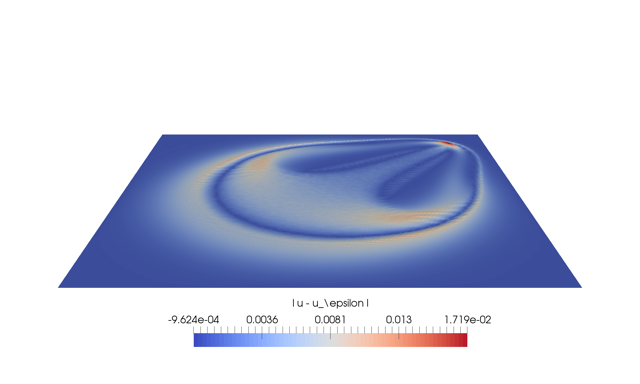}
    }
    %%%%%%%%%%%%%%%%%%%%%%%%%%%%%%%%%%%%%%%%%%%%%%%%%%%%%%%%%%%%%%%%%%%% 
    \end{center}
  \end{figure}

\clearpage
\subsection{Test 3 : The Isothermal Navier-Stokes system}
\label{sec:INS}
In this test we examine the use of model adaptivity on the Isothermal
Navier-Stokes system in a situation where K\'arm\'an vortices are produced by
a flow over a cylinder. This can be achieved by seeking $\qp{\rho_\mu,
  \rho\bv_{\mu}}$ such that
\begin{equation}
  \label{eq:NS}
  \begin{split}
    \partial_t \rho_\mu + \div(\rho_\mu \bv_\mu) &=0\\
    \partial_t (\rho_\mu \bv_\mu) + \div(\rho_\mu \bv_\mu  \otimes \bv_\mu) + \nabla p &= \div (\mu \nabla \bv_\mu),
  \end{split}
\end{equation}
with a Reynolds number of $100$ (i.e., $\mu = \frac 1 {100}$). In this case our ``complex'' model is given by (\ref{eq:NS}) and the approximation is given by
\begin{equation}
  \label{eq:Euler}
  \begin{split}
    \partial_t \rho + \div(\rho \bv) &=0\\
    \partial_t (\rho \bv) + \div(\rho \bv  \otimes \bv) + \nabla p &= 0.
  \end{split}
\end{equation}
Here we use the relation $p(\rho) = 0.2 \rho$, through Boyle's law.
We write $\vec w = \qp{\rho_\mu, \rho_\mu \bv_\mu}$ for the unknowns
and then (\ref{eq:NS}) becomes
\begin{equation}
  \pd t {\vec w} + \div(\vec f(\vec w)) = \vec D(\vec w),
\end{equation}
with
\begin{equation}
  \vec D(\vec w) =
  \begin{bmatrix}
    0
    \\
    \div\qp{\mu \nabla \bv_\mu}
  \end{bmatrix}
\end{equation}
representing the diffusion term. We choose a piecewise linear dG
scheme similar to that in previous tests of the form
\begin{equation}\label{eq:insdg}
\begin{split}
  \int_{\cT} \partial_t \vec w_h \cdot \vec \phi -  \vec f(\vec w_h) : \nabla \vec \phi + \int_{\cE} \vec F (\vec w_h^-, \vec w_h^+) \jump{\vec \phi}
  &=
  \cD\qp{{\vec w_h,\vec \phi; \widehat \mu }}
  \quad \text{for all } \vec \phi \in \fes_q,
  \\
  u_h(0) &= \cP_q [u_0]
\end{split}
\end{equation}
where, in $2d$,
\begin{equation}
  \vec f(\vec w)
  =
  \begin{bmatrix}
    \rho v_1 & \rho v_2
    \\
    \rho v_1^2 + p & \rho v_1 v_2
    \\
    \rho v_2 v_1 & \rho v_2^2 + p
  \end{bmatrix},
\end{equation}
the $\div$ operator is understood to act row-wise,
\begin{equation}
  \cD\qp{\vec w_h, \vec \phi; \widehat \mu}
  =
  \begin{pmatrix}
    0
    \\
    -\bih{v_1}{\phi_2; \widehat \mu}
    \\
    -\bih{v_2}{\phi_3; \widehat \mu}
 \end{pmatrix},
\end{equation}
and
\begin{equation}
  \bih{v}{\phi; \widehat \mu }
  =
  \int_{\T{}} \widehat \mu \nabla v \cdot \nabla {\phi}
  -
  \int_\E
  \jump{v} \cdot \avg{\widehat \mu \nabla {\phi}}
  +
  \jump{\phi} \cdot \avg{\widehat \mu \nabla v}
  -
  \frac{\sigma(\mu)}{h} \jump{v} \cdot \jump{\phi}
\end{equation}
with $\sigma(\mu) = 10\mu$ as the penalty parameter.

As in the previous tests we make use of a first order IMEX scheme for
the timestepping and Richtmyer fluxes for $\vec F$. We take the domain
as a rectangular region $[-1, 7.5]\times [-0.7, 0.7]$ with a circular
hole centered at the origin of radius $0.05$. In this case $\max h_K
\approx 0.14$ occuring near the upper and lower boundaries and $\min
h_K \approx 0.00008$ occuring near the hole. The initial data we use is
\begin{equation}
  \rho(\vec x, 0) = 1 \qquad \vec v(\vec x, 0) = \qp{1,0}^\transpose.
\end{equation}
We impose slip boundary conditions on the top and bottom of the
rectangular region, an inflow and outflow on the left and right hand
side respectively, compatible with the initial conditions and a no
slip condition on the cylinder itself.

The results are given in Figure \ref{fig:karman-vel} which shows the
vorticity, $\w = \partial_x v_2 - \partial_y v_1$, of the simulation
with model adaptivity and the location of where diffusion was switched
on. Figure \ref{fig:karman-err} gives an indication of the qualitative
difference between the adaptive approximation and the simulation of
the full INS system for $T=6.02$.

\begin{figure}[!ht]
  \caption[]
          {\label{fig:karman-vel}
            %%%%%%%%%%%%%%%%%%%%%%%%%%%%%%%%%%%%%%%%%%%%%%%%%%%%%%%%%%%%%%%%%
            A numerical experiment testing model adaptivity on the Isothermal Navier-Stokes system. The simulation is described in \S \ref{sec:INS}. Here we display the vorticity, $\w = \partial_x v_2 - \partial_y v_1$, of the solution at various times (top) together with a representation of both $\mu$ (bottom). Blue is the region $\mu=0$, where the simplified (Euler system) problem is being computed and red is where $\mu \neq 0$, where the full (Isothermal Navier-Stokes system) problem is being computed.
            %%%%%%%%%%%%%%%%%%%%%%%%%%%%%%%%%%%%%%%%%%%%%%%%%%%%%%%%%%%%%%%%%% 
          }
          \begin{center}
            \subfigure[{
                %%%%%%%%%%%%%%%%%%%%%%%%%%%%%%%%%%%%%%%%%%%%%%%%%%%%%%%%%%%%%%%
                $t=0.04$
                %%%%%%%%%%%%%%%%%%%%%%%%%%%%%%%%%%%%%%%%%%%%%%%%%%%%%%%%%%%%%%%
            }]{
              \includegraphics[scale=\figscale,width=0.47\figwidth]{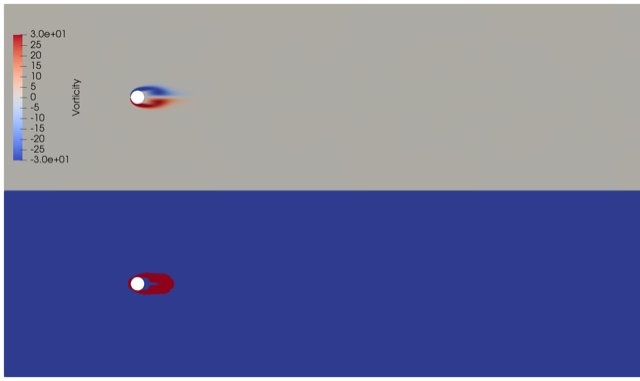}
            }
            %%%%%%%%%%%%%%%%%%%%%%%%%%%%%%%%%%%%%%%%%%%%%%%%%%%%%%%%%%%%%%%%%%%% 
            \hfill
            \subfigure[{
                %%%%%%%%%%%%%%%%%%%%%%%%%%%%%%%%%%%%%%%%%%%%%%%%%%%%%%%%%%%%%%%
                $t=0.37$
                %%%%%%%%%%%%%%%%%%%%%%%%%%%%%%%%%%%%%%%%%%%%%%%%%%%%%%%%%%%%%%% 
            }]{
              \includegraphics[scale=\figscale,width=0.47\figwidth]{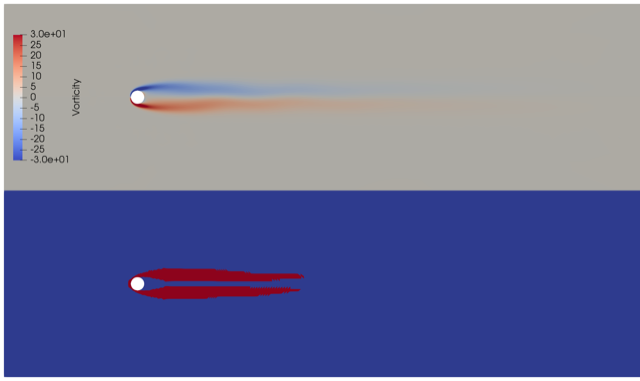}
            }
                        \subfigure[{
                %%%%%%%%%%%%%%%%%%%%%%%%%%%%%%%%%%%%%%%%%%%%%%%%%%%%%%%%%%%%%%%
                $t=0.94$
                %%%%%%%%%%%%%%%%%%%%%%%%%%%%%%%%%%%%%%%%%%%%%%%%%%%%%%%%%%%%%%%
            }]{
              \includegraphics[scale=\figscale,width=0.47\figwidth]{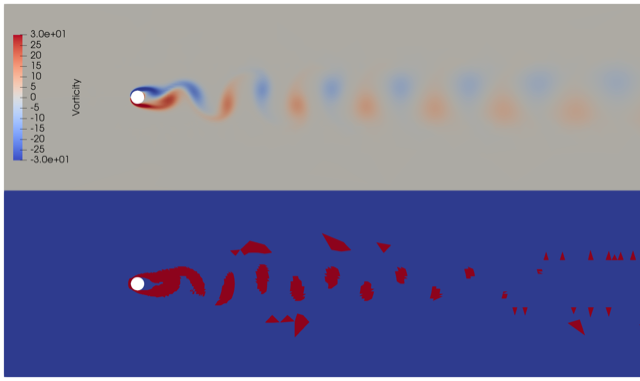}
            }
            %%%%%%%%%%%%%%%%%%%%%%%%%%%%%%%%%%%%%%%%%%%%%%%%%%%%%%%%%%%%%%%%%%%% 
            \hfill
            \subfigure[{
                %%%%%%%%%%%%%%%%%%%%%%%%%%%%%%%%%%%%%%%%%%%%%%%%%%%%%%%%%%%%%%%
                $t=1.5$
                %%%%%%%%%%%%%%%%%%%%%%%%%%%%%%%%%%%%%%%%%%%%%%%%%%%%%%%%%%%%%%% 
            }]{
              \includegraphics[scale=\figscale,width=0.47\figwidth]{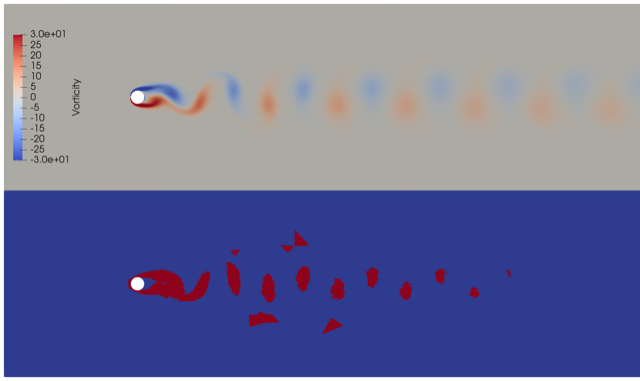}
            }
                       \subfigure[{
                %%%%%%%%%%%%%%%%%%%%%%%%%%%%%%%%%%%%%%%%%%%%%%%%%%%%%%%%%%%%%%%
                $t=2.46$
                %%%%%%%%%%%%%%%%%%%%%%%%%%%%%%%%%%%%%%%%%%%%%%%%%%%%%%%%%%%%%%%
            }]{
              \includegraphics[scale=\figscale,width=0.47\figwidth]{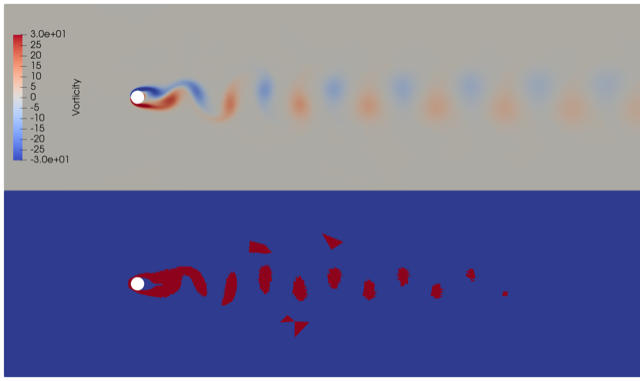}
            }
            %%%%%%%%%%%%%%%%%%%%%%%%%%%%%%%%%%%%%%%%%%%%%%%%%%%%%%%%%%%%%%%%%%%% 
            \hfill
            \subfigure[{
                %%%%%%%%%%%%%%%%%%%%%%%%%%%%%%%%%%%%%%%%%%%%%%%%%%%%%%%%%%%%%%%
                $t=6.02$
                %%%%%%%%%%%%%%%%%%%%%%%%%%%%%%%%%%%%%%%%%%%%%%%%%%%%%%%%%%%%%%% 
            }]{
              \includegraphics[scale=\figscale,width=0.47\figwidth]{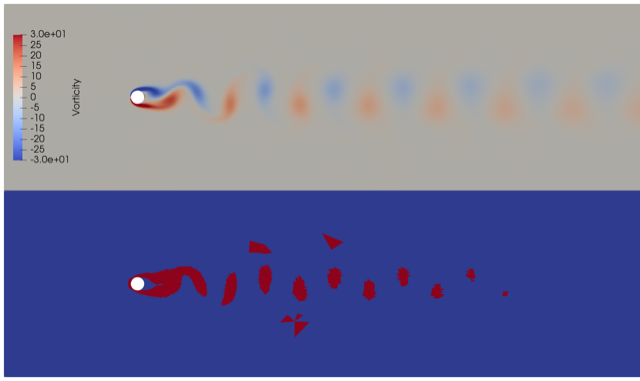}
            }
          \end{center}
\end{figure}

\begin{figure}[h!]
  \caption[]
          {\label{fig:karman-err}
            %%%%%%%%%%%%%%%%%%%%%%%%%%%%%%%%%%%%%%%%%%%%%%%%%%%%%%%%%%%%%%%%%
            A numerical experiment testing model adaptivity on the Isothermal Navier-Stokes system. The simulation is described in \S \ref{sec:INS}. Here we display the error induced through the model adaptive strategy. We plot $\norm{\vec w_h - \vec w_{\mu, h}}$, that is, the difference between the approximation of the full expensive model and that of the adaptive approximation. Note that $\Norm{\vec w_{\mu,h}}_{\leb{\infty}} \approx 1.7$ hence the relative error at final time is around 15\%.
            %%%%%%%%%%%%%%%%%%%%%%%%%%%%%%%%%%%%%%%%%%%%%%%%%%%%%%%%%%%%%%%%%% 
          }
          \begin{center}
            \subfigure[{
                %%%%%%%%%%%%%%%%%%%%%%%%%%%%%%%%%%%%%%%%%%%%%%%%%%%%%%%%%%%%%%%
                $t=0.04$
                %%%%%%%%%%%%%%%%%%%%%%%%%%%%%%%%%%%%%%%%%%%%%%%%%%%%%%%%%%%%%%%
            }]{
              \includegraphics[scale=\figscale,width=0.4\figwidth]{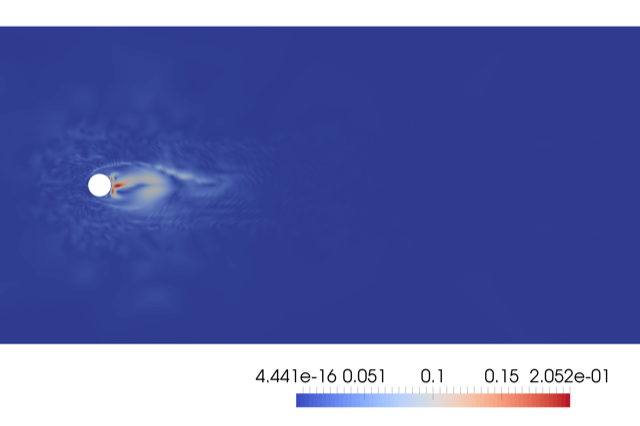}
            }
            %%%%%%%%%%%%%%%%%%%%%%%%%%%%%%%%%%%%%%%%%%%%%%%%%%%%%%%%%%%%%%%%%%%% 
            \hfill
            \subfigure[{
                %%%%%%%%%%%%%%%%%%%%%%%%%%%%%%%%%%%%%%%%%%%%%%%%%%%%%%%%%%%%%%%
                $t=6.02$
                %%%%%%%%%%%%%%%%%%%%%%%%%%%%%%%%%%%%%%%%%%%%%%%%%%%%%%%%%%%%%%% 
            }]{
              \includegraphics[scale=\figscale,width=0.4\figwidth]{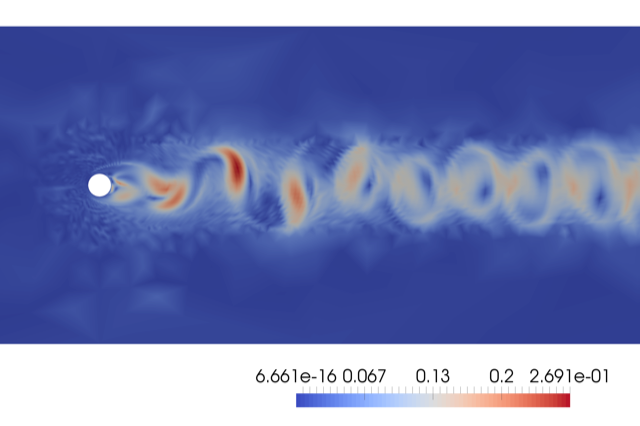}
            }
          \end{center}
\end{figure}

\subsection{Test 4 : The Navier-Stokes-Fourier system}
\label{sec:NSF}
In this test we detail the application of full spatial-model
adaptivity to the scenarios of classical forward facing step problem
as well as fluid flow around an aerofoil. We simulate the
Navier-Stokes-Fourier system which is given by seeking $\qp{\rho_\mu,
  \rho_\mu\bv_\mu, e_\mu}$ such that
\begin{equation}
  \label{eq:NSF3}
  \begin{split}
    \partial_t \rho_\mu + \div(\rho_\mu \bv_\mu) &=0\\
    \partial_t (\rho_\mu \bv_\mu) + \div(\rho_\mu \bv_\mu  \otimes \bv_\mu) + \nabla p &= \div (\mu \nabla \bv_\mu) \\
    \partial_t e_\mu + \div((e_\mu+ p)\bv_\mu )  &= \div (\mu(\nabla \bv_\mu) \cdot \bv_\mu + \mu \nabla T_\mu),
  \end{split}
\end{equation}
in a supersonic regime. In this case our ``complex'' model is given by (\ref{eq:NSF3}) and the approximation is given by
\begin{equation}
  \label{eq:EF}
  \begin{split}
    \partial_t \rho + \div(\rho \bv) &=0\\
    \partial_t (\rho \bv) + \div(\rho \bv  \otimes \bv) + \nabla p &= 0 \\
    \partial_t e + \div((e+ p)\bv )  &= 0.
  \end{split}
\end{equation}
The specific parameters are detailed in (\ref{eq:ideal}). As in Test 3
we can rewrite this as a system with $\vec w = \qp{\rho_\mu, \rho_\mu
  \bv_\mu, e_\mu}$ for the unknowns and (\ref{eq:NS}) becomes
\begin{equation}
  \pd t {\vec w} + \div(\vec f(\vec w)) = \vec D(\vec w),
\end{equation}
with
\begin{equation}
  \vec D(\vec w) =
  \begin{bmatrix}
    0
    \\
    \div\qp{\mu \nabla \bv_\mu}
    \\
    \div\qp{\mu \qp{\nabla \bv_\mu}\cdot \bv_\mu + \mu \nabla T_\mu}.
  \end{bmatrix}
\end{equation}
representing the diffusion term.
We choose a piecewise linear dG scheme similar to that in previous tests of the form
\begin{equation}\label{eq:nsfdg}
\begin{split}
  \int_{\cT} \partial_t \vec w_h \cdot \vec \phi -  \vec f(\vec w_h) \cdot \nabla \vec \phi + \int_{\cE} \vec F (\vec w_h^-, \vec w_h^+) \jump{\vec \phi}
  &=
  \cD\qp{{\vec w_h,\vec \phi; \widehat \mu }}
  \quad \text{for all } \vec \phi \in \fes_q,
  \\
  u_h(0) &= \cP_q [u_0]
\end{split}
\end{equation}
where, in $2d$,
\begin{equation}
  \vec f(\vec w)
  =
  \begin{bmatrix}
    \rho v_1 & \rho v_2
    \\
    \rho v_1^2 + p & \rho v_1 v_2
    \\
    \rho v_2 v_1 & \rho v_2^2 + p
    \\
    (e + p) v_1 & (e + p) v_2
  \end{bmatrix},
\end{equation}
\begin{equation}
  \cD\qp{\vec w_h, \vec \phi; \widehat \mu}
  =
  \begin{pmatrix}
    0
    \\
    -\bih{v_1}{\phi_2; \widehat \mu}
    \\
    -\bih{v_2}{\phi_3; \widehat \mu}
    \\
    -\bih{v_1}{\phi_4; v_1 \widehat \mu}
    -\bih{v_2}{\phi_4; v_2 \widehat \mu}
    -\bih{T}{\phi_4; \widehat \mu}
  \end{pmatrix}.
\end{equation}

We make use of a 3rd order strong stability preserving Runge-Kutta
IMEX scheme for the temporal discretisation, a piecewise linear dG
scheme with the Richtmyer fluxes (\ref{eq:Richtmyer}) and an IP
discretisation of both dissipative terms. Note that in these
experiments we do not compute the difference between the adaptive
approximation and the ``complex model''. The reason is that a fair and
accurate comparison would require the simulation of the full INS
system on a uniform mesh of mesh width matching the smallest mesh
width in the adaptive approximation. This is computationally
unfeasible.

The algorithm we used to simulate this problem is given in pseudocode
below. This is a maximum type strategy where elements with the largest
local error in space, determined by the indicator $\cE_D$ are
refined. In addition we solve the full NSF system over those elements
with largest modelling error, determined by the indictor $\cE_M$. 

\subsection{{$\Algoname{Spatial-Model Adaptation}$}}
\label{alg:spatial-model-adapt}
\begin{algorithmic}
  \Require
  $(\tau,t_0,T,\vec w^0, \tol, \tol_c,\mu)$

  \Ensure $(\vec w_h^n)_{\rangefromto n1N}$, spatial-model adaptive solution

  \State $\widehat \mu(x,0):=0$
  \Comment{Initialise parameters}

  \State $t = t_0 + \tau, n=1$

  \While{$t\leq T$}
  \Comment{Loop in time}

  \State
  $(\vec w_h^n) := \Algoname{Solve one timestep of dg scheme} (\vec w_h^{n-1},\widehat \mu)$
  \State
  Compute $\cE_D \AND \cE_M$  
  \For{$K\in \T{}$}
  
  \If{$\cE_M|_K < \tol_c \max_{J\in\T{}} \cE_M|_J$}
  \Comment{Model ``coarsening'' strategy}  
  \State $\widehat \mu(x,t)|_K = 0$
  \EndIf

  \If{$\cE_D|_K < \tol_c \max_{J\in\T{}} \cE_D|_J$}
  \Comment{Mesh coarsening strategy}  
  \State Mark $K$ for coarsening
  \EndIf

  \EndFor

  \State Coarsen marked elements
  
  \State Recompute $\cE_D \AND \cE_M$  
  
  \For{$K\in \T{}$}  

  \If{$\cE_M|_K \geq \tol \max_{J\in\T{}} \cE_M|_J$}
  \Comment{Model ``refinement'' strategy}  
  \State $\widehat \mu(x,t)|_K = \mu$
  \EndIf
  
  \If{$\cE_D|_K \geq \tol \max_{J\in\T{}} \cE_D|_J$}
  \Comment{Mesh refinement strategy}  
  \State Mark $K$ for refinement
  \EndIf

  \EndFor

  \State Refine marked elements
  
  \State $t := t + \tau, n := n+1$
    
  \EndWhile
  
  \State return $(\vec w_h^n)_{\rangefromto n1N}$,
\end{algorithmic}

\begin{Rem}[Comparing the two model adaptive strategies]
  The strategy proposed in Algorithm \ref{alg:spatial-model-adapt} has
  a distinct disadvantage over Algorithm \ref{alg:model-adapt} since
  the ``complex'' model will be solved on a portion of the domain from
  the very beginning of the simulation. Consider, for example, Tests 1
  and 2 where the solution remains smooth for some time. Applying
  Algorithm \ref{alg:spatial-model-adapt} to these problems results in
  the model adaptivity being switched on before it is really
  needed. For these supersonic NSF simulations the solution is
  nonsmooth from the very beginning of the simulation and linking
  model adaptivity with mesh refinement is done quite simply with the
  implementation of Algorithm \ref{alg:spatial-model-adapt} with
  successful results as can be seen from Figures \ref{fig:step} and
  \ref{fig:naca}.

  Note also that a feature of both strategies, due to the form of the
  estimator $\cE_M$, is that in the scalar case the ``model
  coarsening'' step results in \emph{all} cells being coarsened and
  then new cells selected for ``model refinement'' immediately
  afterwards. This ensures that the complex model is only used on
  cells where it is truly needed.
\end{Rem}

\subsection*{Forward step parameters} The numerical domain we use for the forward step problem is the $[0,0] \times [3/2, 3/4] \not \ \ [1/2, 0] \times [3/2, 1/4]$. We select $\mu = 10^{-6}$. We produce a uniform initial mesh with $h = 0.01$ and allow a maximum refinement level of $2$. This means the minimum meshsize through the whole simulation is $h = 0.0025$. We take a uniform timestep of $\tau \approx 0.00062$. Adaptation is described with Algorithm \ref{alg:spatial-model-adapt}. We take initial conditions
\begin{equation}
    \rho(\vec x, 0) = 1.4 \qquad \bv(\vec x, 0) = \qp{3,
      0}^\transpose \qquad e(\vec x, 0) = 8
\end{equation}
and enforce boundary conditions that match the initial conditions on
the left boundary. We enforce slip boundary conditions on the top and
bottom boundaries and an outflow boundary condition on the right
boundary. With the speed of sound and Mach number defined as
\begin{equation}
  \label{speed}
  a = \sqrt{\norm{\frac{\gamma p}{\rho}}} \qquad M = \frac{\norm{\bv}}{a},
\end{equation}
respectively, we compute that for this experiment $a \approx 0.82$ and
hence $M \approx 3.66$.

We use Algorithm \ref{alg:spatial-model-adapt} with $\tol =
0.5$ and $\tol_c = 0.1$. Figure \ref{fig:step} gives some snapshots of
the solution, the model adaptivity parameter and the underlying mesh
at various times.

\subsection*{Aerofoil parameters}
We used data from the Langely Research Center website to construct a
NACA 0012 aerofoil in a box $[-1/2, -2] \times [3,2]$. The aerofoil
leading edge is situated at the origin and tail is at $[1,0]$. We
select $\mu = 10^{-6}$. We produce a quasiuniform initial mesh with $h
\approx 0.1$ and allow a maximum refinement level of $4$. This means
the minimum meshsize throughout the whole simulation is $h\approx
0.00625$. Note that due to the shape of the aerofoil the mesh is not
Cartesian. This means the reconstruction proposed in
\S\ref{subs:rec:2d} is not valid. To compute an approximation to this
reconstruction we overlay a Cartesian mesh and take the nodal values
as an average of the element (or elements) containing this node. We
take a uniform timestep of $\tau \approx 0.0018$. Adaptation is
described with Algorithm \ref{alg:model-adapt}. We take initial
conditions
\begin{equation}
  \rho(\vec x, 0) = 1 \qquad \bv(\vec x, 0) = \qp{3/2,
    0}^\transpose \qquad e(\vec x, 0) = 5
\end{equation}
and enforce boundary conditions that match the initial conditions on
the left boundary. We enforce slip boundary conditions on the top and
bottom boundaries as well as on the aerofoil and an outflow boundary
condition on the right boundary. Using (\ref{speed}) we compute
$a\approx 1.47$ and $M \approx 1.02$.

We use Algorithm
\ref{alg:spatial-model-adapt} with $\tol = 0.5$ and $\tol_c =
0.1$. Figure \ref{fig:naca} gives some snapshots of the solution, the
model adaptivity parameter and the underlying mesh at various times.

\begin{Rem}[Coupling of viscosity and heat conduction]
  In our experiment we choose $\mu$ as both the coefficient of
  viscosity and the heat conduction. In practical situations these
  parameters may scale differently and by splitting the adaptation
  estimator the model adaptivity can be conducted independently for
  both the viscous and heat conduction term. We will not persue this
  here.
\end{Rem}

\begin{figure}[!ht]
  \caption[]{\label{fig:step}
            %%%%%%%%%%%%%%%%%%%%%%%%%%%%%%%%%%%%%%%%%%%%%%%%%%%%%%%%%%%%%%%%%
    An experiment using spatial-model adaptivity on the forward facing
    step problem. We plot the density and the model adaptivity
    parameter together with the underlying mesh. Notice that the mesh
    correctly adapts as the shocks form and propogate and that model
    adaptivity is on when the solution is close to shocking and near
    the step on the boundary, where the solution is most singular.
            %%%%%%%%%%%%%%%%%%%%%%%%%%%%%%%%%%%%%%%%%%%%%%%%%%%%%%%%%%%%%%%%%% 
  }
  \subfigure[{
      %%%%%%%%%%%%%%%%%%%%%%%%%%%%%%%%%%%%%%%%%%%%%%%%%%%%%%%%%%%%%%%
      Density at $t=0.0774$.
      %%%%%%%%%%%%%%%%%%%%%%%%%%%%%%%%%%%%%%%%%%%%%%%%%%%%%%%%%%%%%%% 
  }]{
    \includegraphics[scale=\figscale,width=0.3\figwidth]{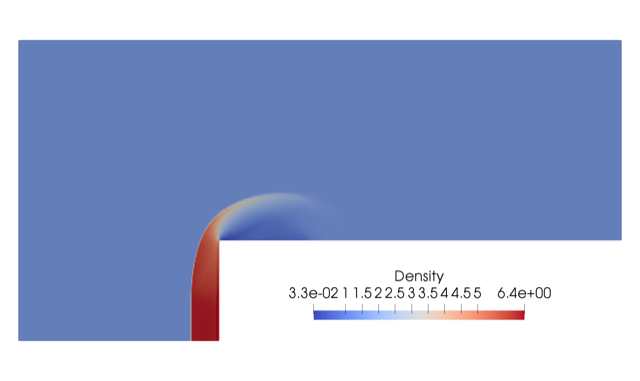}
  }
  \hfill
  \subfigure[{
      %%%%%%%%%%%%%%%%%%%%%%%%%%%%%%%%%%%%%%%%%%%%%%%%%%%%%%%%%%%%%%%
      $\varepsilon$   at $t=0.0774$.
      %%%%%%%%%%%%%%%%%%%%%%%%%%%%%%%%%%%%%%%%%%%%%%%%%%%%%%%%%%%%%%% 
  }]{
    \includegraphics[scale=\figscale,width=0.3\figwidth]{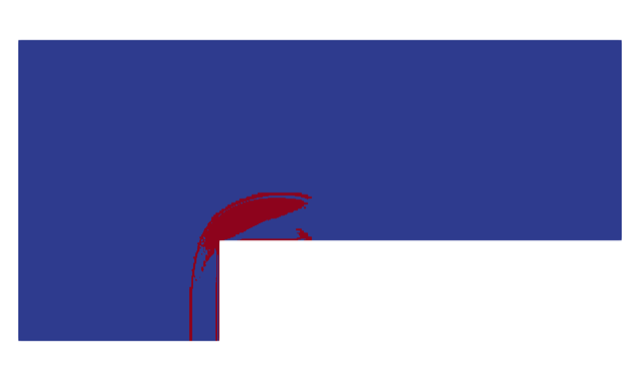}
  }
  \hfill
  \subfigure[{
      %%%%%%%%%%%%%%%%%%%%%%%%%%%%%%%%%%%%%%%%%%%%%%%%%%%%%%%%%%%%%%%
      Mesh at $t=0.0774$.
      %%%%%%%%%%%%%%%%%%%%%%%%%%%%%%%%%%%%%%%%%%%%%%%%%%%%%%%%%%%%%%% 
  }]{
    \includegraphics[scale=\figscale,width=0.3\figwidth]{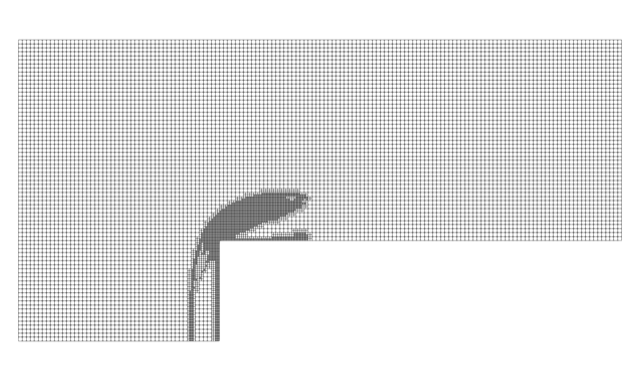}
  }
  \hfill
  \subfigure[{
      %%%%%%%%%%%%%%%%%%%%%%%%%%%%%%%%%%%%%%%%%%%%%%%%%%%%%%%%%%%%%%%
      Density at $t=0.2477$.
      %%%%%%%%%%%%%%%%%%%%%%%%%%%%%%%%%%%%%%%%%%%%%%%%%%%%%%%%%%%%%%% 
  }]{
    \includegraphics[scale=\figscale,width=0.3\figwidth]{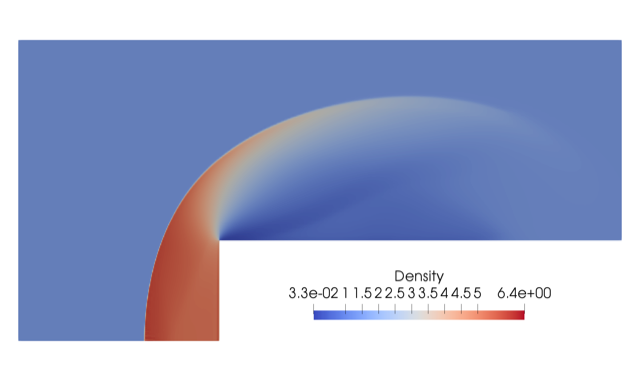}
  }
  \hfill
  \subfigure[{
      %%%%%%%%%%%%%%%%%%%%%%%%%%%%%%%%%%%%%%%%%%%%%%%%%%%%%%%%%%%%%%%
      $\varepsilon$   at $t=0.2477$.
      %%%%%%%%%%%%%%%%%%%%%%%%%%%%%%%%%%%%%%%%%%%%%%%%%%%%%%%%%%%%%%% 
  }]{
    \includegraphics[scale=\figscale,width=0.3\figwidth]{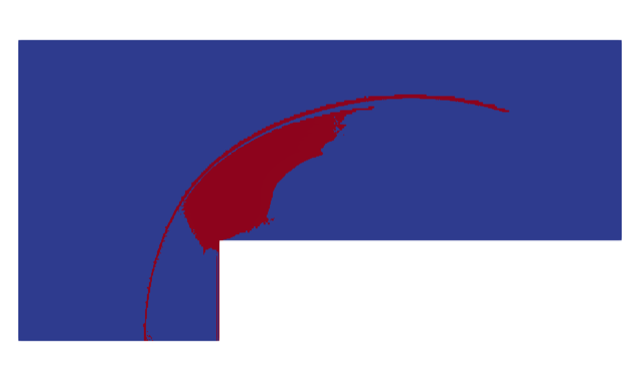}
  }
  \hfill
  \subfigure[{
      %%%%%%%%%%%%%%%%%%%%%%%%%%%%%%%%%%%%%%%%%%%%%%%%%%%%%%%%%%%%%%%
      Mesh at $t=0.2477$.
      %%%%%%%%%%%%%%%%%%%%%%%%%%%%%%%%%%%%%%%%%%%%%%%%%%%%%%%%%%%%%%% 
  }]{
    \includegraphics[scale=\figscale,width=0.3\figwidth]{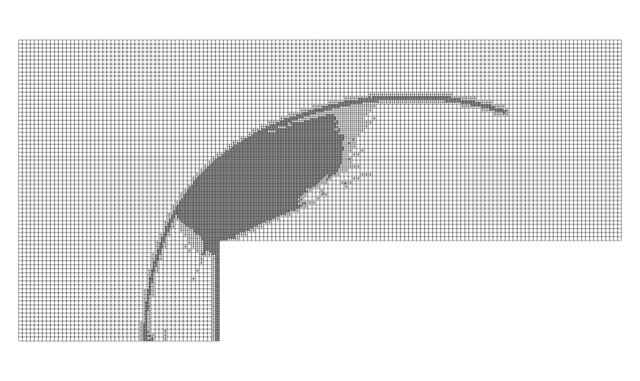}
  }
  \hfill
  \subfigure[{
      %%%%%%%%%%%%%%%%%%%%%%%%%%%%%%%%%%%%%%%%%%%%%%%%%%%%%%%%%%%%%%%
      Density at $t=0.3715$.
      %%%%%%%%%%%%%%%%%%%%%%%%%%%%%%%%%%%%%%%%%%%%%%%%%%%%%%%%%%%%%%% 
  }]{
    \includegraphics[scale=\figscale,width=0.3\figwidth]{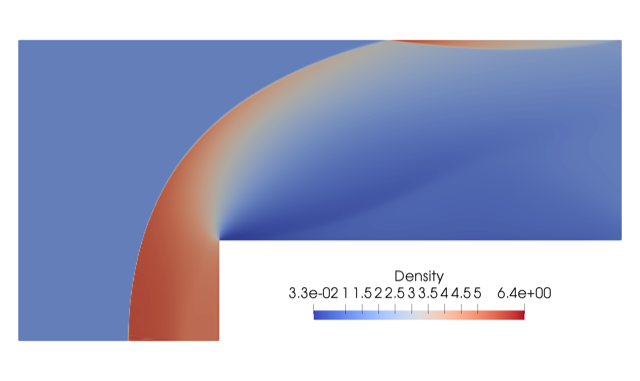}
  }
  \hfill
  \subfigure[{
      %%%%%%%%%%%%%%%%%%%%%%%%%%%%%%%%%%%%%%%%%%%%%%%%%%%%%%%%%%%%%%%
      $\varepsilon$   at $t=0.3715$.
      %%%%%%%%%%%%%%%%%%%%%%%%%%%%%%%%%%%%%%%%%%%%%%%%%%%%%%%%%%%%%%% 
  }]{
    \includegraphics[scale=\figscale,width=0.3\figwidth]{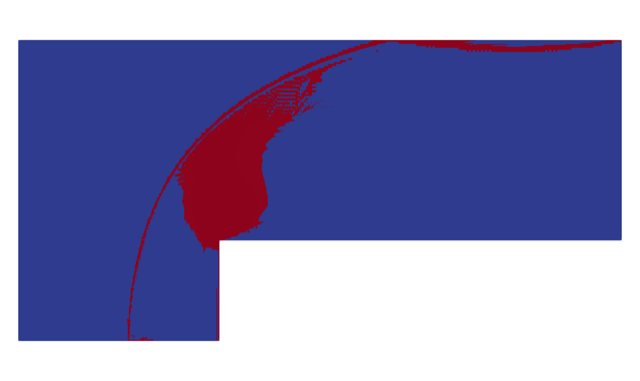}
  }
  \hfill
  \subfigure[{
      %%%%%%%%%%%%%%%%%%%%%%%%%%%%%%%%%%%%%%%%%%%%%%%%%%%%%%%%%%%%%%%
      Mesh at $t=0.3715$.
      %%%%%%%%%%%%%%%%%%%%%%%%%%%%%%%%%%%%%%%%%%%%%%%%%%%%%%%%%%%%%%% 
  }]{
    \includegraphics[scale=\figscale,width=0.3\figwidth]{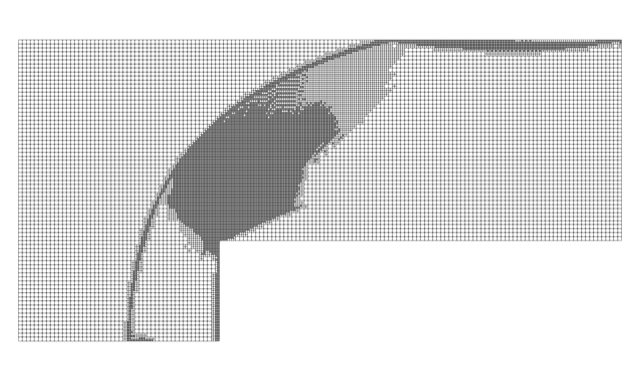}
  }
  \hfill
  \subfigure[{
      %%%%%%%%%%%%%%%%%%%%%%%%%%%%%%%%%%%%%%%%%%%%%%%%%%%%%%%%%%%%%%%
      Density at $t=0.743$.
      %%%%%%%%%%%%%%%%%%%%%%%%%%%%%%%%%%%%%%%%%%%%%%%%%%%%%%%%%%%%%%% 
  }]{
    \includegraphics[scale=\figscale,width=0.3\figwidth]{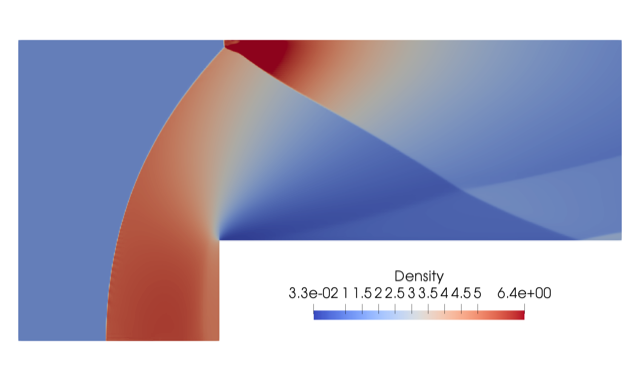}
  }
  \hfill
  \subfigure[{
      %%%%%%%%%%%%%%%%%%%%%%%%%%%%%%%%%%%%%%%%%%%%%%%%%%%%%%%%%%%%%%%
      $\varepsilon$   at $t=0.743$.
      %%%%%%%%%%%%%%%%%%%%%%%%%%%%%%%%%%%%%%%%%%%%%%%%%%%%%%%%%%%%%%% 
  }]{
    \includegraphics[scale=\figscale,width=0.3\figwidth]{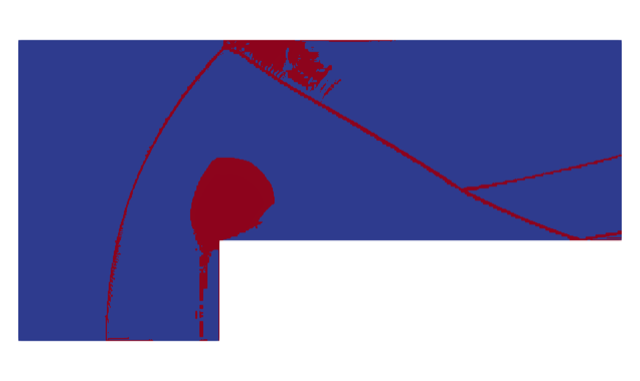}
  }
  \hfill
  \subfigure[{
      %%%%%%%%%%%%%%%%%%%%%%%%%%%%%%%%%%%%%%%%%%%%%%%%%%%%%%%%%%%%%%%
      Mesh at $t=0.743$.
      %%%%%%%%%%%%%%%%%%%%%%%%%%%%%%%%%%%%%%%%%%%%%%%%%%%%%%%%%%%%%%% 
  }]{
    \includegraphics[scale=\figscale,width=0.3\figwidth]{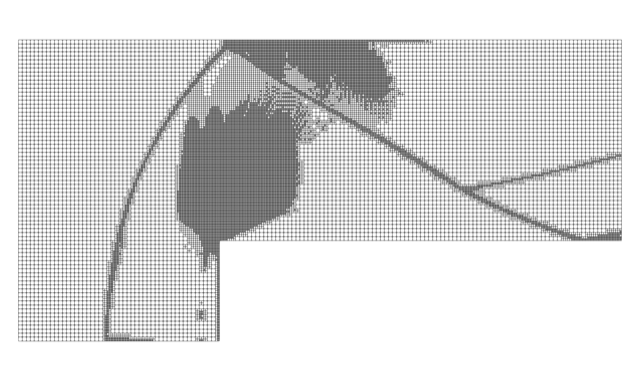}
  }
  \hfill
  \subfigure[{
      %%%%%%%%%%%%%%%%%%%%%%%%%%%%%%%%%%%%%%%%%%%%%%%%%%%%%%%%%%%%%%%
     Density at $t=1.2384$.
      %%%%%%%%%%%%%%%%%%%%%%%%%%%%%%%%%%%%%%%%%%%%%%%%%%%%%%%%%%%%%%% 
  }]{
    \includegraphics[scale=\figscale,width=0.3\figwidth]{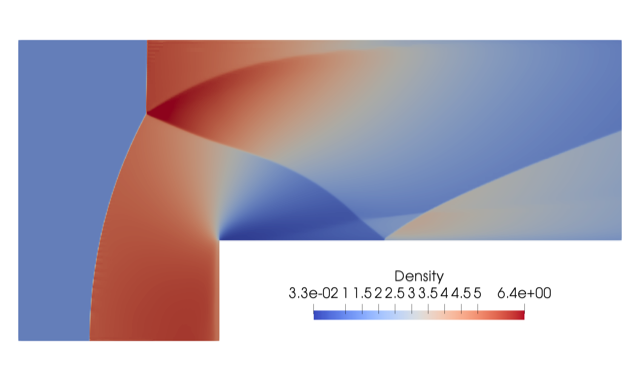}
  }
  \hfill
  \subfigure[{
      %%%%%%%%%%%%%%%%%%%%%%%%%%%%%%%%%%%%%%%%%%%%%%%%%%%%%%%%%%%%%%%
      $\varepsilon$   at $t=1.2384$.
      %%%%%%%%%%%%%%%%%%%%%%%%%%%%%%%%%%%%%%%%%%%%%%%%%%%%%%%%%%%%%%% 
  }]{
    \includegraphics[scale=\figscale,width=0.3\figwidth]{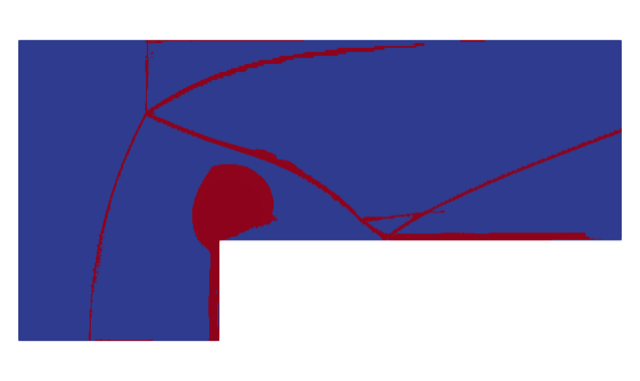}
  }
  \hfill
  \subfigure[{
      %%%%%%%%%%%%%%%%%%%%%%%%%%%%%%%%%%%%%%%%%%%%%%%%%%%%%%%%%%%%%%%
      Mesh at $t=1.2384$.
      %%%%%%%%%%%%%%%%%%%%%%%%%%%%%%%%%%%%%%%%%%%%%%%%%%%%%%%%%%%%%%% 
  }]{
    \includegraphics[scale=\figscale,width=0.3\figwidth]{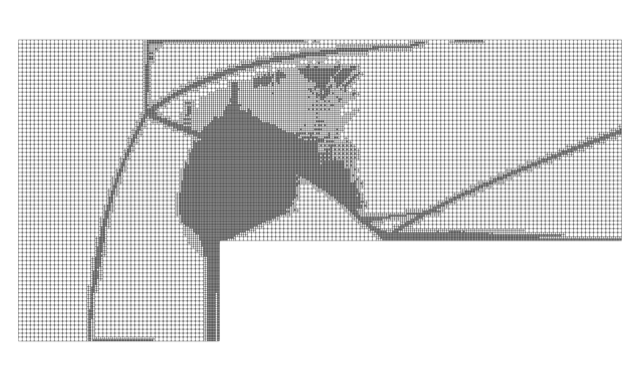}
  }
  \hfill
  \subfigure[{
      %%%%%%%%%%%%%%%%%%%%%%%%%%%%%%%%%%%%%%%%%%%%%%%%%%%%%%%%%%%%%%%
      Density at $t=1.8576$.
      %%%%%%%%%%%%%%%%%%%%%%%%%%%%%%%%%%%%%%%%%%%%%%%%%%%%%%%%%%%%%%% 
  }]{
    \includegraphics[scale=\figscale,width=0.3\figwidth]{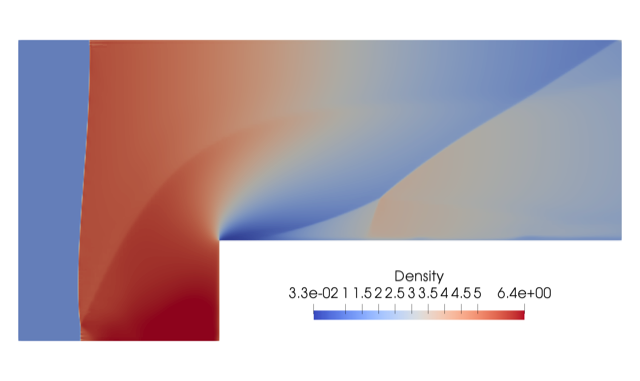}
  }
  \hfill
  \subfigure[{
      %%%%%%%%%%%%%%%%%%%%%%%%%%%%%%%%%%%%%%%%%%%%%%%%%%%%%%%%%%%%%%%
      $\varepsilon$   at $t=1.8576$.
      %%%%%%%%%%%%%%%%%%%%%%%%%%%%%%%%%%%%%%%%%%%%%%%%%%%%%%%%%%%%%%% 
  }]{
    \includegraphics[scale=\figscale,width=0.3\figwidth]{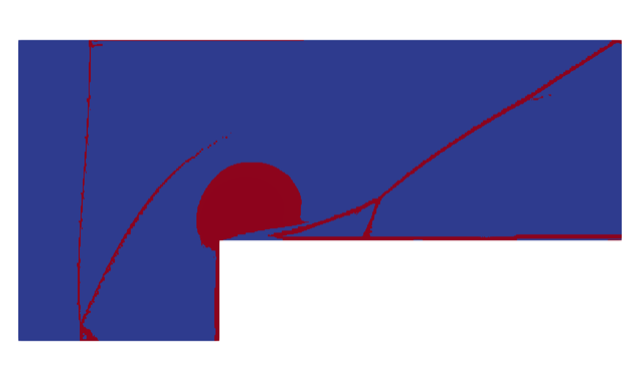}
  }
  \hfill
  \subfigure[{
      %%%%%%%%%%%%%%%%%%%%%%%%%%%%%%%%%%%%%%%%%%%%%%%%%%%%%%%%%%%%%%%
      Mesh at $t=1.8576$.
      %%%%%%%%%%%%%%%%%%%%%%%%%%%%%%%%%%%%%%%%%%%%%%%%%%%%%%%%%%%%%%% 
  }]{
    \includegraphics[scale=\figscale,width=0.3\figwidth]{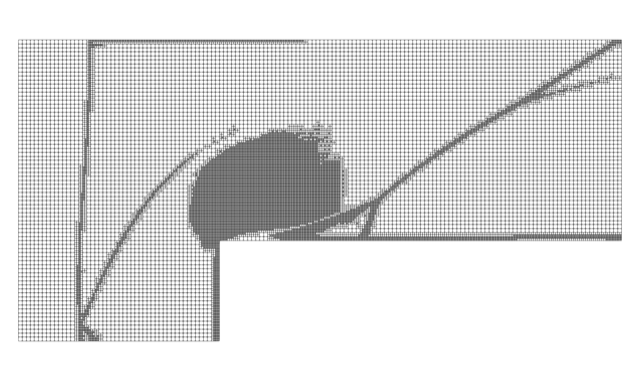}
  }
\end{figure}

\begin{figure}[!ht]
  \caption[]{\label{fig:naca}
            %%%%%%%%%%%%%%%%%%%%%%%%%%%%%%%%%%%%%%%%%%%%%%%%%%%%%%%%%%%%%%%%%
    An experiment using spatial-model adaptivity for flow around an
    aerofoil. We plot the density and the model adaptivity parameter
    together with the underlying mesh. Notice that the mesh correctly
    adapts as the shocks form and propogate and that model adaptivity
    is on when the solution is close to shocking.
    %%%%%%%%%%%%%%%%%%%%%%%%%%%%%%%%%%%%%%%%%%%%%%%%%%%%%%%%%%%%%%%%%% 
  }
  \subfigure[{
      %%%%%%%%%%%%%%%%%%%%%%%%%%%%%%%%%%%%%%%%%%%%%%%%%%%%%%%%%%%%%%%
      Density at $t=0.0363$.
      %%%%%%%%%%%%%%%%%%%%%%%%%%%%%%%%%%%%%%%%%%%%%%%%%%%%%%%%%%%%%%% 
  }]{
    \includegraphics[height=0.3\figwidth,width=0.3\figwidth]{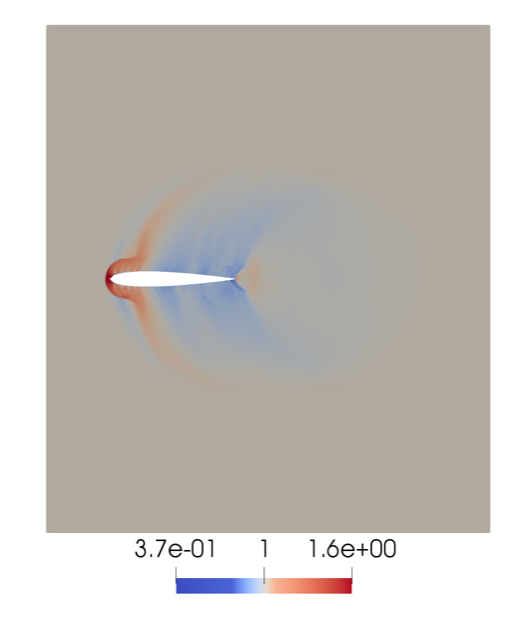}
  }
  \hfill
  \subfigure[{
      %%%%%%%%%%%%%%%%%%%%%%%%%%%%%%%%%%%%%%%%%%%%%%%%%%%%%%%%%%%%%%%
      $\varepsilon$   at $t=0.0363$.
      %%%%%%%%%%%%%%%%%%%%%%%%%%%%%%%%%%%%%%%%%%%%%%%%%%%%%%%%%%%%%%% 
  }]{
    \includegraphics[height=0.3\figwidth,width=0.3\figwidth]{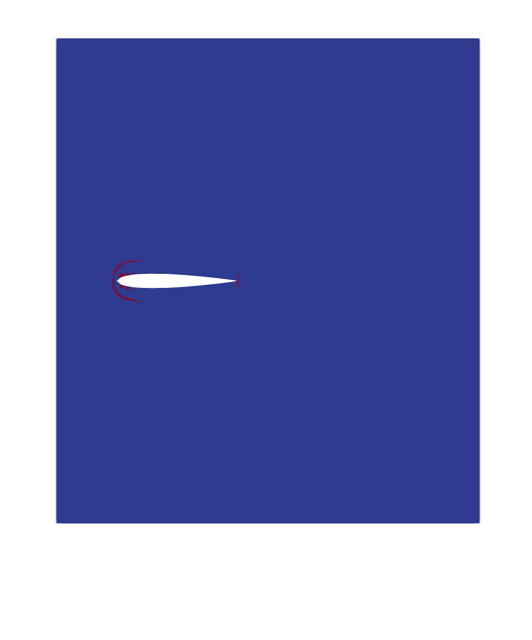}
  }
  \hfill
  \subfigure[{
      %%%%%%%%%%%%%%%%%%%%%%%%%%%%%%%%%%%%%%%%%%%%%%%%%%%%%%%%%%%%%%%
      Mesh at $t=0.0363$.
      %%%%%%%%%%%%%%%%%%%%%%%%%%%%%%%%%%%%%%%%%%%%%%%%%%%%%%%%%%%%%%% 
  }]{
    \includegraphics[height=0.3\figwidth,width=0.3\figwidth]{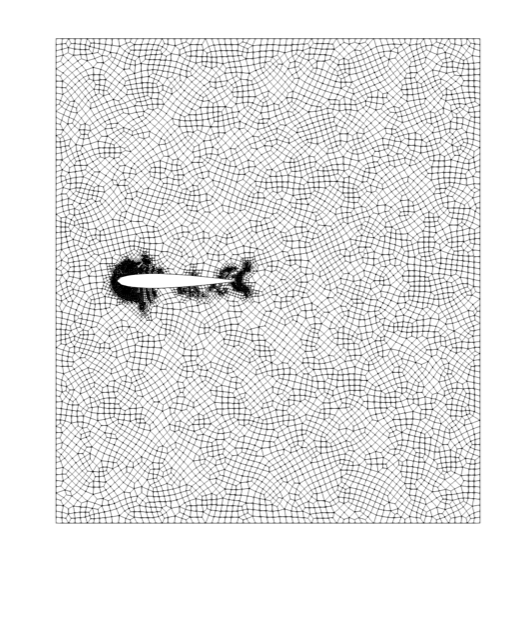}
  }
  \hfill
  \subfigure[{
      %%%%%%%%%%%%%%%%%%%%%%%%%%%%%%%%%%%%%%%%%%%%%%%%%%%%%%%%%%%%%%%
      Density at $t=0.127$.
      %%%%%%%%%%%%%%%%%%%%%%%%%%%%%%%%%%%%%%%%%%%%%%%%%%%%%%%%%%%%%%% 
  }]{
    \includegraphics[height=0.3\figwidth,width=0.3\figwidth]{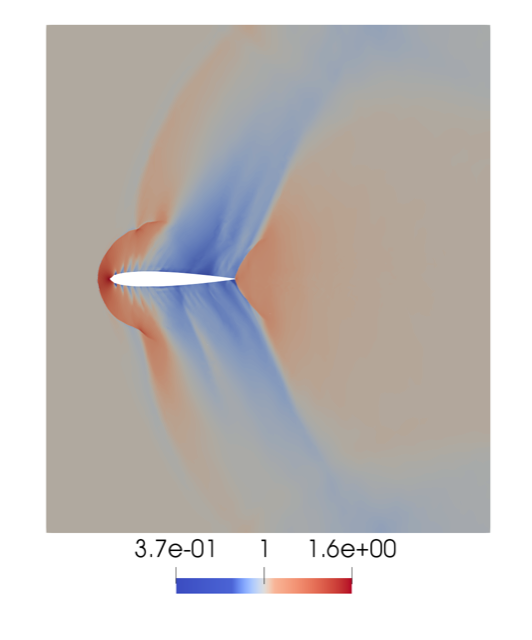}
  }
  \hfill
  \subfigure[{
      %%%%%%%%%%%%%%%%%%%%%%%%%%%%%%%%%%%%%%%%%%%%%%%%%%%%%%%%%%%%%%%
      $\varepsilon$   at $t=0.127$.
      %%%%%%%%%%%%%%%%%%%%%%%%%%%%%%%%%%%%%%%%%%%%%%%%%%%%%%%%%%%%%%% 
  }]{
    \includegraphics[height=0.3\figwidth,width=0.3\figwidth]{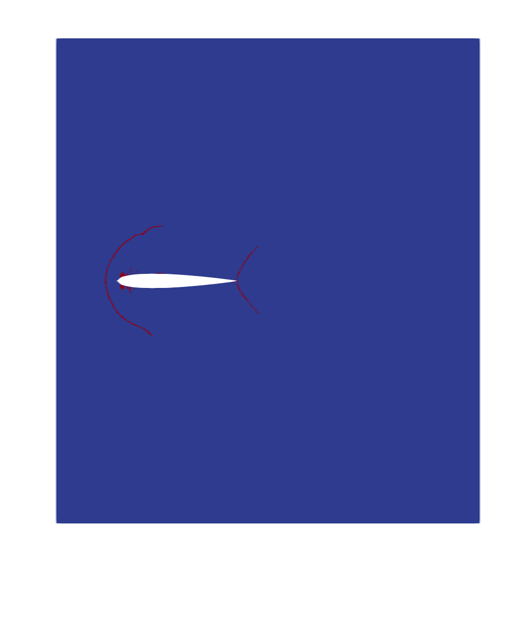}
  }
  \hfill
  \subfigure[{
      %%%%%%%%%%%%%%%%%%%%%%%%%%%%%%%%%%%%%%%%%%%%%%%%%%%%%%%%%%%%%%%
      Mesh at $t=0.127$.
      %%%%%%%%%%%%%%%%%%%%%%%%%%%%%%%%%%%%%%%%%%%%%%%%%%%%%%%%%%%%%%% 
  }]{
    \includegraphics[height=0.3\figwidth,width=0.3\figwidth]{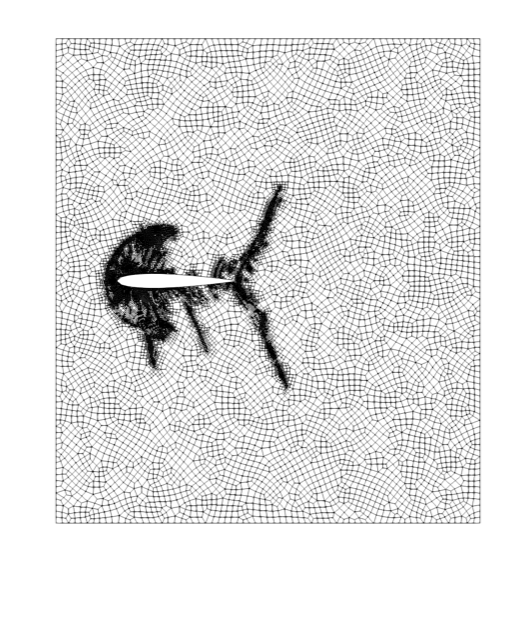}
  }
  \hfill
  \subfigure[{
      %%%%%%%%%%%%%%%%%%%%%%%%%%%%%%%%%%%%%%%%%%%%%%%%%%%%%%%%%%%%%%%
      Density at $t=0.363$.
      %%%%%%%%%%%%%%%%%%%%%%%%%%%%%%%%%%%%%%%%%%%%%%%%%%%%%%%%%%%%%%% 
  }]{
    \includegraphics[height=0.3\figwidth,width=0.3\figwidth]{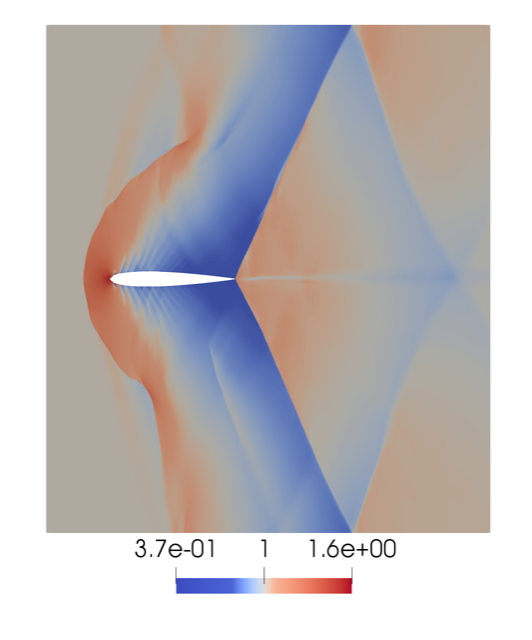}
  }
  \hfill
  \subfigure[{
      %%%%%%%%%%%%%%%%%%%%%%%%%%%%%%%%%%%%%%%%%%%%%%%%%%%%%%%%%%%%%%%
      $\varepsilon$   at $t=0.363$.
      %%%%%%%%%%%%%%%%%%%%%%%%%%%%%%%%%%%%%%%%%%%%%%%%%%%%%%%%%%%%%%% 
  }]{
    \includegraphics[height=0.3\figwidth,width=0.3\figwidth]{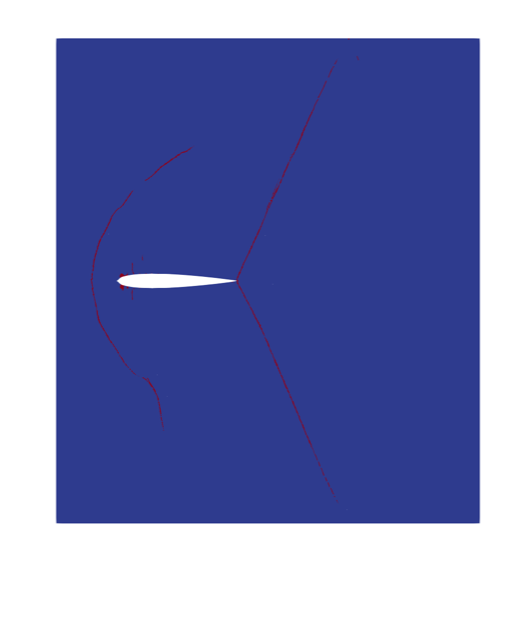}
  }
  \hfill
  \subfigure[{
      %%%%%%%%%%%%%%%%%%%%%%%%%%%%%%%%%%%%%%%%%%%%%%%%%%%%%%%%%%%%%%%
      Mesh at $t=0.363$.
      %%%%%%%%%%%%%%%%%%%%%%%%%%%%%%%%%%%%%%%%%%%%%%%%%%%%%%%%%%%%%%% 
  }]{
    \includegraphics[height=0.3\figwidth,width=0.3\figwidth]{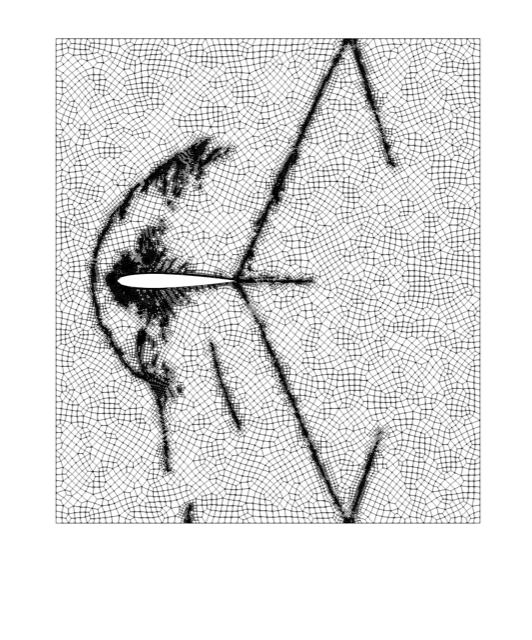}
  }
  \hfill
  \subfigure[{
      %%%%%%%%%%%%%%%%%%%%%%%%%%%%%%%%%%%%%%%%%%%%%%%%%%%%%%%%%%%%%%%
     Density at $t=0.9074$.
      %%%%%%%%%%%%%%%%%%%%%%%%%%%%%%%%%%%%%%%%%%%%%%%%%%%%%%%%%%%%%%% 
  }]{
    \includegraphics[height=0.3\figwidth,width=0.3\figwidth]{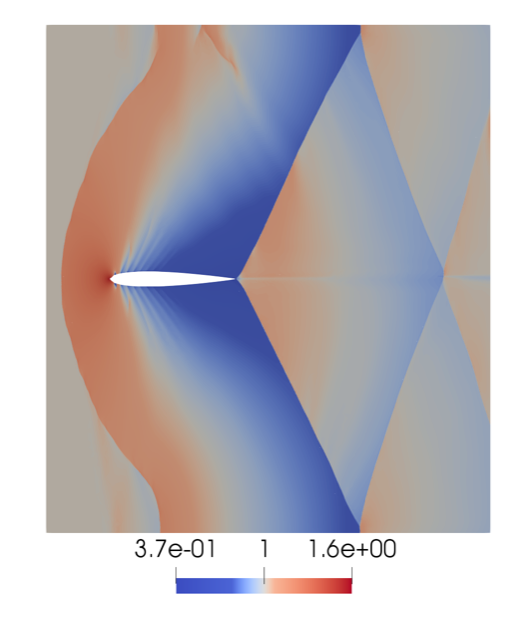}
  }
  \hfill
  \subfigure[{
      %%%%%%%%%%%%%%%%%%%%%%%%%%%%%%%%%%%%%%%%%%%%%%%%%%%%%%%%%%%%%%%
      $\varepsilon$   at $t=0.9074$.
      %%%%%%%%%%%%%%%%%%%%%%%%%%%%%%%%%%%%%%%%%%%%%%%%%%%%%%%%%%%%%%% 
  }]{
    \includegraphics[height=0.3\figwidth,width=0.3\figwidth]{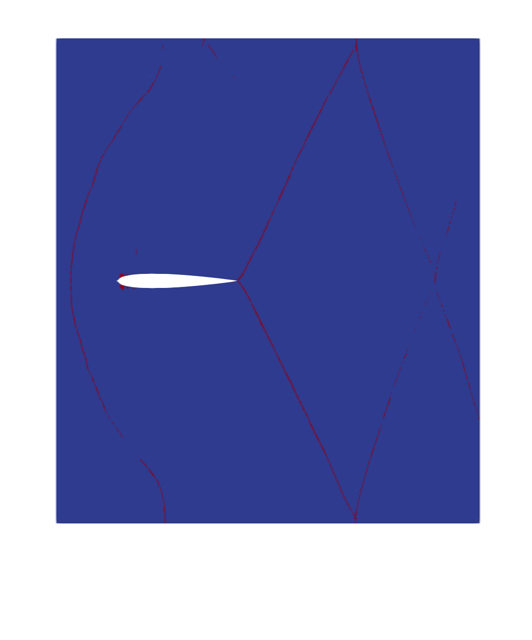}
  }
  \hfill
  \subfigure[{
      %%%%%%%%%%%%%%%%%%%%%%%%%%%%%%%%%%%%%%%%%%%%%%%%%%%%%%%%%%%%%%%
      Mesh at $t=0.9074$.
      %%%%%%%%%%%%%%%%%%%%%%%%%%%%%%%%%%%%%%%%%%%%%%%%%%%%%%%%%%%%%%% 
  }]{
    \includegraphics[height=0.3\figwidth,width=0.3\figwidth]{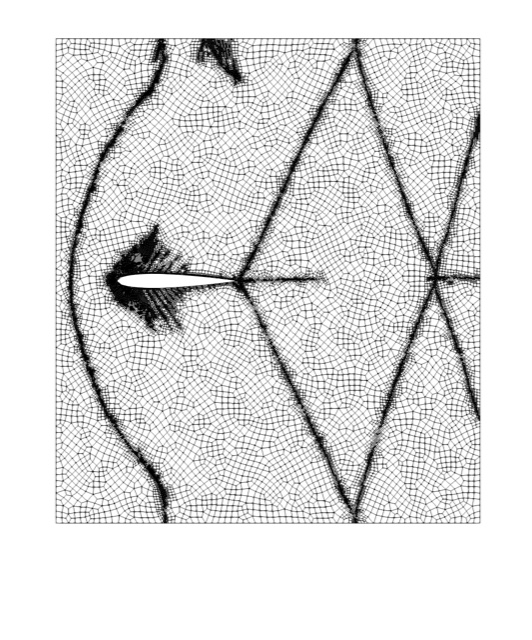}
  }
\end{figure}

%%%%%%%%%%%%%%%%%%%%%%%%%%%%%%%%%%%%%%%%%%%%%%%%%%%%%%%%%%%%%%%%%%%%%%%%
%% Bibliography
 \bibliographystyle{alpha}
 \bibliography{nskbib}

\newcommand{\etalchar}[1]{$^{#1}$}
\def\cprime{$'$}
\begin{thebibliography}{vOBPvB15}

\bibitem[AGQ06]{AGQ06}
Valery Agoshkov, Paola Gervasio, and Alfio Quarteroni.
\newblock Optimal control in heterogeneous domain decomposition methods for
  advection-diffusion equations.
\newblock {\em Mediterranean Journal of Mathematics}, 3(2):147--176, 2006.

\bibitem[AP93]{AP93}
Yves Achdou and Olivier Pironneau.
\newblock The {$\chi$}-method for the {N}avier-{S}tokes equations.
\newblock {\em IMA J. Numer. Anal.}, 13(4):537--558, 1993.

\bibitem[ASS99]{ASS99}
Ricardo~L. Actis, Barna~A. Szabo, and Christoph Schwab.
\newblock Hierarchic models for laminated plates and shells.
\newblock {\em Comput. Methods Appl. Mech. Engrg.}, 172(1-4):79--107, 1999.

\bibitem[BBRR07]{BBRR07}
R.~Becker, M.~Braack, R.~Rannacher, and T.~Richter.
\newblock Mesh and model adaptivity for flow problems.
\newblock In {\em Reactive flows, diffusion and transport}, pages 47--75.
  Springer, Berlin, 2007.

\bibitem[BCR89]{BCR89}
Franco Brezzi, Claudio Canuto, and Alessandro Russo.
\newblock A self-adaptive formulation for the {E}uler/{N}avier-{S}tokes
  coupling.
\newblock {\em Comput. Methods Appl. Mech. Engrg.}, 73(3):317--330, 1989.

\bibitem[BE03]{BE03}
Malte Braack and Alexandre Ern.
\newblock A posteriori control of modeling errors and discretization errors.
\newblock {\em Multiscale Model. Simul.}, 1(2):221--238 (electronic), 2003.

\bibitem[BE04]{BE04}
Malte Braack and Alexandre Ern.
\newblock Coupling multimodelling with local mesh refinement for the numerical
  computation of laminar flames.
\newblock {\em Combust. Theory Model.}, 8(4):771--788, 2004.

\bibitem[BHR11]{BHR11}
Michel Borrel, Laurence Halpern, and Juliette Ryan.
\newblock Euler -- {N}avier--{S}tokes coupling for aeroacoustics problems.
\newblock In Alexander Kuzmin, editor, {\em Computational Fluid Dynamics 2010:
  Proceedings of the Sixth International Conference on Computational Fluid
  Dynamics, ICCFD6, St Petersburg, Russia, on July 12-16, 2010}, pages
  427--433. Springer Berlin Heidelberg, 2011.

\bibitem[CBvB05]{CBB05}
J.~M. Cnossen, H.~Bijl, and E.~H. van Brummelen.
\newblock Model-error estimation for goal-oriented model adaptation in
  flow-simulations.
\newblock In {\em Finite volumes for complex applications {IV}}, pages
  173--183. ISTE, London, 2005.

\bibitem[CCG{\etalchar{+}}16]{CCGMS13}
Cl{\'e}ment Canc{\`e}s, Fr{\'e}d{\'e}ric Coquel, Edwige Godlewski,
  H{\'e}l{\`e}ne Mathis, and Nicolas Seguin.
\newblock Error analysis of a dynamic model adaptation procedure for nonlinear
  hyperbolic equations.
\newblock {\em Commun. Math. Sci.}, 14(1):1--30, 2016.

\bibitem[CW01]{CW01}
Cristian~A. Coclici and Wolfgang~L. Wendland.
\newblock Analysis of a heterogeneous domain decomposition for compressible
  viscous flow.
\newblock {\em Math. Models Methods Appl. Sci.}, 11(4):565--599, 2001.

\bibitem[DG15]{DG_15}
Andreas Dedner and Jan Giesselmann.
\newblock A posteriori analysis of fully discrete method of lines dg schemes
  for systems of conservation laws.
\newblock Technical report, 2015.

\bibitem[DGQ12]{DGQ14}
Marco Discacciati, Paola Gervasio, and Alfio Quarteroni.
\newblock Heterogeneous mathematical models in fluid dynamics and associated
  solution algorithms.
\newblock In {\em Multiscale and adaptivity: modeling, numerics and
  applications}, volume 2040 of {\em Lecture Notes in Math.}, pages 57--123.
  Springer, Heidelberg, 2012.

\bibitem[DPE12]{DE12}
Daniele~Antonio Di~Pietro and Alexandre Ern.
\newblock {\em Mathematical aspects of discontinuous {G}alerkin methods},
  volume~69 of {\em Math\'ematiques \& Applications (Berlin) [Mathematics \&
  Applications]}.
\newblock Springer, Heidelberg, 2012.

\bibitem[Fis15]{Fis15}
Julian Fischer.
\newblock A posteriori modeling error estimates for the assumption of perfect
  incompressibility in the {N}avier-{S}tokes equation.
\newblock {\em SIAM J. Numer. Anal.}, 53(5):2178--2205, 2015.

\bibitem[FJN12]{FJN12}
Eduard Feireisl, Bum~Ja Jin, and Anton{\'{\i}}n Novotn{\'y}.
\newblock Relative entropies, suitable weak solutions, and weak-strong
  uniqueness for the compressible {N}avier-{S}tokes system.
\newblock {\em J. Math. Fluid Mech.}, 14(4):717--730, 2012.

\bibitem[FNS11]{FNS11}
Eduard Feireisl, Anton{\'{\i}}n Novotn{\'y}, and Yongzhong Sun.
\newblock Suitable weak solutions to the {N}avier-{S}tokes equations of
  compressible viscous fluids.
\newblock {\em Indiana Univ. Math. J.}, 60(2):611--631, 2011.

\bibitem[GHM14]{GeorgoulisHallMakridakis:2014}
Emmanuil~H. Georgoulis, Edward Hall, and Charalambos Makridakis.
\newblock A posteriori error control for discontinuous galerkin methods for
  first order hyperbolic problems.
\newblock {\em Recent Developments in Discontinuous Galerkin Finite Element
  Methods for Partial Differential Equations: 2012 John H Barrett Memorial
  Lectures}, 2014.

\bibitem[GHM16]{GHM16}
Martin~J. Gander, Laurence Halpern, and V{\'e}ronique Martin.
\newblock A new algorithm based on factorization for heterogeneous domain
  decomposition.
\newblock {\em Numerical Algorithms}, pages 1--29, 2016.

\bibitem[GMP15]{GMP_15}
Jan Giesselmann, Charalambos Makridakis, and Tristan Pryer.
\newblock A posteriori analysis of discontinuous galerkin schemes for systems
  of hyperbolic conservation laws.
\newblock {\em accepted for publication in SIAM J. Numer. Anal.}, 2015.

\bibitem[HGA{\etalchar{+}}12]{Hin12}
Florian Hindenlang, Gregor~J. Gassner, Christoph Altmann, Andrea Beck, Marc
  Staudenmaier, and Claus-Dieter Munz.
\newblock Explicit discontinuous {G}alerkin methods for unsteady problems.
\newblock {\em Comput. \& Fluids}, 61:86--93, 2012.

\bibitem[HH02]{HH02}
Ralf Hartmann and Paul Houston.
\newblock Adaptive discontinuous {G}alerkin finite element methods for
  nonlinear hyperbolic conservation laws.
\newblock {\em SIAM J. Sci. Comput.}, 24(3):979--1004 (electronic), 2002.

\bibitem[HW08]{HW08}
Jan~S. Hesthaven and Tim Warburton.
\newblock {\em Nodal discontinuous {G}alerkin methods}, volume~54 of {\em Texts
  in Applied Mathematics}.
\newblock Springer, New York, 2008.
\newblock Algorithms, analysis, and applications.

\bibitem[Kru70]{Kru70}
S.~N. Kru{\v{z}}kov.
\newblock First order quasilinear equations with several independent variables.
\newblock {\em Mat. Sb. (N.S.)}, 81 (123):228--255, 1970.

\bibitem[LT13]{LT13}
Corrado Lattanzio and Athanasios~E. Tzavaras.
\newblock Relative entropy in diffusive relaxation.
\newblock {\em SIAM J. Math. Anal.}, 45(3):1563--1584, 2013.

\bibitem[MCGS15]{MCGS12}
H{\'e}l{\`e}ne Mathis, Cl{\'e}ment Canc{\`e}s, Edwige Godlewski, and Nicolas
  Seguin.
\newblock Dynamic model adaptation for multiscale simulation of hyperbolic
  systems with relaxation.
\newblock {\em J. Sci. Comput.}, 63(3):820--861, 2015.

\bibitem[MN03]{MN03}
Charalambos Makridakis and Ricardo~H. Nochetto.
\newblock Elliptic reconstruction and a posteriori error estimates for
  parabolic problems.
\newblock {\em SIAM J. Numer. Anal.}, 41(4):1585--1594, 2003.

\bibitem[Nic73]{Nic73}
K~Nickel.
\newblock Prandtl's boundary-layer theory from the viewpoint of a
  mathematician.
\newblock {\em Annual Review of Fluid Mechanics}, 5(1):405--428, 1973.

\bibitem[OV00]{OV00}
J.~Tinsley Oden and Kumar~S. Vemaganti.
\newblock Estimation of local modeling error and goal-oriented adaptive
  modeling of heterogeneous materials. {I}. {E}rror estimates and adaptive
  algorithms.
\newblock {\em J. Comput. Phys.}, 164(1):22--47, 2000.

\bibitem[PV14]{PV14}
Simona Perotto and Alessandro Veneziani.
\newblock Coupled model and grid adaptivity in hierarchical reduction of
  elliptic problems.
\newblock {\em J. Sci. Comput.}, 60(3):505--536, 2014.

\bibitem[SO99]{SO99}
E.~Stein and S.~Ohnimus.
\newblock Anisotropic discretization- and model-error estimation in solid
  mechanics by local {N}eumann problems.
\newblock {\em Comput. Methods Appl. Mech. Engrg.}, 176(1-4):363--385, 1999.
\newblock New advances in computational methods (Cachan, 1997).

\bibitem[SRO11]{SRO11}
Erwin Stein, Marcus R\"uter, and Stephan Ohnimus.
\newblock Implicit upper bound error estimates for combined expansive model and
  discretization adaptivity.
\newblock {\em Comput. Methods Appl. Mech. Engrg.}, 200(37-40):2626--2638,
  2011.

\bibitem[USDM06]{USDM06}
Jens Utzmann, Thomas Schwartzkopff, Michael Dumbser, and Claus-Dieter Munz.
\newblock Heterogeneous domain decomposition for numerical aeroacoustics.
\newblock In {\em Multifield problems in solid and fluid mechanics}, volume~28
  of {\em Lect. Notes Appl. Comput. Mech.}, pages 429--459. Springer, Berlin,
  2006.

\bibitem[VO01]{VO01}
Kumar~S. Vemaganti and J.~Tinsley Oden.
\newblock Estimation of local modeling error and goal-oriented adaptive
  modeling of heterogeneous materials. {II}. {A} computational environment for
  adaptive modeling of heterogeneous elastic solids.
\newblock {\em Comput. Methods Appl. Mech. Engrg.}, 190(46-47):6089--6124,
  2001.

\bibitem[vOBPvB15]{OBPB15}
T.~M. van Opstal, P.~T. Bauman, S.~Prudhomme, and E.~H. van Brummelen.
\newblock Goal-oriented model adaptivity for viscous incompressible flows.
\newblock {\em Comput. Mech.}, 55(6):1181--1190, 2015.

\bibitem[Xu00]{Xu00}
Chuanju Xu.
\newblock An efficient method for the {N}avier-{S}tokes/{E}uler coupled
  equations via a collocation approximation.
\newblock {\em SIAM J. Numer. Anal.}, 38(4):1217--1242, 2000.

\end{thebibliography}
%%%%%%%%%%%%%%%%%%%%%%%%%%%%%%%%%%%%%%%%%%%%%%%%%%%%%%%%%%%%%%%%%%%%%%%%
\end{document}